\documentclass[11pt]{amsart}
\usepackage{amsmath,amssymb,a4wide,scalerel}
\usepackage{mathabx}
\usepackage{color}
\usepackage{xcolor,graphicx}
\usepackage{mathrsfs}
\usepackage{tikz,pgfplots}
\usepackage{subcaption}
\usepackage[normalem]{ulem}
\usepackage{hyperref}
\usepackage{dsfont}
\usepackage{comment}
\usepackage{enumitem}
\usepackage{ctable} 
\usepackage[appendix=append,bibliography=common]{apxproof} 

\newtheorem{theorem}{Theorem}

\newtheorem{proposition}[theorem]{Proposition}
\newtheorem{remark}[theorem]{Remark}

\newtheorem{assumption}{Assumption}

\newcommand{\R}{{\mathbb R}}
\newcommand{\IC}{\mathbb{C}}
\newcommand{\N}{{\mathbb N}}
\newcommand{\E}{{\mathbb E}}

\newcommand{\ii}{\ensuremath{\mathrm{i}}}

\newcommand{\IS}{\mathbb{S}}
\newcommand{\tr}{{\operatorname{Tr}\, }}
\newcommand{\Id}{\text I}
\newcommand{\cE}{\mathcal E}
\newcommand{\cM}{\mathcal M}
\newcommand{\HS}{\mathrm{HS}}
\newcommand{\abs}[1]{\lvert #1 \rvert}
\providecommand{\norm}[1]{\left\lVert#1\right\rVert}
\newcommand*\dd{\mathop{}\!\mathrm{d}}
\newcommand{\expp}[1]{\exp\left(#1\right)}

\renewcommand{\Re}{{\operatorname{Re}}}

\newcommand{\realSH}{{Y}}
\newcommand{\mean}{\mathrm{m}}

\newcommand{\varlm}{v_{\ell,m}}

\newcommand{\std}{{v}}
\newcommand{\Deltalb}{\Delta_{\IS^2}}
\newcommand{\nablalb}{\nabla_{\IS^{2}}}
\newcommand{\bPhi}{\mathbf{\Phi}}
\DeclareMathOperator{\Var}{Var}
\newcommand{\cP}{\mathcal{P}}

\DeclareMathOperator{\spn}{span}

\newif\ifapx
\apxfalse

\usepackage{color}

\hypersetup{hidelinks}

\author{David Cohen}
              \address{Department of Mathematical Sciences,
              Chalmers University of Technology and University of Gothenburg, 41296~Gothenburg, Sweden}
              \email{\tt david.cohen@chalmers.se}

\author{Bj\"orn M\"uller}
              \address{Department of Mathematical Sciences,
              Chalmers University of Technology and University of Gothenburg, 41296~Gothenburg, Sweden}
              \email{\tt bjornmul@chalmers.se}

\author{Andrea Papini}
              \address{Department of Mathematical Sciences,
              Chalmers University of Technology and University of Gothenburg, 41296~Gothenburg, Sweden}
              \email{\tt andreapa@chalmers.se}

\begin{document}

\title[Long-time behavior of stochastic evolution equations on the sphere]{Long-time behavior of exact and numerical solutions of stochastic evolution equations on the sphere}

\begin{abstract}
	We investigate the long-time behavior of exact solutions and numerical approximations of linear stochastic evolution equations defined on the sphere.
	We focus on three classical models arising in mathematical physics: the stochastic wave equation, the stochastic Schrödinger equation, and the stochastic Maxwell's equations.
	For these SPDEs, we analyze several widely used time integrators with respect to trace formulas describing the evolution of physically relevant quantities such as energy, mass, and momentum dependent on the forcing term. In particular, we prove that the forward and backward Euler--Maruyama schemes fail to reproduce the correct long-time behavior of the exact solutions. In addition, we prove that the stochastic exponential integrator preserves the correct long-time behavior of the physical quantities of interest.
	Finally, several numerical experiments are provided to illustrate our theoretical findings.
\end{abstract}

\maketitle
{\small\noindent

}

\section{Introduction}\label{sec-intro}

The analysis of the long-time behavior of numerical schemes for stochastic differential equations (SDEs) started with the paper \cite{MR2083326} (to the best of our knowledge). In this work, the authors gave precise results on the long-time behavior of the forward, backward, and partitioned Euler--Maruyama schemes when applied to a linear stochastic oscillator.
This paper was then followed by the works \cite{MR2922170,MR2909913,MR3831957,MR3608313,MR2333534,MR3923632,MR3348201,MR3906838,MR2312503,math14010017,MR3702519}
on the long-time behavior of numerical methods for linear stochastic oscillators. Furthermore, results on the long-time numerical simulation
of (nonlinear) Hamiltonian and Poisson SDEs driven by additive noise are given in \cite{MR2926251,MR3172331,MR4077238,MR4375496,MR3579605}.

This research has naturally been extended to the infinite-dimensional setting. For example,
the following references on the long-time behavior of numerical methods for stochastic partial differential equations (SPDEs) defined on Euclidean domains are of interest to the present paper. The papers \cite{MR3484400,MR4285774,MR3033008,MR2384109,MR2379913,MR3012932} investigate the long-time simulations of stochastic wave equations. The works \cite{akrivisFiniteDifferenceDiscretization1993,antoineComputationalMethodsDynamics2013,MR3771721,besseRelaxationSchemeNonlinear2004,besseEnergypreservingMethodsNonlinear2021,cohenOnestageExponentialIntegrators2012,MR3007549,sanz-sernaConerservativeNonconservativeSchemes1986} are concerned with stochastic Schr\"odinger equations. Finally,
the papers \cite{MR3432362,MR4077824,MR4410964,MR4739347} deals with stochastic Maxwell's equations.

Moving beyond the traditional Euclidean framework, the numerical analysis of SPDEs defined on Riemannian manifolds,
such as the unit sphere, has recently gained traction in the literature, see for example \cite{MR4161976,MR2370085,Elworthy_1982,MR1246744,MR4059183}. We refer to  \cite{MR4714764,Anh2018,cohen2026fullydiscreteapproximationsemilinear,MR4462619,MR3907363,MR4701212,MR4935789,papi2025,Gia2019} for finite-time convergence results
of numerical methods for SPDEs on the sphere. 
In particular, the stochastic wave and Schrödinger equations on the sphere have been considered in \cite{MR4942808,cohen2026fullydiscreteapproximationsemilinear,MR4462619}.
To the best of our knowledge, nothing is known on the long-time behavior
of (numerical) solutions to SPDEs on the sphere.

The aim of the present paper is to fill this gap by investigating the long-time behavior of numerical methods when applied to some classical
linear stochastic evolution equations on the sphere,
namely the stochastic wave equation, the stochastic Schr\"odinger equation, and the stochastic Maxwell's equations.
On the one hand, we prove that certain widely-used numerical methods (forward and backward Euler--Maruyama schemes) fail to capture the correct long-time behavior of the problem, with respect to trace formulas of physical quantities of interest, such as energy. On the other hand, we prove that a class of stochastic exponential integrators has the exact same long-time behavior, with respect to these trace formulas, as exact solutions to the considered equations.

The paper is organized as follows: Section~\ref{sect-setting} gives the setting needed to define
the considered linear SPDEs on the sphere as well as their spatial and temporal discretizations.
Section~\ref{sect-swe} studies the long-time behavior of numerical solutions of the stochastic wave equation on the sphere in terms of
the energy of the system. In Section~\ref{sect-sch}, we investigate the long-time behavior, with respect to energy, mass and momentum,
of the stochastic Schr\"odinger equation on the sphere and its numerical discretizations. Finally, the long-time behavior of the energy of the stochastic Maxwell's equations on the sphere for the transverse electric mode is considered in Section~\ref{sect-maxeq}.
Each section's results are illustrated by numerical simulations. The code used for these simulations is available at \href{https://github.com/muellerbjoern/TraceSphere}{https://github.com/muellerbjoern/TraceSphere}.
The proofs for Sections~\ref{sect-sch}~and~\ref{sect-maxeq} can be found in Appendix~\ref{apx:Schrödinger}~and~\ref{apx:Maxwell}.

\section{Setting and three SPDEs on the sphere}\label{sect-setting}

In this section, we present the following three SPDEs on the sphere
(details on the notation are given below)
\begin{itemize}
\item the stochastic wave equation on the sphere 
$\partial_{tt}u(t)-\Deltalb u(t)=\dot{L}(t)$;
\item the free stochastic Schr\"odinger equation
on the sphere $\ii\partial_tu(t)=\Deltalb u(t)+\dot L(t)$;
\item the time-dependent stochastic Maxwell's equations in vacuum on the sphere
\begin{align*}
	\partial_t \begin{pmatrix}
		E(t) \\ H(t)
	\end{pmatrix} = \begin{pmatrix}
		\varepsilon^{-1}\nablalb \times H(t) \\ -\mu^{-1}\nablalb \times E(t)
	\end{pmatrix} + \begin{pmatrix} \dot{L}_E(t) \\ \dot{L}_H(t)\end{pmatrix},
\end{align*}
\end{itemize}
as well as their numerical discretizations. To do so, we first give
the setting and the definition of the L\'evy noises $L, L_E$ and $L_H$. 
We then 
define an abstract stochastic evolution equation that will encompass these three SPDEs on the sphere. After that, we provide a spectral discretization of this abstract stochastic evolution equation. Finally, we present three time integrators for this abstract problem.

\subsection{L\'evy noise on the sphere}
In this subsection, we recall some notation and the definition of the L\'evy noise on the sphere, primarily from \cite{MR4462619,papi2025,MR3404631}.

Let $(\Omega, \mathcal{F}, (\mathcal{F}_t)_{t\geq 0},\mathbb{P})$ denote a complete filtered probability space. Let $d\geq 3$ and consider the unit sphere $\mathbb S^{d-1}=\{x\in\R^d, \|x\|_{\R^d}=1\}\subset \R^d$, where $\|\cdot\|_{\R^d}$ denotes the Euclidean norm. 
   
For ease of presentation, throughout the paper 
we will work with $d-1=2$, i.\,e. the unit sphere $\IS^2$. 
All computations for the stochastic wave and Schrödinger equations generalize to any arbitrary dimension $d\geq3$.

We thus consider on $\IS^2$ the geodesic metric given by $d(x,y)=\arccos(\langle x,y\rangle_{\R^3})$ for any $x,y\in\IS^2$. 
Furthermore, we identify the Cartesian coordinates $x\in\IS^2$ with the polar coordinates $(\vartheta,\varphi)\in[0,\pi]\times[0,2\pi)$ by the classical transformation $x=(\sin(\vartheta)\cos(\varphi),\sin(\vartheta)\sin(\varphi),\cos(\vartheta))$.

The Lebesgue measure on the sphere has the representation
\[
\dd x = \sin(\vartheta)\dd \vartheta \dd \varphi.\]
We denote by $\mathcal{B}(\IS^2)$ the associated Borel $\sigma$-algebra. Then, the space $L^2(\IS^2)=L^2(\IS^2,\R)$ is a Hilbert space with the inner product $\langle \cdot,\cdot \rangle_{L^2(\IS^2)}$ defined by
\[
\langle f,g \rangle=\langle f,g \rangle_{L^2(\IS^2)}=\int_{\IS^2} f(x)g(x) \dd x
\]
for any $f,g\in L^2(\IS^2)$. The induced norm is denoted by
$\norm{\cdot}=\norm{\cdot}_{L^2(\IS^2)}=\sqrt{\langle \cdot,\cdot \rangle_{L^2(\IS^2)}}$.

It is known that the (real-valued) spherical harmonics form an orthonormal basis of the Hilbert space $L^2(\IS^2)$, see \cite{marinucciRandomFieldsSphere2011a}.
The spherical harmonics are denoted $\mathcal{Y}= (Y_{\ell,m})_{\ell\in\N_0, m=-\ell,\ldots,\ell}$, where each function $Y_{\ell,m}\colon[0,\pi]\times[0,2\pi) \to \R$ is given by
\begin{equation}
    Y_{\ell,m}(\vartheta,\varphi) = 
    \begin{cases}
        \sqrt{2}(-1)^m \sqrt{\frac{2\ell+1}{4\pi}\frac{(\ell-|m|)!}{(\ell+|m|)!}} P_{\ell,|m|}(\cos(\vartheta))\sin(|m|\varphi), & \text{for } m<0 \\[3mm]
     \sqrt{\frac{2\ell+1}{4\pi}} P_{\ell,m}(\cos(\vartheta)), 
     & \text{for } m=0 \\[3mm]
         \sqrt{2}(-1)^m \sqrt{\frac{2\ell+1}{4\pi}\frac{(\ell-m)!}{(\ell+m)!}} P_{\ell,m}(\cos(\vartheta))\cos(m\varphi), & \text{for }  m>0.
    \end{cases}
\end{equation}
Above, $(P_{\ell,m})_{\ell\in\N_0, m=0,\ldots,\ell}$ are the Legendre polynomials, see for instance \cite{szeg1939orthogonal}.

Furthermore, the spherical harmonics are eigenfunctions of  the Laplace--Beltrami operator, see, e.\,g., \cite[Chapter 6]{byerly1893elementary}
defined by 
\begin{equation*}
    \Deltalb = (\sin(\vartheta))^{-1} \frac{\partial}{\partial \vartheta}\Big(\sin(\vartheta)\frac{\partial}{\partial \vartheta} \Big) +  (\sin(\vartheta))^{-2} \frac{\partial^2}{\partial \varphi^2}.    
\end{equation*}
That is, for all $\ell\in\N_0$, $m=-\ell,\ldots,\ell$, we have
\begin{equation*}
    \Deltalb Y_{\ell,m}
    = -\lambda_\ell Y_{\ell, m},
\end{equation*}
with eigenvalues $\lambda_\ell = \ell(\ell+1)$.
Note that by a linear basis transformation, the results for the complex-valued spherical harmonics given in \cite{marinucciRandomFieldsSphere2011a} are applicable for the real-valued case as well, see, e.\,g., \cite[Appendix A]{janssonNonstationaryGaussianRandom2024}.

The functional spaces, in which the solution to the SPDEs under investigation and the driving noise $L$ live,
are the Sobolev spaces $H^\eta(\IS^2)$ with a regularity index $\eta\in\R$ that we now present. For $\eta=0$, we set $H^0(\mathbb{S}^2)=L^2(\mathbb{S}^2)$. When the index $\eta\in\R_+$ is positive, we use the Bessel potentials to define these spaces as
\begin{equation*}
    H^\eta(\IS^2) = (\Id-\Deltalb)^{-\eta/2}L^2(\IS^2).
\end{equation*}
with the induced inner product
\begin{equation*}
    \langle f, g \rangle_{H^\eta(\IS^2)} = \langle (\Id-\Deltalb)^{\eta/2}f,(\Id-\Deltalb)^{\eta/2} g \rangle.
\end{equation*} 
For $\eta \in \R_+$, using the series expansion of functions in terms of spherical harmonic functions with coefficients in $L^2(\Omega, \R)$, we interpret the negative Sobolev space $H^{-\eta}(\IS^2)$ within the framework of a Gelfand triple. That is, we consider, for $\mathcal{H}=L^2(\IS^2)$ and $\eta\geq0$, the following embeddings
$$
V=H^\eta \hookrightarrow \mathcal{H}\simeq \mathcal{H}^*\hookrightarrow V^* = H^{-\eta},
$$
which give a proper definition of $H^{-\eta}(\IS^2)$. For further details on the introduced Sobolev spaces and the Bessel potentials, we refer to \cite{AnMan}.

For the coupling between space regularity and probability, we use the Lebesgue--Bochner spaces $L^p(\Omega;H^\eta(\IS^2))$, for $p\geq1, \eta\in \R$, equipped with the norm
$$
\|X\|_{L^p(\Omega;H^\eta(\IS^2))}=\mathbb{E}[\|X\|^p_{H^\eta(\IS^2)}]^{1/p},
$$
similarly to \cite{papi2025,MR3404631}.

Let us now introduce the driving L\'evy process $L=(L(t))_{t\geq 0}$ of the considered SPDEs. Following~\cite[Definition 4.1]{peszat2007}, within the same framework as in \cite{papi2025}, we assume that $L$ belongs to the class of square-integrable L\'evy  processes in $H^\eta(\IS^2)$, for $\eta\geq 0$, i.\,e., $L(t) \in L^2(\Omega,H^\eta(\IS^2))$. 
A L\'evy process $L$ is a stochastic process with stationary and independent increments, which is stochastically continuous and satisfies $L(0)=0$ almost surely. The process $L$ has a càdlàg modification, see for example \cite[Theorem 4.3]{peszat2007}.
Its mean $\E\left[L(t)\right] = t \mean \in H^\eta(\IS^2)$ and covariance operator $Q \in L^+_1(H^\eta(\IS^2))$ exist, see \cite[Theorem 4.4]{peszat2007} for details. Here,  $L^+_1(H^\eta(\IS^2))$ denotes the space of all linear, trace-class, symmetric, positive semi-definite operators from $H^\eta(\IS^2)$ into itself. For such operators, we define the Hilbert--Schmidt norm as
$$
\| Q \|_{HS(H^\eta(\IS^2))}=\left( \sum_{k=0}^\infty 
\| Qe_k\|^2_{H^\eta(\IS^2)} \right)^{1/2},
$$
where $\{e_k\}_{k=0}^\infty$ is an orthonormal basis of $H^\eta(\IS^2)$. 
For $\eta=0$, we use the notation $\norm{\cdot}_{HS}$.

In the present setting, we can give series expansions for the mean and covariance of the stochastic process $L$; we now provide some details on these quantities. Since $H^\eta(\IS^2) \subset L^2(\IS^2)$ for $\eta \ge 0$, we can write
the L\'evy process as the following expansion in $L^2({\IS^{2}})$,
 \begin{equation}\label{eq:KL-expansion_Levy}
	L(t) = \sum_{\ell=0}^\infty \sum_{m=-\ell}^\ell  L_{\ell,m}(t)Y_{\ell,m}
\end{equation}
with real-valued c\`{a}dl\`{a}g one--dimensional L\'evy processes $L_{\ell,m} = \langle L, Y_{\ell,m}\rangle$.
This series converges in~$L^2(\Omega,H^\eta(\IS^2))$.
The mean of $L$ is given by
\begin{equation*}
	t \, \mean = \E\left[L(t)\right]
		= \sum_{\ell=0}^\infty \sum_{m=-\ell}^\ell \E\left[L_{\ell,m}(t)\right] \realSH_{\ell,m}
		= t \sum_{\ell=0}^\infty \sum_{m=-\ell}^\ell \mean_{\ell,m} \realSH_{\ell,m}.
\end{equation*}

Applying the definition of the trace of an operator, $\tr Q = \sum_{\ell=0}^{\infty}\sum_{m=-\ell}^{\ell} \langle Q Y_{\ell , m}, Y_{\ell , m} \rangle$, and using the orthonormality of the spherical harmonics, we have,
for $t\geq 0$,
\begin{equation*}
	\begin{aligned}
		t \, \tr Q = &t \sum_{\ell=0}^{\infty}\sum_{m=-\ell}^{\ell} \langle Q Y_{\ell , m}, Y_{\ell , m}\rangle = t \sum_{\ell=0}^{\infty}\sum_{m=-\ell}^{\ell} \E\left[\langle L(1) - \mean, Y_{\ell , m}\rangle \langle L(1)-\mean, Y_{\ell , m}\rangle \right]\\
        = & t \sum_{\ell=0}^\infty \sum_{m=-\ell}^\ell \E\left[(L_{\ell,m}(1) - \mean_{\ell, m})^2\right]
		= \sum_{\ell=0}^\infty \sum_{m=-\ell}^\ell  \E\left[(L_{\ell,m}(t) - \E\left[L_{\ell,m}(t)\right])^2\right]\\
        = & \E\left[\norm{L(t) - \E\left[L(t)\right]}_{L^2(\IS^2)}^2\right],
	\end{aligned}
\end{equation*}
where in the second-to-last equality we have used the homogeneity of the L\'evy processes $L_{\ell,m}$, see \cite[Section~4.8]{peszat2007}.

Furthermore, we denote the standard deviation as
\begin{equation}\label{Lvlm}
    \std = \sum_{\ell=0}^\infty \sum_{m=-\ell}^\ell \sqrt{\varlm} \realSH_{\ell,m},\ \text{with }\std_{\ell,m}=\E\left[(L_{\ell,m}(1) - \E\left[L_{\ell,m}(1)\right])^2\right].
\end{equation}
We note that $\std$ is a function in~$H^\eta(\IS^2)$.

Following \cite{papi2025}, throughout the paper we will assume the following regularity for the driving noise of the considered SPDEs on the sphere.

\begin{assumption}\label{ass:StrongLevyProcess}
	The L\'evy process~$L$ is in $L^2(\Omega, H^\eta(\IS^2))$ for some $\eta \geq 0$ with series expansion~\eqref{eq:KL-expansion_Levy} converging in $L^2(\Omega, H^\eta(\IS^2))$, i.\,e.,
	\begin{equation*}
		\E\left[\|L(t)\|_{H^\eta(\IS^2)}^2\right]
		= t \sum_{\ell=0}^\infty \sum_{m=-\ell}^\ell      (1+\lambda_\ell)^\eta (\varlm + \mean_{\ell, m}^2)
		= t (\|\std\|_{H^\eta(\IS^2)}^2 + \|\mean\|_{H^\eta(\IS^2)}^2)
		<\infty.
	\end{equation*}
\end{assumption}

In order to perform the numerical experiments presented below, the following assumption, which is a specialization of Assumption~\ref{ass:StrongLevyProcess}, will be used.
\begin{assumption}\label{ass:SpecialLevyProcess}
	The L\'evy process~$L$ given by the expansion~\eqref{eq:KL-expansion_Levy} satisfies
	\begin{equation*}
		L_{\ell,m} = a_{\ell} \hat{L}_{\ell,m},
	\end{equation*}
	where $a_\ell \leq C \ell^{-\alpha/2}$ for all $\ell\in \N$, for some $\alpha > 0$, $a_0$ bounded, and $(\hat{L}_{\ell,m})_{\ell\in\N_0, m=-\ell,\ldots,\ell}$ is a sequence of identically distributed real-valued Lévy processes with mean $\E\left[\hat{L}_{\ell,m}(t)\right] = \hat{\mean}$ and variance $\E\left[\left(\hat{L}_{\ell,m}(t) - \E\left[\hat{L}_{\ell,m}(t)\right]\right)^2\right] = t$.
\end{assumption}

\begin{remark}
Note that direct computation show that, under Assumption~\ref{ass:SpecialLevyProcess}, $L \in L^2(\Omega, H^\eta({\IS^{2}}))$ for all $0 \leq \eta < \alpha/2 - 1$.
Observe further that, as in \cite{papi2025}, to investigate most of the properties of solutions to SPDEs driven by this type of Lévy noise, it is not necessary
for the components in the expansion~\eqref{eq:KL-expansion_Levy} to be independent or uncorrelated.
Note that the case of isotropic random fields is considered as a special case under Assumption~\ref{ass:SpecialLevyProcess}, see \cite{MR4462619,MR4701212,MR3404631}.
\end{remark}

Let us observe that solutions to the free stochastic Schr\"odinger equation in Section~\ref{sect-sch} are complex-valued. All the above can be extended to the Hilbert space with underlying field $\mathbb{C}$, changing the inner product with the Hermitian one. The covariance operator $Q$ will then be complex-valued with real eigenvalues since it is self-adjoint and Hermitian.

\subsection{Spaces of solutions}\label{solution_spaces}
In our framework, the solutions to the SPDEs we investigate in this paper have vector-valued solutions.
Hence, for reason of exposition, we define the spaces $L^2(\IS^2,\R^n)$, for $n\geq1$, using the tensor product relation, i.\,e.,
$$
L^2(\IS^2,\R^n)\simeq L^2(\IS^2)\otimes \R^n,
$$
where an orthonormal basis of this space can be defined using the usual rule for the tensor product: $\mathbb{Y}=\{Y_{\ell,m}\otimes e_i\}_{\ell=0,m=-\ell,i=1}^{\infty,\ell,n}$.
Here, we interpret
$Y_{\ell,m}\otimes e_i$ as $(0,\dots,\underbrace{Y_{\ell,m}}_i,\dots,0)$.

The solution to the stochastic Maxwell's equations on the sphere (see the next subsection and Section~\ref{sect-maxeq}) takes values in $L^2(\IS^2,\R^3)$, for which it is more convenient to work with the so-called vector spherical harmonics as an orthonormal basis:
for $\ell\in\mathbb N_0$ and  $m=-\ell,\ldots,\ell$, consider
\begin{align}\label{3d-onb}
    \begin{split}
	\mathbf{Y}_{\ell, m}&=Y_{\ell, m}\hat{\textbf{r}}, \\
	\mathbf{\Phi}_{\ell, m}&=\frac{1}{\sqrt{\lambda_\ell}}\nablalb {Y}_{\ell, m}, \\
	\mathbf{\Psi}_{\ell, m}&=\frac{1}{\sqrt{\lambda_\ell}}\hat{\textbf{r}}\times\nablalb Y_{\ell,m}.
    \end{split}
\end{align}
Here, $\hat{ \textbf{r}}$ is a unit vector in the radial direction, see \cite{Barrera_1985} for details.

\subsection{An abstract stochastic evolution equation}
Now we
present an abstract stochastic evolution equation that encompasses the three SPDEs that we will consider in more details in the following sections. 
We consider the following linear abstract stochastic evolution equation on the  sphere $\IS^{2}$:
\begin{align}\label{prob}
\begin{split}
   \dd X(t)&=AX(t)\dd t+B\dd L(t),\\
    X(0)&=X_0,
\end{split}
\end{align}
for $t \geq 0$.  
Here, we assume an $\mathcal{F}_0$-measurable initial condition 
$X_0$ (further assumptions on $X_0$ are given below for each of the three models). The evolution equation~\eqref{prob} is driven by an infinite-dimensional L\'evy noise $L$ satisfying Assumption~\ref{ass:StrongLevyProcess}.

The linear operator
$B\colon L^2(\IS^2,\mathbb R^k) \to L^2(\IS^2,\mathbb R^n)$, for some $k \geq 1$, will be specified below for each considered SPDE. The linear operator $-A$, while it changes for the three considered SPDEs, always generates a strongly continuous semigroup $(S(t))_{t\geq 0}$ of bounded linear operators on $L^2(\IS^{2},\R^n)$, where $n$ depends on the considered model. Moreover, throughout the paper, the operator $-A$ will be diagonalizable with a complete orthonormal basis of the Hilbert space $L^2(\IS^{2},\R^n)$ with eigenvalues $\{\lambda_j\}_{j\in J}$, where $J$ is some index set.

Under these assumptions, we consider 
the notion of mild solution of the stochastic evolution equation~\eqref{prob} as in \cite{peszat2007}: 
\begin{align}\label{mild_s}
X(t)=S(t)X(0)+\int_0^tS({t-s})B\dd L(s),\ t\geq0,
\end{align}
where the stochastic convolution is well defined as shown in \cite{peszat2007}. The existence and uniqueness for the mild solution~\eqref{mild_s} follow from simple modifications of \cite{MR4077824,MR4462619,papi2025}.

Since the spherical harmonics $\mathcal Y$ form an orthonormal basis of the Hilbert space $L^2(\IS^2,\R)$, the mild solution~\eqref{mild_s} has the following series expansion: For $j=1,\ldots,n$, one has 
\begin{align}\label{mildsol}
X_j(t)&=\sum_{\ell=0}^\infty\sum_{m=-\ell}^\ell\langle X_j(t),Y_{\ell,m} \rangle Y_{\ell,m},
\end{align}
where $X_j\in L^2({\IS^{2}}, \R)$ denotes the $j$-th component of $X$ as a vector-valued function.

\subsection{Three types of SPDEs on the sphere}
In this paper, we consider three concrete types of SPDEs on the sphere of the form~\eqref{prob}, that we now define. 

The first type of SPDE we consider is the linear stochastic wave equation on the sphere $\IS^2$ given as
\begin{equation}\label{swe}
\partial_{tt}u(t)-\Deltalb u(t)=\dot{L}(t)
\end{equation}
with initial conditions $u(0)=v_1\in L^2(\IS^2)$ and $\partial_{t}u(0)=v_2\in L^2(\IS^2)$.
The notation $\dot{L}$ stands for the formal time derivative of the $Q$-L\'evy process with series expansion~\eqref{eq:KL-expansion_Levy}. The stochastic wave equation on the sphere~\eqref{swe} can be written in the abstract
form~\eqref{prob} with
\begin{equation*}
 A=\begin{pmatrix}0 & I \\ \Deltalb & 0 \end{pmatrix},
 \quad B=\begin{pmatrix}0\\I \end{pmatrix},
 \quad X=\begin{pmatrix} u\\ \partial_tu\end{pmatrix}=\begin{pmatrix} u_1\\u_2\end{pmatrix},
 \quad X_0=\begin{pmatrix} v_1\\v_2\end{pmatrix}.
\end{equation*}
Here, and below, the operator $I$ is the identity operator in $L^2(\IS)$.

The second type of SPDE we consider is the free stochastic Schr\"odinger equation
on the sphere $\IS^2$,
\begin{equation}\label{eqn:SSE}
\ii\partial_tu(t)=\Deltalb u(t)+\dot L(t),
\end{equation}
with (possibly complex-valued) initial condition $u(0)\in L^2(\IS^2)$.
Here, the unknown $u(t)=u_R(t)+\ii u_I(t)$ is a complex-valued stochastic process.
Considering the real and imaginary parts of the SPDE~\eqref{eqn:SSE}, one can rewrite it in the abstract setting~\eqref{prob}
using
\begin{equation*}
 A=\begin{pmatrix}0 & \Deltalb \\ -\Deltalb & 0 \end{pmatrix},
 \quad B=\begin{pmatrix}0\\-I \end{pmatrix},
 \quad X=\begin{pmatrix} u_R\\u_I\end{pmatrix}, \quad X_0=\begin{pmatrix} u_R(0)\\u_I(0)\end{pmatrix}.
\end{equation*}

The third type of problem we consider is the stochastic Maxwell's equations in vacuum on the sphere,
\begin{align}\label{maxeq}
\partial_t \begin{pmatrix}
        E(t) \\ H(t)
    \end{pmatrix} = \begin{pmatrix}
        \varepsilon^{-1}\nablalb \times H(t) \\ -\mu^{-1}\nablalb \times E(t)
    \end{pmatrix} + \begin{pmatrix} \dot{L}_E(t) \\ \dot{L}_H(t) \end{pmatrix},
\end{align}
with initial conditions $E(0)$ and $H(0)$ in $L^2(\IS^2,\R^3)$. Here,
$E$ and $H$, are the electric and magnetic fields, respectively. Furthermore, it is known that $\text{div}(E)=\text{div}(H)=0$.
In equation~\eqref{maxeq}, $\varepsilon,\ \mu\in L^\infty_{+}(\IS^2)$, are, respectively, the permittivity and permeability,
and we set $\mu=\varepsilon=1$ for ease of presentation. Finally, observe that the SPDE~\eqref{maxeq} can be written in the abstract setting~\eqref{prob}
with
\begin{equation*}
 A=\begin{pmatrix} 0 & \nablalb\times \\ -\nablalb\times & 0 \end{pmatrix},
 \quad B=\begin{pmatrix} I\\ I \end{pmatrix},
 \quad X=\begin{pmatrix} E\\ H\end{pmatrix}, \quad X_0=\begin{pmatrix} E(0)\\ H(0)\end{pmatrix}.
\end{equation*}
We define the precise meaning of the surface curl operator $\nabla_{\IS^{2}}\times$ in Section~\ref{sect-maxeq}.

\subsection{Spectral discretization}
\label{subsec:spectral}
In this subsection, we discretize the abstract SPDE~\eqref{prob} in space using a spectral method.
We borrow some notation from \cite{MR4462619,papi2025}.

The exact solution to the stochastic evolution equation~\eqref{prob} is numerically approximated 
by truncating the series
expansion~\eqref{mildsol} at a given integer index $\kappa>0$. We then obtain the truncated solution: For $j=1,\ldots,n$, one has  
\begin{equation}\label{spectralsol}
X_j^\kappa(t)=\sum_{\ell=0}^\kappa\sum_{m=-\ell}^\ell\langle X_j^\kappa(t),Y_{\ell,m} \rangle Y_{\ell,m}. 
\end{equation}
Observe that the spectral approximation~\eqref{spectralsol} is the solution to the abstract stochastic evolution equation
\begin{align}
    \begin{split}
    \label{spectralprob}
\dd X^\kappa(t)&=A^\kappa X^\kappa(t)\dd t+\cP_\kappa B\dd L(t)\\
X^\kappa(0)&=\cP_\kappa X_0,
    \end{split}
\end{align}
where $A^\kappa\colon H^\kappa \to H^\kappa$ is the spectral projection of $A$
onto $$\mathcal D(A^\kappa) = H^\kappa= \spn\{Y_{\ell , m}, \ell = 0 , \dots, \kappa, m = -\ell, \dots, \ell\}.$$ That is, one has
$$
\left<A^\kappa\phi^\kappa,\psi^\kappa\right>=
\left<A\phi^\kappa,\psi^\kappa\right>,\ \forall\phi^\kappa,\psi^\kappa\in H^\kappa.
$$
The truncation operator $\cP_\kappa$ in the equation~\eqref{spectralprob} is defined componentwise as
$(\displaystyle\cP_\kappa Z)_j=\sum_{\ell=0}^\kappa\sum_{m=-\ell}^\ell Z^j_{\ell,m} Y_{\ell,m}$, for $j=1,\dots,n$ and $Z\in L^2(\IS^2,\R^n)$ with components given by the series expansion $\displaystyle Z_j=\sum_{\ell=0}^\infty\sum_{m=-\ell}^\ell Z^j_{\ell,m} Y_{\ell,m}$.
We can observe a slight abuse of notation when the truncation operator is applied to vector-valued functions. This is done for ease of presentation and computations.
The covariance operator of the truncated noise $L^\kappa = \cP_\kappa L$ is denoted by
$Q^\kappa = \cP_\kappa Q \cP_\kappa.$

\subsection{Temporal discretizations}
In this subsection, we further discretize  the semi-discrete problem~\eqref{spectralprob}
with the forward Euler--Maruyama scheme, the backward Euler--Maruyama scheme, and
the exponential Euler scheme, see for instance the literature on numerics for S(P)DEs
\cite{MR1214374, Lord_Powell_Shardlow_2014, MR4369963}.

For a given time step $\tau > 0$, we define the time grid $t_n=n\tau$ and the L\'evy increments
by $\Delta L^{\ell, m}_n=L_{\ell,m}(t_{n+1})-L_{\ell,m}(t_{n})$ for $m=-\ell,\ldots,\ell$,
and $\ell=0,\ldots,\kappa$, and $n\geq 0$.

When applied to the SPDE~\eqref{spectralprob}, the forward Euler--Maruyama scheme reads
\begin{equation}\label{eEM}
X^{\text{EM}}_{n+1}=X^{\text{EM}}_{n}+\tau A^\kappa X^{\text{EM}}_{n}+\cP_\kappa\left(B\Delta L_n\right).
\end{equation}
where $\Delta L_n= \displaystyle \sum_{\ell=0}^{\infty}\sum_{m = - \ell}^{\ell} \Delta L^{\ell, m}_n Y_{\ell , m}$.

When applied to the SPDE~\eqref{spectralprob}, the backward Euler--Maruyama scheme reads
\begin{equation}\label{bEM}
X^{\text{bEM}}_{n+1}=X^{\text{bEM}}_{n}+\tau A^\kappa X^{\text{bEM}}_{n+1}+\cP_\kappa\left(B\Delta L_n\right).
\end{equation}

When applied to the SPDE~\eqref{spectralprob}, the exponential Euler scheme reads
\begin{equation}\label{expEM}
X^{\text{expEM}}_{n+1}=e^{A^\kappa\tau}\left(X^{\text{expEM}}_{n}+\cP_\kappa\left(B\Delta L_n\right)\right).
\end{equation}

In the next sections, we will study the ability of these numerical schemes to reproduce the long-time behavior of
the three SPDEs on the sphere with respect to physical quantities such as the energy, mass, and momentum.

\section{Stochastic wave equation on the sphere}\label{sect-swe}
The conservation of the total energy of the deterministic linear wave equation on Euclidean domains is well known,
see for instance \cite[Section 2.4.3]{MR2597943}. When randomly perturbing the wave equation
(defined in a Euclidean domain) with an additive $Q$-Wiener process, the total energy is no longer a conserved quantity.
Indeed, one can show a trace formula for the total energy, that is, the expected value of the total energy of the exact solution grows linearly with time,
see for instance \cite{MR3033008,MR3012932}. In this section, we investigate the behavior of the total energy of the exact solution to the stochastic wave equation on the sphere~\eqref{swe}. Furthermore, we prove long-time results on the behavior of the energy along its numerical solutions.
The long-time behavior of the numerical solutions will then be illustrated numerically at the end of this section.
For ease of presentation, we consider the case of L\'evy processes with mean zero.
We conclude this section with a remark on the case when the mean of the L\'evy process is not zero.

For $t>0$, the total energy of the stochastic wave equation~\eqref{swe} at time $t$ is defined as
\begin{equation}\label{swe-energy}
\mathcal{E}(t)=\frac12\left(\norm{\partial_t u(t)}^2+\norm{(-\Deltalb)^{1/2}u(t)}^2\right).
\end{equation}

Let us first investigate the long-time behavior of the exact solution to the stochastic wave equation~\eqref{swe}, or its abstract setting~\eqref{prob}. Following, e.\,g., \cite{MR4462619}, we can explicitly write the semigroup $S(t)$ associated to the operator $A$ as follows
\begin{align}\label{sg_wave}
S(t)&=\begin{pmatrix}
\cos(t(-\Delta_{\IS^2})^{\frac12}) & 
(-\Delta_{\IS^2})^{-\frac12}\sin(t(-\Delta_{\IS^2})^{\frac12})\\
-(-\Delta_{\IS^2})^{\frac12}\sin(t(-\Delta_{\IS^2})^{\frac12}) & 
\cos(t(-\Delta_{\IS^2})^{\frac12})
\end{pmatrix}.
\end{align}

We can then rewrite the mild solution to the stochastic wave equation~\eqref{swe} as
\begin{align}
	\label{eqn:SWE_soln}
	\begin{split}
u_1(t)&=\cos(t(-\Deltalb)^{1/2})v_1+(-\Deltalb)^{-1/2}\sin(t(-\Deltalb)^{1/2})v_2\\
&\qquad+\int_0^t(-\Deltalb)^{-1/2}\sin((t-s)(-\Deltalb)^{1/2})\dd L(s),\\
u_2(t)&=-(-\Deltalb)^{1/2}\sin(t(-\Deltalb)^{1/2})v_1+\cos(t(-\Deltalb)^{1/2})v_2\\
&\qquad+\int_0^t\cos((t-s)(-\Deltalb)^{1/2})\dd L(s).
	\end{split}
\end{align}

We now show that the expected value of the energy of the stochastic wave equation on the sphere~\eqref{swe} has a linear growth in time.

\begin{proposition}\label{prop:energy-exact}
Consider the stochastic wave equation on the sphere~\eqref{swe} with initial values $(u(0),\partial_tu(0))=(v_1,v_2)\in [L^2(\IS^2,\R)]^2$.
Assume that the expected value of the initial energy $\E\left[\mathcal{E}(0)\right]<\infty$ and that the $Q$-L\'evy process $L$ is trace-class.
The solution to this SPDE satisfies the trace formula for the energy:
\begin{align*}
  \E\left[\mathcal{E}(t)\right]=\E\left[\mathcal{E}(0)\right]+\frac{t}2\tr(Q),
\end{align*}
for every time $t\geq0$.
\end{proposition}
\begin{proof}
Let us recall that we are assuming that $\mathbb E\left[L(t)\right]=0$ and that we use the notation $u_1=u$ and $u_2=\partial_tu$. We use the expression~\eqref{eqn:SWE_soln} for the mild solution to the considered SPDE and compute the expectation of the energy as follows
\begin{align*}
2\E\left[\cE(t)\right] &= \mathbb{E}\left[\norm{u_{2}(t)}^2+\norm{(-\Deltalb)^{1/2}u_{1}(t)}^2\right]\\
        &= \E\left[ \norm{(-\Deltalb)^{1/2}\sin(t(-\Deltalb)^{1/2})v_1}^2 + \norm{\cos(t(-\Deltalb)^{1/2})v_2}^2\right.\\
        & \left.\quad -2 \langle (-\Deltalb)^{1/2}\sin(t(-\Deltalb)^{1/2})v_1, \cos(t(-\Deltalb)^{1/2})v_2\rangle\right.\\
        & \left.\quad + \norm{ \int_0^t\cos((t-s)(-\Deltalb)^{1/2})\dd L(s)}^2\right.\\
        & \left.\quad + \norm{ (-\Deltalb)^{1/2}\cos(t(-\Deltalb)^{1/2})v_1}^2 + \norm{\sin(t(-\Deltalb)^{1/2})v_2}^2  \right.\\
        &\left. \quad + 2 \langle (-\Deltalb)^{1/2} \cos(t(-\Deltalb)^{1/2})v_1, \sin(t(-\Deltalb)^{1/2})v_2\rangle  \right.\\
        &\left. \quad + \norm{\int_0^t\sin((t-s)(-\Deltalb)^{1/2})\dd L(s)}^2
        \right]\\
        & = 2\E\left[\cE(0)\right] + \E\left[\norm{ \int_0^t\cos((t-s)(-\Deltalb)^{1/2})\dd L(s)}^2 +\norm{\int_0^t\sin((t-s)(-\Deltalb)^{1/2})\dd L(s)}^2\right].
\end{align*}
In the last step, we have used the definition of the initial energy $\cE(0)$. Using It\^o's isometry, we then get the relation
\begin{align*}
2\E\left[\cE(t)\right] &= 2\E\left[\cE(0)\right] +\int_0^t\left(\norm{\sin((t-s)(-\Deltalb)^{1/2})Q^{1/2}}^2_{HS}+\norm{\cos((t-s)(-\Deltalb)^{1/2})Q^{1/2}}_{HS}^2\right)\dd s\\
&=2\E\left[\cE(0)\right] +\int_0^t\norm{Q^{1/2}}^2_{HS}\dd s=2\E\left[\cE(0)\right] +t\ \tr(Q).
\end{align*}
Here, we have used the properties of the
Hilbert--Schmidt norm and the property of the 
cosine and sine operators: $\cos^2(t(-\Deltalb)^{1/2})+\sin^2(t(-\Deltalb)^{1/2})=I$.

This concludes the proof.
\end{proof}

\subsection{Spectral discretization}
Next, we investigate the long-time behavior of the numerical solution given by the spectral discretization~\eqref{spectralsol}. We recall that this numerical solution satisfies the stochastic evolution equation~\eqref{spectralprob} and, 
in the case of the stochastic wave equation on the sphere~\eqref{swe}, this numerical solution reads
\begin{align}
\label{eqn:SWE_spectral}
\begin{split}
u_1^\kappa(t)&=\cos(t(-\Deltalb^\kappa)^{1/2})\cP_\kappa v_1+(-\Deltalb^\kappa)^{-1/2}\sin(t(-\Deltalb^\kappa)^{1/2})\cP_\kappa v_2\\
&\qquad+\int_0^t(-\Deltalb^\kappa)^{-1/2}\sin((t-s)(-\Deltalb^\kappa)^{1/2})\cP_\kappa\dd L(s),\\
u_2^\kappa(t)&=-(-\Deltalb^\kappa)^{1/2}\sin(t(-\Deltalb^\kappa)^{1/2})\cP_\kappa v_1+\cos(t(-\Deltalb^\kappa)^{1/2})\cP_\kappa v_2\\
&\qquad+\int_0^t\cos((t-s)(-\Deltalb^\kappa)^{1/2})\cP_\kappa\dd L(s).
\end{split}
\end{align}

We then define the total energy of the spectral discretization by
\begin{equation}\label{swe-spectral-energy}
\mathcal{E}^\kappa(t)=\frac{1}{2}\left(\norm{u_2^\kappa(t)}^2+\norm{(-\Delta_{\IS^2}^\kappa)^{1/2}u_1^\kappa(t)}^2\right).
\end{equation}
We show that the numerical solution~\eqref{eqn:SWE_spectral} has the same long-time behavior, with respect to the energy, as the exact solution~\eqref{eqn:SWE_soln} to the stochastic wave equation on the sphere~\eqref{swe}.
\begin{proposition}\label{prop:energy-spectral}Consider the stochastic wave equation on the sphere~\eqref{swe} with initial values $(u(0),\partial_tu(0))=(v_1,v_2)\in [L^2(\IS^2,\R)]^2$.
Assume that the expected value of the initial energy $\E\left[\mathcal{E}(0)\right]<\infty$ and that the $Q$-L\'evy process $L$ is trace-class.
Let ($u_1^\kappa=u^\kappa$,$u_2^\kappa=\partial_tu^\kappa$) denote the spectral approximation given by equation~\eqref{eqn:SWE_spectral}.
The spectral approximation satisfies the trace formula for the energy:
\begin{align*}
  \E\left[\mathcal{E}^\kappa(t)\right]=\E\left[\mathcal{E}^\kappa(0)\right]+\frac{t}2\tr(Q^\kappa),
\end{align*}
for every time $t\geq0$.
\end{proposition}
\begin{proof}
The proof of this result is very similar to the one for the exact solution of the stochastic wave equation on the sphere~\eqref{swe} in Proposition~\ref{prop:energy-exact}. The only difference being in the treatment of the noise.

\newcommand{\stkout}[1]{\ifmmode\text{\sout{\ensuremath{#1}}}\else\sout{#1}\fi}
After inserting the spectral approximation~\eqref{eqn:SWE_spectral} into the energy~\eqref{swe-spectral-energy}, using trigonometric identities and It\^o's isometry, one obtains
\begin{align*}
&2\E\left[\cE^\kappa(t)\right]\\
&= 2\E\left[\cE^\kappa(0)\right] +\int_0^t\left(\norm{\sin((t-s)(-\Deltalb^\kappa)^{1/2})\cP_\kappa Q^{1/2}}^2_{HS}+\norm{\cos((t-s)(-\Deltalb^\kappa)^{1/2})\cP_\kappa Q^{1/2}}_{HS}^2\right)\dd s\\
&=2\E\left[\cE^\kappa(0)\right] +\int_0^t\norm{\cP_\kappa  Q^{1/2}}^2_{HS}\dd s=2\E\left[\cE^\kappa(0)\right] +t\ \tr(Q^\kappa)
\end{align*}
and thus concludes the proof.
\end{proof}

\subsection{Forward Euler--Maruyama scheme}
It is now time to study the evolution of the total energy for the three time integrators (forward/backward Euler--Maruyama schemes and exponential Euler scheme,
also called trigonometric scheme in this context), see Section~\ref{sect-setting}.

We start by investigating the long-time behavior, with respect to the energy, of the forward Euler--Maruyama scheme~\eqref{eEM}. When applied to the stochastic wave equation on the sphere~\eqref{swe}, this time integrator reads 
\begin{align}
    \begin{split}
\label{eqn:SWE_EM}
u_{1,n+1}^\kappa&=u_{1,n}^\kappa+\tau u_{2,n}^\kappa,\\
u_{2,n+1}^\kappa&=u_{2,n}^\kappa+\tau \Deltalb^\kappa u_{1,n}^\kappa+\cP_\kappa\Delta L_n.
    \end{split}
\end{align}

The total energy of the fully discrete solution is then defined as
\begin{equation}\label{swe-full-energy}
\mathcal{E}^\kappa_n=\frac{1}{2}\left(\norm{u_{2,n}^\kappa}^2+\norm{(-\Delta_{\IS^2}^\kappa)^{1/2}u_{1,n}^\kappa}^2\right).
\end{equation}

Expanding the numerical solution~\eqref{eqn:SWE_EM} in terms of the spherical harmonics (omitting the index $\kappa$ in their coefficients for ease of presentation)
\begin{equation*}
u_{1,n}^\kappa = \sum_{\ell = 0}^{\kappa} \sum_{m=-\ell}^{\ell} u_{1,n}^{\ell, m} Y_{\ell, m}\quad\text{and}\quad
u_{2,n}^\kappa = \sum_{\ell = 0}^{\kappa} \sum_{m=-\ell}^{\ell} u_{2,n}^{\ell, m} Y_{\ell, m},
\end{equation*}
one then obtains the system of equations, for $\ell=0,\ldots,\kappa$, $m=-\ell,\ldots,\ell$, and $n\geq 0$,
\begin{align}
    \begin{split}
    \label{eqn:SWE_EM2}
    u_{1,n+1}^{\ell, m} &= u_{1, n}^{\ell, m} + \tau u_{2, n}^{\ell, m}\\
    u_{2,n+1}^{\ell, m} &= u_{2, n}^{\ell, m} - \tau  \lambda_\ell u_{1, n}^{\ell, m} + \Delta L^{\ell, m}_n
    \end{split}
\end{align}
for the forward Euler--Maruyama scheme~\eqref{eqn:SWE_EM}. Note that $\E\left[(\Delta L^{\ell, m}_n)^2\right]=v_{\ell,m}t_n>0$ by~\eqref{Lvlm}.

The energy of the forward Euler--Maruyama scheme~\eqref{eqn:SWE_EM} can thus be written as
\begin{align*}
\mathcal{E}^\kappa_n&=\frac12\left( \norm{\sum_{\ell = 0}^{\kappa} \sum_{m=-\ell}^{\ell} u_{2,n}^{\ell, m} Y_{\ell, m} }^2
+\norm{(-\Delta_{\IS^2}^\kappa)^{1/2}\sum_{\ell = 0}^{\kappa} \sum_{m=-\ell}^{\ell} u_{1,n}^{\ell, m} Y_{\ell, m} }^2  \right)\\
&=\frac12\left( \sum_{\ell = 0}^{\kappa} \sum_{m=-\ell}^{\ell}\left| u_{2,n}^{\ell, m}\right|^2
+  \sum_{\ell = 0}^{\kappa} \sum_{m=-\ell}^{\ell} \left| \sqrt{\lambda_\ell} u_{1,n}^{\ell, m} \right|^2  \right)
\end{align*}
due to orthonormality of the spherical harmonics. We denote the above as
\begin{equation}\label{energEM}
\mathcal{E}^\kappa_n=\mathcal{E}^{0,0}_n+\sum_{\ell=1}^{\kappa}\sum_{m=-\ell}^{\ell} \mathcal{E}^{\ell,m}_n,
\end{equation}
where we set
$$
\mathcal{E}^{\ell,m}_n= \frac{1}{2}\left(\left| u_{2,n}^{\ell, m}\right|^2+\left| \sqrt{\lambda_\ell} u_{1,n}^{\ell, m} \right|^2\right).
$$
The next result shows that the expected energy along the forward Euler--Maruyama scheme, grows exponentially
instead of linearly as for the exact solution, see Proposition~\ref{prop:energy-exact}.
\begin{proposition}\label{prop:energy-em}
Consider the stochastic wave equation on the sphere~\eqref{swe} with initial values $(u(0),\partial_tu(0))=(v_1,v_2)\in [L^2(\IS^2,\R)]^2$.
Assume that the expected value of the initial energy $\E\left[\mathcal{E}(0)\right]<\infty$ and that the $Q$-L\'evy process $L$ is trace-class.
Furthermore, assume that the expected value of the initial energy satisfies $0<\E\left[\mathcal{E}^\kappa_0-\cE_0^{0,0}\right]$.
Let $(u_{1,n}^\kappa, u_{2,n}^\kappa)$ denote the fully-discrete approximation given by the forward Euler--Maruyama scheme~\eqref{eqn:SWE_EM}.
The energy of this fully-discrete approximation grows exponentially fast in time:
\begin{equation*}
\E\left[\mathcal{E}^\kappa_n\right] \geq \exp\left(\frac12\tau t_n\right)\E\left[\mathcal{E}^\kappa_0 - \mathcal{E}_0^{0,0}\right]
\end{equation*}
for every discrete times $t_n=n\tau$ with integers $n\geq0$.
\end{proposition}
\begin{proof}
Let $\ell\geq1$ and consider the component $\mathcal E_{n+1}^{\ell,m}$ of the total energy $\mathcal{E}^\kappa_{n+1}$ in equation~\eqref{energEM}.
We now borrow some ideas in the proof of \cite[Theorem~3]{MR2083326} given for scalar linear oscillators driven by a standard Brownian motion.
Using the definitions of $\mathcal E_{n+1}^{\ell,m}$ and of the forward Euler--Maruyama scheme~\eqref{eqn:SWE_EM2}, we obtain
\begin{align}\label{energy0EM}
    \begin{split}
    \E\left[\cE_{n+1}^{\ell,m}\right] &=\frac{1}{2}\E\left[\lambda_\ell\left(u_{1,n+1}^{\ell, m}\right)^2 + \left(u_{2,n+1}^{\ell, m}\right)^2\right]\\
    &= \frac{1}{2} \E\left[\lambda_\ell \left(u_{1,n}^{\ell, m}\right)^2 + \tau^2 \lambda_\ell \left(u_{2,n}^{\ell, m}\right)^2
    + 2 \lambda_\ell^2\tau u_{1,n}^{\ell, m}u_{2,n}^{\ell, m} + \left(u_{2,n}^{\ell, m}\right)^2\right.\\
    & \left.+ \tau^2 \lambda_\ell^2\left(u_{1,n}^{\ell, m}\right)^2 - 2 \lambda_\ell^2 \tau u_{2,n}^{\ell, m}u_{1,n}^{\ell, m} + \left(\Delta L^{\ell, m}_n\right)^2\right]\\
    &=\E\left[\cE_{n}^{\ell,m}\right] +  \frac12\tau^2 \lambda_\ell \E\left[\left(u_{2,n}^{\ell, m}\right)^2 + \lambda_\ell\left(u_{1,n}^{\ell, m}\right)^2\right]+\frac{1}{2}\E\left[\left(\Delta L^{\ell, m}_n\right)^2\right]\\
    & \geq \left(1 + \lambda_\ell\tau^2\right)\E\left[\cE_{n}^{\ell,m}\right].
    \end{split}
\end{align}

An iteration of the above provides us with the inequality
\begin{equation*}
\E\left[\cE_{n}^{\ell,m}\right] \geq (1 + \lambda_\ell\tau^2)^n\E\left[\cE_{0}^{\ell,m}\right].
\end{equation*}
For a time-step size $\tau<1$, one thus obtains
$$
\E\left[\cE_{n}^{\ell,m}\right] \geq (1 + \lambda_\ell\tau^2)^n\E\left[\cE_{0}^{\ell,m}\right]\geq \exp\left(\frac12\tau^2 n\right)\E\left[\cE_{0}^{\ell,m}\right]
\geq \exp\left(\frac12\tau t_n\right)\E\left[\cE_{0}^{\ell,m}\right].
$$
The expected total energy of the forward Euler--Maruyama scheme thus satisfies
$$
\E\left[\mathcal{E}^\kappa_n\right]=\E\left[\cE_{n}^{0,0}\right]+\sum_{\ell=1}^\kappa\sum_{m=-\ell}^{\ell}\E\left[\cE_{n}^{\ell,m}\right]\geq
\exp\left(\frac12\tau t_n\right)\sum_{\ell=1}^\kappa\sum_{m=-\ell}^{\ell}\E\left[\cE_{0}^{\ell,m}\right]
\geq \exp\left(\frac12\tau t_n\right)\E\left[\mathcal{E}^\kappa_0 - \mathcal{E}_0^{0,0}\right],
$$
completing the proof of the proposition.
\end{proof}
\begin{remark}
We remark that if the initial energy $\E\left[\mathcal{E}^\kappa_0\right]=0$, then instead of the inequality~\eqref{energy0EM}, one would get the relation
\begin{equation*}
\E\left[\cE_{1}^{\ell,m}\right] = \left(1 + \lambda_\ell\tau^2\right)\E\left[\cE_{0}^{\ell,m}\right]+\frac{1}{2}\E\left[\left(\Delta L^{\ell, m}_0\right)^2\right]=\frac12v_{\ell,m}\tau\geq0
\end{equation*}
for the first time-step. One then obtains the following exponential increase in the energy for the Euler--Maruyama scheme
\begin{equation*}
\E\left[\mathcal{E}^{\ell,m}_{n+1}\right] \geq \exp\left(\frac12\tau t_n\right)\frac12v_{\ell,m}\tau.
\end{equation*}
\end{remark}

\subsection{Backward Euler--Maruyama method}
Our next goal is to investigate the long-time behavior, with respect to the energy, of the backward Euler--Maruyama scheme~\eqref{bEM} when applied to the stochastic wave equation on the sphere~\eqref{swe}. This numerical scheme is given by the following implicit relation
\begin{align}
\begin{split}
    \label{eqn:SWE_BEM}
    u_{1,n+1}^\kappa&=u_{1,n}^\kappa+\tau u_{2,n+1}^\kappa,\\
    u_{2,n+1}^\kappa&=u_{2,n}^\kappa+\tau \Deltalb^\kappa u_{1,n+1}^\kappa+\cP_\kappa\Delta L_n.
\end{split}
\end{align}
As we did previously for the forward Euler--Maruyama scheme, we obtain the following system of equation for the backward Euler--Maruyama scheme~\eqref{eqn:SWE_BEM}
\begin{align}
    \begin{split}
    \label{eqn:SWE_BEM2}
    u_{1,n+1}^{\ell, m} &= u_{1, n}^{\ell, m} + \tau u_{2, n+1}^{\ell, m}\\
    \left(1+\tau^2\lambda_\ell\right)u_{2,n+1}^{\ell, m} &= u_{2, n}^{\ell, m} - \tau\lambda_\ell u_{1, n}^{\ell, m} + \Delta L^{\ell, m}_n
    \end{split}
\end{align}
for $\ell=0,\ldots,\kappa$, $m=-\ell,\ldots,\ell$, and $n\geq 0$. We use the same notation for the energy of the backward Euler--Maruyama scheme as
for the forward Euler--Maruyama scheme, see~\eqref{energEM}.

We now prove that the expected energy of the backward Euler--Maruyama scheme grows more slowly than that of the exact solution, see Proposition~\ref{prop:energy-exact}.
\begin{proposition}\label{prop:energy-bem}
Consider the stochastic wave equation on the sphere~\eqref{swe} with initial values $(u(0),\partial_tu(0))=(v_1,v_2)\in [L^2(\IS^2,\R)]^2$.
Assume that the expected value of the initial energy $\E\left[\mathcal{E}(0)\right]<\infty$ and that the $Q$-L\'evy process $L$ is trace-class.
Furthermore, assume that the expected value of the initial energy satisfies $0<\E\left[\mathcal{E}^\kappa_0\right]$.
Let $\left(u_{1,n}^\kappa, u_{2,n}^\kappa\right)$ denote the fully-discrete approximation given by the backward Euler--Maruyama scheme~\eqref{eqn:SWE_BEM}.
The energy of this fully-discrete approximation grows at a slower rate than the exact solution to the considered SPDE:
\begin{equation*}
		\E\left[\cE^\kappa_n\right] <
         \E\left[\cE^\kappa_0\right]+ \frac{1}{2\tau}\tr\left(Q^\kappa\right) + \frac {t_n} {2} v_{0,0}
\end{equation*}
for every discrete times $t_n=n\tau$ with integers $n\geq0$ and where we recall that $v_{0,0}=\Var(L_{0,0}(1))$. More so, we have
\begin{equation}\label{eqn:SWE_energy_BEM_limit}
    \lim_{t_n\to\infty}\frac{\E\left[\mathcal{E}^\kappa_n\right]}{t_n}=\frac 1 2 v_{0,0}.
\end{equation}
\end{proposition}
\begin{proof}	
	Let $\ell\geq 1$. As in the proof of Proposition~\ref{prop:energy-em}, we substitute \eqref{eqn:SWE_BEM2} in the $(\ell, m)$ mode of the energy and obtain:
	\begin{align*}
		\E\left[\cE^{\ell, m}_{n+1}\right] &= \frac{1}{2}\E\left[\left(u^{\ell, m}_{2, n+1}\right)^2 + \lambda_\ell \left(u^{\ell, m}_{1, n+1}\right)^2\right]=\left(1 + \tau^2\lambda_\ell\right)^{-1} \left(\E\left[\cE^{\ell, m}_{n}\right] + \frac{1}2v_{\ell, m}\tau\right).
	\end{align*}
	We again observe that for the mode $\ell = 0$, the energy is preserved exactly.
	However, for all other modes, energy is lost for the backward Euler--Maruyama scheme.
	
	In particular, we have that
	\[
	\E\left[\cE^{\ell, m}_{n+1}\right] \leq \left(1 + \tau^2\right)^{-1} \left(\E\left[\cE^{\ell, m}_{n}\right] + \frac{1}2v_{\ell, m}\tau\right).
	\]
	Iterating this for the $n$--th increment implies
	\begin{align*}
		\E\left[\cE^{\ell, m}_{n}\right] &\leq \left(1+\tau^2\right)^{-n} \E\left[\cE^{\ell, m}_{0}\right] + \sum_{k=1}^{n}\left(1+\tau^2\right)^{-k}\frac{1}{2}v_{\ell, m}\tau\\
		& \leq (1+\tau^2)^{-n} \E\left[\cE^{\ell, m}_{0}\right]  + \frac{\tau}{2}v_{\ell,m} \left(\frac{1-\left(1 + \tau^2\right)^{-n} }{\tau^2}\right)\\
		& \leq (1+\tau^2)^{-n} \E\left[\cE^{\ell, m}_{0}\right]  + \frac{1}{2\tau}v_{\ell,m}.
	\end{align*}
	
	Summing across all modes, this gives directly the evolution of the expected energy:
	\begin{align}\label{bem:bound}
		\E\left[\cE^\kappa_n\right] &\leq \left(1+\tau^2\right)^{-n} \E\left[\cE^\kappa_0 - \cE_0^{0, 0}\right]+ \frac{1}{2\tau}\tr\left(Q^\kappa\right) + \E\left[\cE_n^{0,0}\right],
	\end{align}
    where we note that $\tr\left(Q^\kappa\right)=\sum_{\ell=0}^{\kappa}\sum_{m=-\ell}^{\ell} v_{\ell,m}$.
	This clearly shows that the energy of the numerical solution grows more slowly than the energy of the true solution. 

    Moreover, one observes that the inequality \eqref{bem:bound} implies the estimates
	\begin{equation*}
		\E\left[\cE^\kappa_n\right] \leq (1+\tau^2)^{-n} \E\left[\cE^\kappa_0\right]+ \frac{1}{2\tau}\tr\left(Q^\kappa\right) + \frac {t_n} {2} v_{0,0} < \E\left[\cE^\kappa_0\right]+ \frac{1}{2\tau}\tr\left(Q^\kappa\right) + \frac {t_n} {2} v_{0,0}.
	\end{equation*}
	This concludes the proof of the proposition.
\end{proof}

\subsection{Stochastic trigonometric method}
In the context of the stochastic wave equation on the sphere~\eqref{swe}, the exponential Euler scheme~\eqref{expEM} is also known as the stochastic trigonometric scheme. This explicit numerical scheme reads 
\begin{align}
\label{eqn:SWE_trigo}
\begin{split}
	u_{1,n+1}^\kappa&=\cos(\tau(-\Deltalb^\kappa)^{1/2})u_{1,n}^\kappa+(-\Deltalb^\kappa)^{-1/2}\sin(\tau(-\Deltalb^\kappa)^{1/2})u_{2,n}^\kappa\\
	&\quad+(-\Deltalb^\kappa)^{-1/2}\sin(\tau(-\Deltalb^\kappa)^{1/2})\cP_\kappa\Delta L_n,\\
    u_{2,n+1}^\kappa&=-(-\Deltalb^\kappa)^{1/2}\sin(\tau(-\Deltalb^\kappa)^{1/2})u_{1,n}^\kappa+\cos(\tau(-\Deltalb^\kappa)^{1/2})u_{2,n}^\kappa\\
    &\quad+\cos(\tau(-\Deltalb^\kappa)^{1/2})\cP_\kappa\Delta L_n.
\end{split}    
\end{align}

Unlike both previously studied Euler--Maruyama schemes, the stochastic trigonometric scheme~\eqref{eqn:SWE_trigo} has the same long-time behavior with respect to the energy as the exact solution to the stochastic wave equation on the sphere~\eqref{swe}.
\begin{proposition}\label{prop:energy-trigo}
Consider the stochastic wave equation on the sphere~\eqref{swe} with initial values $(u(0),\partial_tu(0))=(v_1,v_2)\in [L^2(\IS^2,\R)]^2$.
Assume that the expected value of the initial energy $\E\left[\mathcal{E}(0)\right]<\infty$ and that the $Q$-L\'evy process $L$ is trace-class.
Let $(u_{1,n}^\kappa, u_{2,n}^\kappa)$ denote the fully-discrete approximation given by the stochastic trigonometric integrator~\eqref{eqn:SWE_trigo}.
This fully-discrete approximation satisfies the trace formula for the energy:
\begin{align*}
\E\left[\mathcal{E}^\kappa_n\right]=\E\left[\mathcal{E}^\kappa_0\right]+\frac{t_n}2\tr(Q^\kappa),
\end{align*}
for every discrete times $t_n=n\tau$ with integers $n\geq0$.
\end{proposition}
\begin{proof}
Let us first observe that the stochastic part of the fully-discrete numerical solution~\eqref{eqn:SWE_trigo} can be written as a stochastic 
integral:
$$
\cos(\tau(-\Deltalb^\kappa)^{1/2})\cP_\kappa\Delta L_{n}=
\int_{t_{n}}^{t_{n+1}}\cos(\tau(-\Deltalb^\kappa)^{1/2})\cP_\kappa\dd L(s)
$$
and similarly for the other term. An application of It\^o's isometry then gives 
$$
\E\left[\norm{\cos(\tau(-\Deltalb^\kappa)^{1/2})\cP_\kappa\Delta L_{n}}^2\right]=
\int_{t_{n}}^{t_{n+1}}\norm{\cos(\tau(-\Deltalb^\kappa)^{1/2})\cP_\kappa Q^{1/2}}_{HS}^2\dd s.
$$
Next, we insert the numerical solution and use the above into the energy~\eqref{swe-full-energy}.
Finally, one uses trigonometric identities and gets the relation
\begin{align*}
2\E\left[\cE^\kappa_{n+1}\right] &= 2\E\left[\cE^\kappa_{n}\right] +\int_{t_{n}}^{t_{n+1}}\left(\norm{\sin(\tau(-\Deltalb^\kappa)^{1/2})\cP_\kappa Q^{1/2}}^2_{HS}+\norm{\cos(\tau(-\Deltalb^\kappa)^{1/2})\cP_\kappa Q^{1/2}}_{HS}^2\right)\dd s\\
&=2\E\left[\cE^\kappa_n\right] +\int_{t_{n}}^{t_{n+1}}\norm{\cP_\kappa  Q^{1/2}\cP_\kappa}^2_{HS}\dd s= 2\E\left[\cE^\kappa_n\right]+\tau\ \tr(Q^\kappa).
\end{align*}
A recursion on the index $n$ concludes the proof.
\end{proof}

\subsection{Numerical experiments}
We proceed with some numerical experiments illustrating the above theoretical results.

Let us first illustrate the behavior of the Euler--Maruyama schemes on a short time interval. To do this,
we consider the linear stochastic wave equation on the sphere~\eqref{swe} on the time interval $[0,T]$ with $T=3$.
As initial values, we choose the Gaussian random fields
$$
v_1 = \sum_{\ell=0}^{\kappa} \sum_{m=-\ell}^{\ell} \ell^{-\gamma} u_{\ell , m} Y_{\ell , m}, \quad v_2 = \sum_{\ell=0}^{\kappa} \sum_{m=-\ell}^{\ell} \ell^{-\gamma+1} z_{\ell , m} Y_{\ell , m},
$$
where $\gamma=3+10^{-5}$ (using the convention $0^{-x}=1$ for $x>0$), and where $u_{\ell , m}$ and $z_{\ell , m}$ are iid standard Gaussian random variables. For the L\' evy noise, we choose
$$
L = \sum_{\ell=0}^{\kappa} \sum_{m=-\ell}^{\ell} A_\ell L_{\ell, m} Y_{\ell , m},
$$
for independent Lévy processes $L_{\ell,m}$, given by $L_{\ell, m}(t) = \left(W_{\ell, m}(t) + P_{\ell, m}(t) - t\right) / \sqrt{2}$.
That is, $L_{\ell, m}$ is a linear combination of a standard Brownian motion $W_{\ell, m}$ and a compensated Poisson process $P_{\ell, m}(t) - t$ that is independent of $W_{\ell, m}$.
This choice implies by \cite{marinucciRandomFieldsSphere2011a} that $Q^\kappa$, the covariance operator of the truncated noise, is the covariance operator of an isotropic random field.
Furthermore, we choose $A_0 = 1$ and $A_\ell = \ell^{-4}$ for $\ell = 1, \dots, \kappa$.
We use the spectral method~\eqref{spectralsol}
with truncation index $\kappa=2^6$ and take $N=500$ steps of each time integrators. The expectation of the energy is approximated using $M=10000$ Monte Carlo samples.
The results of this fist numerical experiment are presented in Figure~\ref{fig:swe-energy1}.
In this figure, one can observe the exponential increase in the expected energy of the forward Euler--Maruyama scheme~\eqref{eqn:SWE_EM}
as shown in Proposition~\ref{prop:energy-em}. Furthermore, the damping in the expected energy of the backward Euler--Maruyama
scheme~\eqref{eqn:SWE_BEM} as described in Proposition~\ref{prop:energy-bem} is observed. Finally, the correct behavior, in term
of the expected energy, of the stochastic trigonometric scheme~\eqref{eqn:SWE_trigo}, proved in Proposition~\ref{prop:energy-trigo}, is seen.

\begin{figure}[h!]
    \centering
    \begin{subfigure}[t]{0.5\textwidth}
        \centering
        \includegraphics[width=\textwidth]{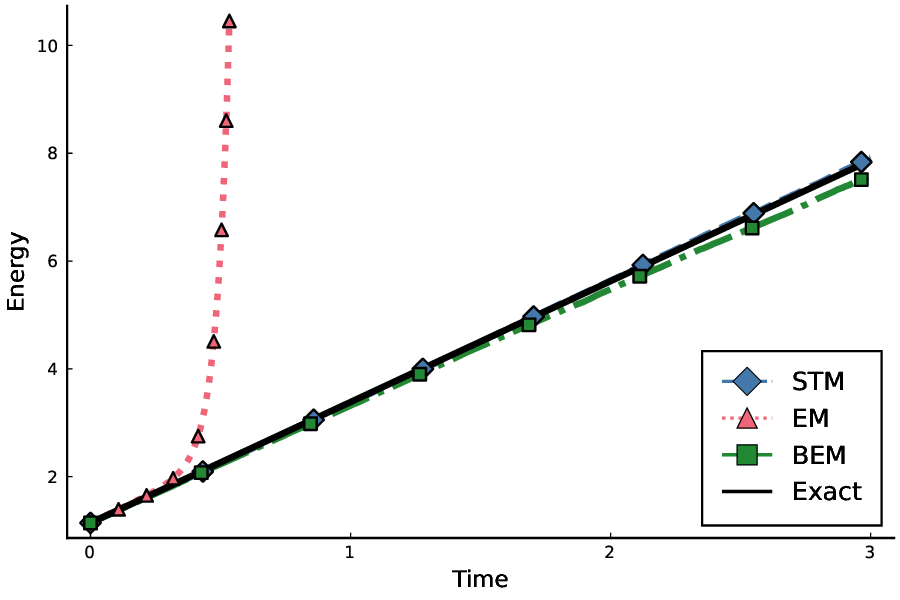}
        \caption{$T=3$}
        \label{fig:swe-energy1}
    \end{subfigure}%
    \begin{subfigure}[t]{0.5\textwidth}
        \centering
        \includegraphics[width=\textwidth]{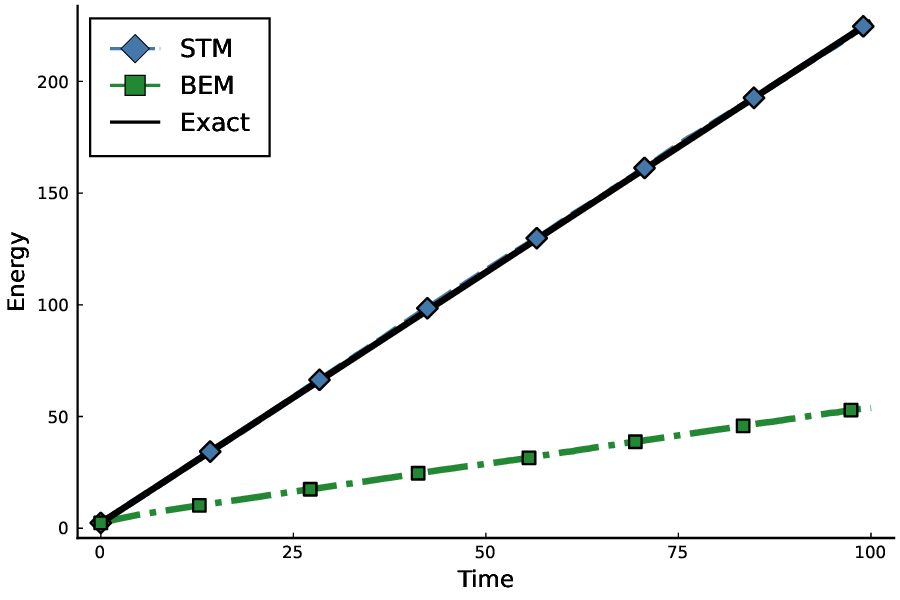}
        \caption{$T=100$}
        \label{fig:swe-energy2}
    \end{subfigure}
    \caption{Expected energy of the stochastic wave equation on the sphere~\eqref{swe}: The stochastic trigonometric scheme~\eqref{eqn:SWE_trigo} ({\upshape STM}),
    the forward Euler--Maruyama scheme~\eqref{eqn:SWE_EM} ({\upshape EM}), and the backward Euler--Maruyama scheme~\eqref{eqn:SWE_BEM} ({\upshape BEM}).}
    \label{fig:swe-energy}
\end{figure}

In order to illustrate the long-time behavior of the expected energy for the different time integrators,
we perform a simulation with final time $T=100$. All other parameters are kept the same, in particular we keep $N = 500$,
so that the time step is $\tau = 0.2$. Figure~\ref{fig:swe-energy2} shows the result of this numerical simulation.
Due to the poor long-time behavior, with respect to the expected energy, of the forward Euler--Maruyama scheme,
we do not display it in this figure. The stochastic trigonometric scheme behaves correctly,
with respect to the behavior of the expected energy, even for this long-time simulation.
For the backward Euler--Maruyama scheme, Figure~\ref{fig:swe-energy2} shows how the expected energy approaches a linear function with a lower slope than that of the exact expected energy, illustrating equation~\eqref{eqn:SWE_energy_BEM_limit}.

We conclude this section with a comment on the case when the mean of the L\'evy process in the SPDE~\eqref{swe} is not zero.

\begin{remark}
  If we loosen the requirement of considering only mean-zero L\'evy processes, some care needs to be taken to adapt the stochastic trigonometric scheme to the altered behavior of the energy of the true solution. A direct computation, analogous to that in the proof of Proposition~\ref{prop:energy-exact}, yields the following behavior of the energy of the exact solution:
\allowdisplaybreaks
\begin{align}
    \begin{split}
    \label{eqn:swe_energy_nonzero_mean}
    2 \E\left[\cE(t_{n+1})\right]
    & = 2\E\left[\cE(t_n)\right] + \tau\, \tr(Q) + 2 \norm{(-\Deltalb)^{-1/2}\mean }^2 \\
    & - 2\langle {(-\Deltalb)^{-1/2}}\cos(\tau(-\Deltalb)^{1/2})\mean , {(-\Deltalb)^{-1/2}}\mean  \rangle\\
    & - 2 \langle \E\left[{u_1(t_n)}\right], \mean  \rangle
    + 2 \langle \cos(\tau(-\Deltalb)^{1/2})\E\left[{u_1(t_n)}\right], \mean \rangle \\
    & + 2 \langle \sin(\tau(-\Deltalb)^{1/2})\E\left[{u_2(t_n)}\right], {(-\Deltalb)^{-1/2}}\mean \rangle.
    \end{split}
\end{align}

This behavior of the energy can be reproduced by an adapted stochastic trigonometric integrator that we now present.
In order to treat the drift term exactly, we write  $L(t) = L(t) - \mean t + \mean t$ and split the stochastic integral in the mild solution as
\begin{align*}
    \int_{t_n}^{t_{n+1}}S(t_{n+1}-s)\dd L(s) &= \int_{t_n}^{t_{n+1}}S(t_{n+1}-s)\dd(L(s) - \mean s)
    + \int_{t_n}^{t_{n+1}}S(t_{n+1}-s) \mean \dd s.
\end{align*}
Integrating the second deterministic integral exactly and treating the remaining terms as for the
stochastic trigonometric scheme~\eqref{eqn:SWE_trigo}, we obtain the following adapted stochastic trigonometric scheme:
\begin{align*}
    u_{1,n+1}^\kappa&=\cos(\tau(-\Deltalb)^{1/2})u_{1,n}^\kappa
    +(-\Deltalb)^{-1/2}\sin(\tau(-\Deltalb)^{1/2})u_{2,n}^\kappa\\
    & \quad+(-\Deltalb)^{-1/2}\sin(\tau(-\Deltalb)^{1/2})(\Delta L^\kappa_n - \mean^\kappa\tau)
     + (-\Deltalb)^{-1/2}\left(1 - \cos(\tau(-\Deltalb)^{1/2})\right)\mean^\kappa,\\
    u_{2,n+1}^\kappa&=-(-\Deltalb)^{1/2}\sin(\tau(-\Deltalb)^{1/2})u_{1,n}^\kappa
    +\cos(\tau(-\Deltalb)^{1/2})u_{2,n}^\kappa\\
    & \quad+\cos(\tau(-\Deltalb)^{1/2})(\Delta L^\kappa_n - \mean^\kappa\tau)
     + (-\Deltalb)^{-1/2}\sin(\tau(-\Deltalb)^{1/2})\mean^\kappa.
\end{align*}

A computation analogous to the proof of Proposition~\ref{prop:energy-trigo} yields that this scheme preserves the long-term behavior
of energy of the exact solution:
\begin{align*}
    \E\left[\cE^\kappa_{n+1}\right] &= \E\left[\cE^\kappa_n\right] + \frac{\tau}{2}\tr(Q^\kappa) + \norm{(-\Deltalb)^{-1/2}\mean^\kappa}^2 \\
    & \quad - \langle(-\Deltalb)^{-1/2} \cos(\tau(-\Deltalb)^{1/2}) \mean^\kappa, (-\Deltalb)^{-1/2} \mean^\kappa\rangle - \langle \E\left[{u_{1,n}^\kappa}\right], \mean^\kappa \rangle\\
    & \quad
    + \langle \cos(\tau(-\Deltalb)^{1/2})\E\left[{u_{1,n}^\kappa}\right], \mean^\kappa \rangle + \langle \sin(\tau(-\Deltalb)^{1/2})\E\left[{u_{2,n}^\kappa}\right], {(-\Deltalb)^{-1/2}}\mean^\kappa\rangle
\end{align*}
Direct computation shows that
\begin{equation*}
    \E\left[u^\kappa_{i, n}\right]  = \E\left[u^\kappa_i(t_{n})\right], \quad i = 1, 2,
\end{equation*}
and therefore, by comparison with equation~\eqref{eqn:swe_energy_nonzero_mean}, we obtain
\begin{equation*}
    \E\left[\cE^\kappa_{n+1}\right]  = \E\left[\cE^\kappa(t_{n+1})\right].
\end{equation*}

We conclude this remark by numerically illustrating the long-time behavior, in terms of the energy, of the adapted stochastic trigonometric integrator in Figure~\ref{fig:SWE_energy_nonzero_mean}. For this numerical experiment, all parameters are chosen as in the first experiment
besides choosing a non-compensated Poisson process instead, i.\,e., for all $\ell = 0, \dots, \kappa, m=-\ell, \dots, \ell$, we take
$L_{\ell , m} = \left(W_{\ell, m} + P_{\ell, m}\right)/\sqrt{2}$, where we recall that $W_{\ell, m}$ is a standard Brownian motion
and $P_{\ell, m}$ a Poisson process. In Figure~\ref{fig:swe-energy1_nonzero_mean}, one can observe that the oscillations of the exact energy, given by the additional terms in equation~\eqref{eqn:swe_energy_nonzero_mean}, are recovered by the adapted stochastic trigonometric integrator.
The forward Euler--Maruyama scheme exhibits exponential growth in the energy while the backward Euler--Maruyama scheme underestimates the energy.
In Figure~\ref{fig:swe-energy2_nonzero_mean}, we observe that the adapted stochastic trigonometric scheme retains the correct long-time behavior, regarding the evolution of the energy, as opposed to the backward Euler--Maruyama scheme which fails to do so.

\begin{figure}[h!]
    \centering
    \begin{subfigure}[t]{0.5\textwidth}
        \centering
        \includegraphics[width=\textwidth]{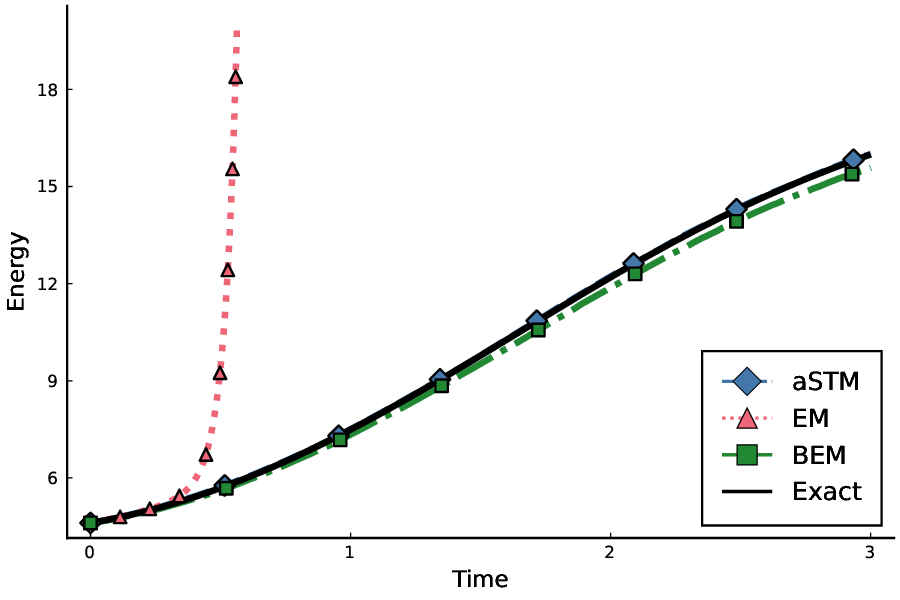}
        \caption{$T=3$}
        \label{fig:swe-energy1_nonzero_mean}
    \end{subfigure}%
    \begin{subfigure}[t]{0.5\textwidth}
        \centering
        \includegraphics[width=\textwidth]{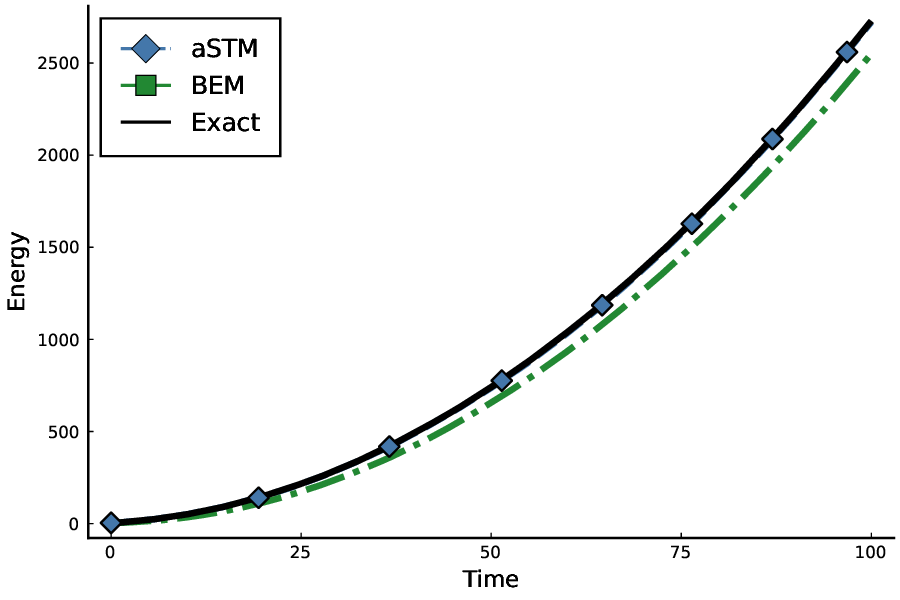}
        \caption{$T=100$}
        \label{fig:swe-energy2_nonzero_mean}
    \end{subfigure}
    \caption{Expected energy of the stochastic wave equation on the sphere~\eqref{swe} driven by a Lévy process with nonzero mean: The adapted stochastic trigonometric scheme~\eqref{eqn:SWE_trigo} ({\upshape aSTM}),
    the forward Euler--Maruyama scheme~\eqref{eqn:SWE_EM} ({\upshape EM}), and the backward Euler--Maruyama scheme~\eqref{eqn:SWE_BEM} ({\upshape BEM}).}
    \label{fig:SWE_energy_nonzero_mean}
\end{figure}
\end{remark}

\section{Stochastic Schr\"odinger equation on the sphere}\label{sect-sch}
\begin{toappendix}
  \label{apx:Schrödinger}
\end{toappendix}
It is well-known that the deterministic free Schr\"odinger equation on an interval has the
mass, energy, and momentum as conserved quantities,
see, e.\,g., \cite{carlesSemiclassicalAnalysisNonlinear2008,lubichQuantumClassicalMolecular2008}.
In this section, inspired by the previous section and the work \cite{MR3771721} for $Q$-Wiener noise in the Euclidean setting, we investigate the long-time behavior of the exact and numerical solutions with respect to these quantities for the stochastic Schr\"odinger equation~\eqref{eqn:SSE} driven by a (possibly complex-valued) trace-class $Q$-L\'evy process $L$ on ${\IS^{2}}$ with mean zero:
\begin{equation*}
	i \partial_t u(t) = \Deltalb u(t) + \dot{L}(t).
\end{equation*}
We follow the same structure as in Section~\ref{sect-swe} and start our study in Subsection~\ref{sch-exact} with the exact solution of~\eqref{eqn:SSE}. The behavior of the spectral and time discretizations is presented in Subsections~\ref{sch-spectral}~and~\ref{sch-time}. The obtained results are then numerically confirmed in Subsection~\ref{sch-num}. Note that the proofs of our results
are conceptually similar to those in Section~\ref{sect-swe}, however in some cases more technically involved.
Therefore, we defer all the proofs to Appendix~\ref{apx:Schrödinger}.

\subsection{Exact solution}\label{sch-exact}
The mild solution~\eqref{mild_s} to the stochastic Schr\"odinger equation on the sphere~\eqref{eqn:SSE} is given by
\begin{equation}
    \label{eqn:SSE_mildsol}
    u(t) = \expp{-it\Deltalb}u_0 - i \int_{0}^{t}\expp{-i(t-s)\Deltalb}\dd L(s).
\end{equation}
Its semigroup $\expp{-it\Deltalb}$ is a semigroup of unitary operators and, therefore, an isometry:
\begin{equation}
    \label{eqn:SSE_SG_isometry}
    \norm{\expp{-it\Deltalb}x} = \norm{x}
\end{equation} 
for all $t \geq 0, x \in {L^2(\IS^2; \IC)}$, where in this section, we use the notation $\norm{\cdot} = \norm{\cdot}_{L^2(\IS^2; \IC)}$ and, correspondingly, $\langle \cdot, \cdot \rangle =\langle \cdot, \cdot \rangle_{L^2({\IS^{2}}; \IC)}$.

The mass, energy, and momentum of the stochastic Schr\"odinger equation on the sphere~\eqref{eqn:SSE} are given by
\begin{align}
    \cM(t) &= \norm{u(t)}^2,\label{eqn:def_mass}\\
    \cE(t) &= \norm{\nablalb u(t)}_{L^2(\IS^2; \IC^3)}^2 = \int_{\IS^{2}} \norm{\nablalb u(t, x)}_{\IC^3}^2 \dd x, \label{eqn:def_energy}\\
    p(u(t)) &= i\int_{\IS^{2}} u(t,x)\nablalb \overline{u}(t,x) - \overline{u}(t,x)\nablalb u(t,x) \dd x.\label{eqn:def_mom}
\end{align}
For ease of notation, we define the following vector of Hilbert--Schmidt inner products:
Given the spaces $\mathcal{H}=L^2({\IS^{2}}; \IC)$ and $\mathcal{H}^d = L^2({\IS^{2}}, \IC^d)$, and
two trace-class operators $A\colon \mathcal{H} \to \mathcal{H}$ and $B\colon \mathcal{H} \to \mathcal{H}^d$, we define
$$
\left<\left<A,B\right>\right>_{HS}=(\left<A,B_i\right>_{HS})_{i=1}^d\in\mathbb{C}^d,
$$
where $B_i v = (Bv)_i$ is the $i$-th component of $Bv$ for all $v \in \mathcal{H}$.

The following proposition shows the long-time behavior of the expected mass, energy, and momentum for
the exact solution to the stochastic Schrödinger equation on the sphere~\eqref{eqn:SSE}.

\begin{proposition}
    \label{prop:exact}
    Consider the stochastic Schrödinger equation on the sphere~\eqref{eqn:SSE}
    with a (possibly complex-valued) $Q$-Lévy process $L$ which is of trace-class.

    \begin{enumerate}[label=\roman*)]
    \item If the initial condition $u_0 \in {L^2(\IS^2; \IC)}$ and
    if $\E\left[\cM(0)\right] < \infty$, then the solution $u(t)$ to this SPDE satisfies the following trace formula for the mass: 
    \begin{align*}
        \E\left[\cM(t)\right] &= \E\left[\cM(0)\right] + t\,\tr(Q),
    \end{align*}
    for every time $t \geq 0$.
    \item If $u_0 \in {H^1(\IS^2; \IC)}$, $\E\left[\cE(0)\right] < \infty$, and
    $Q^{1/2} (-\Deltalb) Q^{1/2}$ is trace-class, then $u(t)$ satisfies the following trace formula for the energy:
    \begin{align*}
        \E\left[\cE(t)\right] &= \E\left[\cE(0)\right] + t\,\tr(Q^{1/2}\nablalb^\ast\nablalb Q^{1/2}) =  \E\left[\cE(0)\right] + t\,\tr(Q^{1/2} (-\Deltalb) Q^{1/2}),
    \end{align*}
    for every time $t \geq 0$.

    \item If $u_0 \in {H^1(\IS^2; \IC)}$ and $\mathbb E\left[p(u(0))\right]<\infty$ and $\left<\left<Q^{1/2},\nablalb Q^{1/2}\right>\right>_{HS}$ are finite componentwise,
	then $u(t)$ satisfies the following trace formula for the momentum:
    \begin{align*}
        \E\left[p(u(t))\right] &= \E\left[p(u(0))\right]-2t\operatorname{Im}\left(\left<\left<Q^{1/2},\nablalb Q^{1/2}\right>\right>_{HS}\right),
    \end{align*}
    for every time $t \geq 0$.

    If, in addition, $Q$ is diagonalized by the (real-valued) spherical harmonics $Y_{\ell,m}$ or $L$ is a real-valued $Q$-L\'evy process, then,
    the momentum is a conserved quantity:
	$$
	\E\left[p(u(t))\right] = \E\left[p(u(0))\right],
	$$
    for every time $t \geq 0$.
     \end{enumerate}
\end{proposition}

\begin{appendixproof}[Proof of Proposition~\ref{prop:exact}]

   \begin{enumerate}[label=\roman*)]
    \item \label{pf-exact-mass} To show the long-time behavior of the mass for the stochastic Schr\"odinger equation on the sphere~\eqref{eqn:SSE}, we proceed
    similarly to the proof of Proposition~\ref{prop:energy-exact}. We insert the mild solution~\eqref{eqn:SSE_mildsol}
    into the definition of the mass~\eqref{eqn:def_mass} and obtain the relation
    \begin{align*}
    \E\left[\cM(t)\right] &= \E\left[\norm{u(t)}^2\right]\\
    & = \E\left[\norm{\expp{-it\Deltalb}u_0 - i\int_{0}^{t}\expp{-i(t-s)\Deltalb}\dd L(s)}^2\right]\\
    & = \E\left[\norm{\expp{-it\Deltalb}u_0}^2 + \norm{i \int_{0}^{t}\expp{-i(t-s)\Deltalb}\dd L(s)}^2\right.\\
    & \left. \quad - 2 \Re \left\langle \expp{-it\Deltalb}u_0, i \int_{0}^{t} \expp{-i(t-s)\Deltalb}\dd L(s) \right\rangle\right],
    \end{align*}
    where the last equality follows from properties of complex-valued inner products.
    By the martingale property of the stochastic integral (e.\,g., \cite[Corollary 8.17]{peszat2007}),
    an It\^o isometry (e.\,g., \cite[Corollary 8.17]{peszat2007}), and noting that $\expp{-it\Deltalb}$ is an isometry, we obtain the following relation:
    \begin{align*}
        \E\left[\cM(t)\right] &= \E\left[\norm{u_0}^2\right] + \int_{0}^{t}\norm{\expp{-i(t-s)\Deltalb}Q^{1/2}}^2_{HS({L^2(\IS^2; \IC)}; {L^2(\IS^2; \IC)})} \dd s\\
        &= \E\left[\cM(0)\right] + \int_{0}^{t} \sum_{\ell=0}^{\infty}\sum_{m=-\ell}^{\ell}\norm{\expp{-i(t-s)\Deltalb}Q^{1/2}Y_{\ell , m}}^2\dd s\\
        &= \E\left[\cM(0)\right] + \int_{0}^{t} \sum_{\ell=0}^{\infty}\sum_{m=-\ell}^{\ell}\norm{Q^{1/2}Y_{\ell , m}}^2\dd s\\
        &= \E\left[\cM(0)\right] + \int_{0}^{t} \sum_{\ell=0}^{\infty}\sum_{m=-\ell}^{\ell}\langle Q Y_{\ell , m}, Y_{\ell , m}\rangle\dd s\\
        & = \E\left[\cM(0)\right] + t\,\tr(Q),
    \end{align*}
    where we used the self-adjointness of $Q$ and the definitions of the Hilbert--Schmidt norm and of the trace of an operator.
    This proves the long-time behavior of the mass.
    \item A direct computation, similar to Item~\ref{pf-exact-mass}, allows us to write the expected energy as
    \begin{align*}
        \E\left[\cE(t)\right] &=\E\left[\norm{\nablalb u(t)}_{L^2(\IS^2; \IC^3)}^2\right] \\
         &= \E\left[\norm{\nablalb u_0}^2_{{L^2(\IS^2; \IC^3)}}\right] + \int_{0}^{t} \norm{\nablalb \expp{-i(t-s)\Deltalb}Q^{1/2}}_{\HS({L^2(\IS^2; \IC)}; {L^2(\IS^2; \IC^3)})} \dd s.
    \end{align*}
    By the definition of the Hilbert--Schmidt norm, we then obtain that
    \begin{align*}
        \E\left[\cE(t)\right] &= \E\left[\norm{\nablalb u_0}^2_{{L^2(\IS^2; \IC^3)}}\right] + \int_{0}^{t}\tr\left(Q^{1/2} \nablalb^\ast \nablalb Q^{1/2}\right)\dd s\\
        & = \E\left[\cE(0)\right] + t\,\tr\left(Q^{1/2} \nablalb^\ast \nablalb Q^{1/2}\right).
    \end{align*}
    Note that, by the integration by parts formula for compact Riemannian manifolds without boundary (see, e.g., \cite[Chapter 2]{leeIntroductionRiemannianManifolds2018}), $\nablalb^\ast$ is the divergence operator on $\IS^2$. Therefore, $\nablalb^\ast \nablalb = - \Deltalb$ and we finally obtain
    \begin{equation*}
        \E\left[\cE(t)\right] = \E\left[\cE(0)\right] + t\,\tr\left(Q^{1/2} (-\Deltalb) Q^{1/2}\right)
    \end{equation*}
    for all time $t\geq0$. This proves the long-time behavior of the energy.

\item For simplicity's sake, we start by computing the term $\int_{\IS^2}u(t)\nablalb \overline{u}(t)\dd x$ in the momentum~\eqref{eqn:def_mom}. We have:
\begin{align*}
    \int_{\IS^2}u(t)\nablalb \overline{u}(t)\dd x
   &= \int_{\IS^2}\expp{-it\Deltalb}u_0 \nablalb \overline{\expp{-it\Deltalb}u_0}\dd x \\
   &+ i \int_{\IS^2} \expp{-it\Deltalb}u_0\overline{\int_{0}^{t}\nablalb
   	\expp{-i(t-s)\Deltalb}\dd L(s)}\dd x\\
   &- i \int_{\IS^2} \int_{0}^{t}\expp{-i(t-s)\Deltalb}\dd L(s)\nablalb \overline{\expp{-it\Deltalb}u_0}\dd x \\
   &+ \int_{\IS^2} \int_{0}^{t}\expp{-i(t-s)\Deltalb}\dd L(s)\overline{\int_{0}^{t}\nablalb \expp{-i(t-s)\Deltalb}\dd L(s)}\dd x.
\end{align*}
Taking expectation, since $L(s)$ is independent of $u_0$, we obtain the relation
\begin{align}
    \begin{split}
    \E\left[\int_{\IS^2} u(t)\nablalb \overline{u}(t)\dd x\right] =&\E\left[ \int_{\IS^2} \exp{(-it\Deltalb)}u_0 \nablalb \overline{\exp{(-it\Deltalb)}u_0}\dd x  \right] \\
    &+ \E\left[\int_{\IS^2} \int_{0}^{t}\exp{(-i(t-s)\Deltalb)}\dd L(s)\overline{\int_{0}^{t}\nablalb \exp{(-i(t-s)\Deltalb)}\dd L(s)}\dd x\right]. \label{whatareyoulabel}
    \end{split}
\end{align}
Starting with the first term in \eqref{whatareyoulabel}, we have componentwise, for $j=1,2,3$,
\begin{align*}
    \int_{\IS^2} \exp{(-it\Deltalb)}u_0 \partial_j \overline{\exp{(-it\Deltalb)}u_0}\dd x &=\int_{\IS^2} \exp{(-it\Deltalb)}u_0  \overline{\exp{(-it\Deltalb)}\partial_j u_0}\dd x\\
    &=\left<\exp{(-it\Deltalb)}u_0,\exp{(it\Deltalb)}\partial_j u_0\right>_{L^2(\IS^2,\mathbb C)}\\
    &=\left<u_0,\partial_j u_0\right>_{L^2(\IS^2,\mathbb C)}.
\end{align*}
Hence, we obtain the relation
$$
\int_{\IS^2} \exp{(-it\Deltalb)}u_0 \nablalb \overline{\exp{(-it\Deltalb)}u_0}\dd x=\int_{\IS^2} u_0 \nablalb \overline{u_0}\dd x.
$$
Turning our attention to the second term in \eqref{whatareyoulabel}, we have to compute
\begin{align}\label{exp_mom}
	\E\left[\int_{\IS^2} \int_{0}^{t}\exp{(-i(t-s)\Deltalb)}\dd L(s)\overline{\int_{0}^{t}\nablalb \exp{(-i(t-s)\Deltalb)}\dd L(s)}\dd x\right].
\end{align}
In the general case, when we don't assume commutativity of $Q$ and $\Deltalb$, we obtain that this term is equal to
\begin{align*}
	&\int_{0}^{t}\left<\left<\exp{(-i(t-s)\Deltalb)Q^{1/2}},\left(\nablalb \exp{(-i(t-s)\Deltalb)}Q^{1/2}\right)\right>\right>_{HS}\dd s\\
	&=\int_{0}^{t}\left<\left<Q^{1/2},\nablalb Q^{1/2}\right>\right>_{HS}\dd s,
\end{align*}
where the last equality follows from properties of the Hermitian inner product and from the isometry property~\eqref{eqn:SSE_SG_isometry} of the semigroup.

The above computation finally yields
\begin{equation*}
	\E\left[\int_{\IS^2} u(t)\nablalb \overline u(t)\dd x\right] = \E\left[\int_{\IS^2} u_0\nablalb \overline u_0\dd x\right]+t\left(\left<\left<Q^{1/2},\nablalb Q^{1/2}\right>\right>_{HS}\right).
\end{equation*}

Analogously, one obtains
\begin{equation*}
	\E\left[\int_{\IS^2} \overline{u}(t)\nablalb u(t)\dd x\right] = \E\left[\int_{\IS^2} \overline{u}_0\nablalb  u_0\dd x\right]+t\left(\left<\left<\nablalb Q^{1/2}, Q^{1/2}\right>\right>_{HS}\right).
\end{equation*}

Summing up both integrals gives the relation
\begin{align}
    \begin{split}\label{anotherequation}
	\E\left[p(u(t))\right] &= \E\left[ p(u_0)\right]+it\left(\left<\left< Q^{1/2}, \nablalb Q^{1/2}\right>\right>_{HS}-\left<\left<\nablalb Q^{1/2}, Q^{1/2}\right>\right>_{HS}\right)\\
	&= \E\left[ p(u_0)\right]+it\left(\left<\left< Q^{1/2}, \nablalb Q^{1/2}\right>\right>_{HS}-\overline{\left<\left< Q^{1/2}, \nablalb Q^{1/2}\right>\right>_{HS}}\right)\\
	&= \E\left[ p(u_0)\right]-2t\operatorname{Im}\left(\left<\left< Q^{1/2}, \nablalb Q^{1/2}\right>\right>_{HS}\right)
    \end{split}
\end{align}
which shows the first result on the long-time behavior of the momentum.

To prove the second statement, we first assume that $Q$ is real. We then obtain
$$
\left<\left<Q^{1/2},\nablalb Q^{1/2}\right>\right>_{HS} = 0.
$$

Assume now that \(Q\) is diagonal in the same basis as $\Deltalb$, i.\,e.
\[
Q Y_{\ell m} = q_{\ell m} Y_{\ell m}
\]
with real-valued spherical harmonics. Then, for the term $\operatorname{Im}\left(\left<\left<Q^{1/2},\nablalb Q^{1/2}\right>\right>_{HS}\right)$ in \eqref{anotherequation},
we obtain
\begin{align*}\label{diag_im}
	\left<\left<Q^{1/2},\nablalb Q^{1/2}\right>\right>_{HS}&=  \sum_{\ell=0}^\infty \sum_{m=-\ell}^\ell  \int_{\mathbb S^2} q_{\ell,m}^{1/2}Y_{\ell m}(x) \, \overline{q_{\ell,m}^{1/2}\nablalb Y_{\ell m}(x)} \dd x\\
	&=\sum_{\ell=0}^\infty \sum_{m=-\ell}^\ell \abs{q_{\ell m}} \int_{\mathbb S^2} Y_{\ell m}(x) \, \overline{\nablalb Y_{\ell m}(x)} \dd x \\
	&=\sum_{\ell=0}^\infty \sum_{m=-\ell}^\ell \abs{q_{\ell m}} \int_{\mathbb S^2} Y_{\ell m}(x) \,{\nablalb Y_{\ell m}(x)} \dd x \\
	&=\sum_{\ell=0}^\infty \sum_{m=-\ell}^\ell \abs{q_{\ell m}} \int_{\mathbb S^2} {\frac 12\nablalb Y^2_{\ell m}(x)} \dd x=0.
\end{align*}
The last inequality follows using the divergence theorem since the sphere is a compact manifold without boundary.
This implies the conservation of momentum, i.\,e.
\begin{align*}
	\E\left[p(u(t))\right]
	&= \E\left[ p(u_0)\right].
\end{align*}
This concludes the proof.\qed 
\end{enumerate}
\renewcommand{\qedsymbol}{}
\end{appendixproof}

\subsection{Spectral discretization}\label{sch-spectral}
We now study the long-time behavior, with respect to mass, energy, and momentum, of the spectral discretization~\eqref{spectralsol} of
the stochastic Schrödinger equation on the sphere~\eqref{eqn:SSE}. 
The spatially semi-discretized mild solution reads
\begin{equation}
    \label{eqn:SSE_spectral}
    u^\kappa(t) = \expp{-it\Deltalb^\kappa}u^\kappa(0) - i\int_{0}^{t} \expp{-i(t-s)\Deltalb^\kappa} \cP_\kappa \dd L(s).
\end{equation}
Similarly to the previous subsection, one defines the semi-discretized mass, energy and momentum as
\begin{align*}
    \cM^\kappa(t) &= \norm{u^\kappa(t)}^2,\ \cE^\kappa(t) = \norm{\nablalb^\kappa u^\kappa(t)}_{L^2(\IS^2; \IC^3)}^2,\  \\
    p^\kappa(u^\kappa(t)) &= i\int_{\IS^2} u^\kappa(t,x)\nablalb^\kappa \overline{u^\kappa}(t,x) - \overline{u^\kappa}(t,x)\nablalb^\kappa u^\kappa(t,x) \dd x,
\end{align*}
where we denote $\nablalb^\kappa = \mathcal P_\kappa \nablalb$.

The evolution of the expected mass, energy and momentum of the spatially semi-discretized mild solution~\eqref{eqn:SSE_spectral} is given by the following proposition.
\begin{proposition}\label{prop:semi}
    Consider the stochastic Schrödinger equation on the sphere~\eqref{eqn:SSE} with a (possibly complex-valued) $Q$-Lévy process $L$ which is of trace-class.
    Let $u^\kappa$ denote the spectral approximation given by equation~\eqref{eqn:SSE_spectral}.
    \begin{enumerate}[label=\roman*)]
    \item If the initial value $u_0 \in {L^2(\IS^2; \IC)}$ and if $\E\left[\cM^\kappa(0)\right] < \infty$, then
    the spectral approximation $u^\kappa$ satisfies the following trace formula for the mass:
    \begin{equation*}
        \E\left[\cM^\kappa(t)\right] = \E\left[\cM^\kappa(0)\right] + t\,\tr(Q^\kappa),
    \end{equation*}
    for every time $t \geq 0$.
    \item If $u_0 \in {H^1(\IS^2; \IC)}$ and $\E\left[\cE^\kappa(0)\right] < \infty$, then
    the spectral approximation $u^\kappa$ satisfies the following trace formula for the energy:
	\begin{align*}
		\E\left[\cE^\kappa(t)\right] &= \E\left[\cE^\kappa(0)\right] + t\,\tr((\cP_\kappa Q^{1/2})^\ast(\nablalb^\kappa)^\ast\nablalb^\kappa\cP_\kappa Q^{1/2}) \\
		&=  \E\left[\cE^\kappa(0)\right] + t\,\tr((\cP_\kappa Q^{1/2})^\ast (-\Deltalb^\kappa) \cP_\kappa Q^{1/2}),
	\end{align*}
    for every time $t \geq 0$.
    \item If $u_0 \in {H^1(\IS^2; \IC)}$ and, componentwise, $\mathbb E\left[p^\kappa(u^\kappa(0))\right]<\infty$, then $u^\kappa$ satisfies the following trace formula for the momentum:
	\begin{equation*}
		\E\left[p^\kappa(u^\kappa(t))\right] = \E\left[p^\kappa(u^\kappa(0))\right]-2t\operatorname{Im}\left(\left<\left<\cP_\kappa Q^{1/2},\nablalb^\kappa \cP_\kappa Q^{1/2}\right>\right>_{HS}\right),
	\end{equation*}
    for every time $t \geq 0$.

	If, in addition, $Q^\kappa$ is diagonalized by the (real-valued) spherical harmonics $Y_{\ell,m}$ or $L$ is a real-valued $Q$-L\'evy process, then, the momentum is a conserved quantity:
	$$
	\E\left[p^\kappa(u^\kappa(t))\right] = \E\left[p^\kappa(u^\kappa(0))\right],
	$$
    for every time $t \geq 0$.
    \end{enumerate}
\end{proposition}

\begin{appendixproof}[Proof of Proposition~\ref{prop:semi}]
  \begin{enumerate}[label=\roman*)]
   \item The proof of the long-time behavior of the mass of the spectral discretization
   is identical to that of Proposition~\ref{prop:exact} using Proposition~\ref{prop:energy-spectral} for the treatment of the stochastic terms.
   \item To prove the long-time behavior of the energy of the spectral discretization, we observe that for, $\ell\leq\kappa$, we have
	$$
	(\nablalb^\kappa)^\ast\nablalb^\kappa Y_{\ell,m}=(\cP_\kappa \nablalb)^\ast\cP_\kappa\nablalb Y_{\ell,m} = \nabla_{\IS^{2}}^\ast\cP_\kappa \Phi_{\ell,m}=-\Deltalb Y_{\ell,m}=-\Deltalb^\kappa Y_{\ell,m},
	$$
	where we used properties of the vector spherical harmonics and the spectral truncation of the vector spherical harmonics.
	The proof is then identical to that for the exact solution given in Proposition~\ref{prop:exact}.

  \item To prove the long-time behavior of the momentum, it suffices to follow the steps of the corresponding proof for the exact solution given in Proposition~\ref{prop:exact} and
  to note that
	\begin{align*}
		&\E\left[\int_{\IS^2} \int_{0}^{t}\exp{(-i(t-s)\Deltalb^\kappa)}\dd L^\kappa(s)\overline{\int_{0}^{t}\nablalb^\kappa \exp{(-i(t-s)\Deltalb^\kappa)}\dd L^\kappa(s)}\dd x\right]\\
		&=\int_{0}^{t}\left<\left<\exp{(-i(t-s)\Deltalb^\kappa)\cP_\kappa Q^{1/2}},\nablalb^\kappa \exp{(-i(t-s)\Deltalb^\kappa)}\cP_\kappa Q^{1/2}\right>\right>_{HS}\dd s\\
		&=\int_{0}^{t}\left<\left<\cP_\kappa Q^{1/2},\nablalb^\kappa \cP_\kappa Q^{1/2}\right>\right>_{HS}\dd s,
	\end{align*}
	where we used properties of the projections onto the finite dimensional subspace for the spectral discretization.\qed
\end{enumerate}
\renewcommand{\qedsymbol}{}
\end{appendixproof}

\subsection{Temporal discretizations}\label{sch-time}

We now consider the evolution of mass, energy and momentum for fully discrete approximations of the solution to the stochastic Schrödinger equation on the sphere~\eqref{eqn:SSE}
obtained by the three time integrators defined in Section~\ref{sect-setting}.

While we will consider for the mass and energy both the backward and forward Euler--Maruyama schemes as well as the exponential integrator scheme, for the momentum we will only focus on the exponential integrator scheme. On the one hand, this is due to the wrong behavior observed for the Euler--Maruyama schemes until now. On the other hand, it is due to the technicality of the computations needed for the momentum. 
Furthermore, note that all the numerical schemes preserves the momentum in the case of a real valued L\'evy process.

We denote the fully discrete solution by $u^\kappa_n = \sum_{\ell=0}^{\kappa}\sum_{m=-\ell}^{\ell} u^{\ell,m}_n Y_{\ell , m}$ and define
the discretized mass by
\begin{equation*}
    M_n^\kappa = \norm{\sum_{\ell=0}^{\kappa}\sum_{m=-\ell}^{\ell}u^{\ell, m}_{n}Y_{\ell , m}}^2 
= \sum_{\ell=0}^{\kappa}\sum_{m=-\ell}^{\ell} \left\lvert u^{\ell, m}_{n} \right\rvert^2
= \sum_{\ell=0}^{\kappa}\sum_{m=-\ell}^{\ell} \cM_{n}^{\ell,m}.
\end{equation*}
Similarly, by the orthogonality of the vector spherical harmonics and Parseval's identity, the discretized energy decomposes as:
\begin{align*}
    \cE^\kappa_n = \norm{\nablalb^\kappa u^\kappa_n}^2 &= \norm{\sum_{\ell=0}^{\kappa}\sum_{m=-\ell}^{\ell}\nablalb^\kappa u^{\ell, m}_n Y_{\ell , m}}^2
    = \sum_{\ell=0}^{\kappa}\sum_{m=-\ell}^{\ell}\norm{\nablalb^\kappa u^{\ell, m}_n Y_{\ell , m}}^2
    = \sum_{\ell=0}^{\kappa}\sum_{m=-\ell}^{\ell} \cE^{\ell,m}_n.
\end{align*}

In the next subsections, we analyze long-time behavior, with respect to these quantities, for the three time integrators.

We start by investigating the long-time behavior of the forward Euler--Maruyama scheme~\eqref{eEM} which reads
\begin{align}\label{eqn:SSE_EM}
    u^\kappa_{n+1} = u^\kappa_n - i\tau \Deltalb^\kappa u^\kappa_n - i\Delta L^\kappa_{n}
\end{align}
or, decomposed, as the system of equations
\begin{equation*}
    u^{\ell,m}_{n+1} = u^{\ell,m}_{n} + i \tau \lambda_\ell u^{\ell,m}_{n} - i \Delta L^{\ell,m}_n
\end{equation*}
for $\ell = 0, \dots, \kappa, m = -\ell, \dots, \ell$ and $n \geq 0$.

The following proposition shows that the forward Euler--Maruyama scheme~\eqref{eqn:SSE_EM} exhibits an exponential increase in expected mass and energy.
\begin{proposition}
    \label{prop:semi_EM}
    Consider the stochastic Schrödinger equation on the sphere~\eqref{eqn:SSE} with a (possibly complex-valued) $Q$-L\'evy process $L$ which is of trace class.
    Let $u^\kappa_n$ denote the fully-discrete approximation given by the forward Euler--Maruyama method~\eqref{eqn:SSE_EM}.

    \begin{enumerate}[label=\roman*)]
    \item If the initial value $u_0 \in {L^2(\IS^2; \IC)}$ and $0 < \E\left[\cM_0^{\kappa} - \cM_0^{0,0}\right] < \infty$ and $\E\left[\cM_0^{0,0}\right] < \infty$, then
    the expected mass of $u^\kappa_n$ grows exponentially fast in time:
    \begin{equation}\label{eqn:mass_EM}
        \E\left[\cM^\kappa_n\right] \geq \expp{\frac{1}{2}\tau t_n}\E\left[\cM_0^{\kappa} - \cM_0^{0,0}\right],
    \end{equation}
    for every discrete time $t_n = n\tau, n\geq 0$.
    \item If $u_0 \in {H^1(\IS^2; \IC)}$ and  $0 < \E\left[\cE_0^{\kappa} - \cE_0^{0,0}\right] < \infty$ and $\E\left[\cE_0^{0,0}\right] < \infty$, then
     the expected energy of $u^\kappa_n$ grows exponentially fast in time:
    \begin{equation}\label{eqn:energy_EM}
        \E\left[\cE^\kappa_n\right] \geq \expp{\frac{1}{2}\tau t_n}\E\left[\cE_0^{\kappa} - \cE_0^{0,0}\right],
    \end{equation}
    for every discrete time $t_n = n\tau, n\geq 0$.
    \end{enumerate}
\end{proposition}

\begin{appendixproof}[Proof of Proposition~\ref{prop:semi_EM}]
\begin{enumerate}[label=\roman*)]
\item \label{pf-spectral-mass} Inserting the forward Euler--Maruyama scheme~\eqref{eqn:SSE_EM} into the definition of the fully discretized mass, one obtains
\begin{align*}
    \E\left[\cM_{n+1}^{\ell,m}\right] &= \E\left[\left\lvert u_{n+1}^{\ell,m} \right\rvert^2\right] = \E\left[\left\lvert u_n^{\ell,m} + i\tau \lambda_\ell u_n^{\ell,m} - i\Delta L^{\ell, m}_n \right\rvert^2\right]\\
    &=\E\left[\cM_{n}^{\ell,m}\right] + \E\left[\left\lvert \tau \lambda_\ell u_{n}^{\ell,m}\right\rvert^2\right] + \E\left[\left\lvert \Delta L^{\ell, m}_n \right\rvert^2\right]\\
    & \geq (1+\tau^2 \lambda_\ell^2)\E\left[\cM_{n}^{\ell,m}\right],
\end{align*}
recalling that $\E\left[\left\lvert \Delta L^{\ell, m}_n \right\rvert^2\right] = \Var(\Delta L^{\ell, m}_n) = \tau v_{\ell, m} \geq 0$.

By induction, one then obtains the inequality
\begin{equation*}
    \E\left[\cM_{n}^{\ell,m}\right] \geq (1+\tau^2 \lambda_\ell^2)^n\E\left[\cM_{0}^{\ell,m}\right].
\end{equation*}
Let now $\ell \geq 1$. Then, for $\tau < 1$, the above can be estimated as follows
\begin{equation*}
    \E\left[\cM_{n}^{\ell,m}\right]
    \geq (1+\tau^2)^n\E\left[\cM_{0}^{\ell,m}\right] \geq \expp{\frac{1}{2}\tau^2 n}\E\left[\cM_{0}^{\ell,m}\right] = \expp{\frac{1}{2}\tau t_n}\E\left[\cM_{0}^{\ell,m}\right].
\end{equation*}
Summing over all $\ell \geq 1, m = -\ell, \dots, \ell$, one finally obtains
\begin{equation*}
    \E\left[\cM^\kappa_n - \cM^{0,0}_n\right] = \sum_{\ell=1}^{\kappa}\sum_{m=-\ell}^{\ell}\E\left[\cM^{\ell,m}_n\right] \geq \expp{\frac{1}{2}\tau t_n} \sum_{\ell=1}^{\kappa}\sum_{m=-\ell}^{\ell} \E\left[\cM_{0}^{\ell,m}\right].
\end{equation*}
Noting that $\cM^{0,0}_n \geq 0$ almost surely, the long-time behavior of the mass~\eqref{eqn:mass_EM} follows.

\item Similarly to Item~\ref{pf-spectral-mass}, for the energy, one obtains for $\ell \geq 1, m=-\ell, \dots, \ell$, the relation
    \begin{align*}
        \E\left[\cE^{\ell,m}_{n+1}\right] = \E\left[\norm{\nablalb^\kappa u_{n+1}^{\ell, m}Y_{\ell, m}}^2\right] = \E\left[\norm{\nablalb^\kappa \left(\Id - i\tau\Deltalb^\kappa\right)u_n^{\ell, m}Y_{\ell , m}}^2 + \norm{\nablalb^\kappa \Delta L_n^{\ell, m}Y_{\ell , m}}^2\right].
    \end{align*}
    Applying the eigendecomposition of the discretized Laplace--Beltrami operator yields
    \begin{align*}
        \E\left[\norm{\nablalb^\kappa u_{n+1}^{\ell, m}Y_{\ell, m}}^2\right] \geq \abs{1 + i\tau\lambda_\ell}^2\E\left[\norm{\nablalb^\kappa u_n^{\ell, m}Y_{\ell , m}}^2\right] \geq \left(1 + \tau^2\right)\E\left[\norm{\nablalb^\kappa u_n^{\ell, m}Y_{\ell , m}}^2\right].
    \end{align*}
    The exponential growth of the energy~\eqref{eqn:energy_EM} then follows as in Item~\ref{pf-spectral-mass} above.\qed
\end{enumerate}
\renewcommand{\qedsymbol}{}
\end{appendixproof}

We now turn our attention to the long-time behavior of the backward Euler--Maruyama method~\eqref{bEM} which reads
\begin{align}\label{eqn:BEM}
    u^\kappa_{n+1} = \left(\Id + i\tau \Deltalb^\kappa\right)^{-1}\left(u_n^\kappa - i\Delta L^\kappa_{n}\right)
\end{align}
or, decomposed, as the system of equations
\begin{align}\label{eqn:BEM_decomposed}
    u_{n+1}^{\ell, m} &= \frac{u_n^{\ell,m} - i\Delta L^{\ell, m}_n}{1 - i\tau \lambda_\ell} = \left(u_n^{\ell,m} - i\Delta L^{\ell, m}_n\right)\frac{1+i\tau\lambda_\ell}{1 + \tau^2 \lambda_\ell^2}
\end{align}
for $\ell = 0, \dots, \kappa, m = -\ell, \dots, \ell$ and $n \geq 0$.

The long-time behavior of the expected mass and energy of the backward Euler--Maruyama scheme~\eqref{eqn:BEM} are given by the next result.
\begin{proposition}
    \label{prop:semi_BEM}
    Consider the stochastic Schrödinger equation on the sphere~\eqref{eqn:SSE} with a (possibly complex-valued) $Q$-L\'evy process $L$ which is of trace class.
    Let $u^\kappa_n$ denote the fully-discrete approximation given by the backward Euler--Maruyama method~\eqref{eqn:BEM}.
    \begin{enumerate}[label=\roman*)]
    \item If the initial value $u_0 \in {L^2(\IS^2; \IC)}$ and $\E\left[\cM^\kappa_0\right] < \infty$, then the expected mass of $u^\kappa_n$ grows at most as
    \begin{equation}\label{eqn:mass_BEM_bound}
        \E\left[\cM^\kappa_n\right] \leq \E\left[\cM^\kappa_0\right] + \frac{\tr(Q^\kappa)}{\tau} + t_n v_{0,0},
    \end{equation}
    for every discrete time $t_n = n\tau, n\geq 0$. We recall that $v_{0,0} = \Var\left(L^{0, 0}(1)\right)$.

    Furthermore, one has the relation
    \begin{equation}\label{eqn:mass_BEM_limit}
        \lim_{n\rightarrow\infty} \frac{\E\left[\cM^\kappa_n\right]}{t_n}  = v_{0,0}.
    \end{equation}
     \item If $u_0 \in {H^1(\IS^2; \IC)}$ and $\E\left[\cE^\kappa_0\right] < \infty$, then the expected energy of $u^\kappa_n$ is bounded above by
    \begin{equation}\label{eqn:SSE_energy_BEM_bound}
        \E\left[\cE^\kappa_n\right] \leq \E\left[\cE^\kappa_0\right] + \frac{\tr\left(Q^\kappa(-\Deltalb^\kappa)\right)}{\tau},
    \end{equation}
    for every discrete time $t_n = n\tau, n\geq 0$. Therefore, one has the relation
    \begin{equation}\label{eqn:SSE_energy_BEM_limit}
        \lim_{n\rightarrow\infty} \frac{\E\left[\cE^\kappa_n\right]}{t_n}  = 0.
    \end{equation}
 \end{enumerate}
\end{proposition}

\begin{appendixproof}[Proof of Proposition~\ref{prop:semi_BEM}]
\begin{enumerate}[label=\roman*)]
    \item \label{pf-bem-mass} Similar computations as in the proof of Proposition~\ref{prop:semi_EM} show that the expected mass of the backward Euler--Maruyama scheme~\eqref{eqn:BEM_decomposed} reads
\begin{align*}
     \E\left[\cM_{n+1}^{\ell,m}\right]&=\E\left[\left\lvert u_{n+1}^{\ell,m} \right\rvert^2\right] = \E\left[\left\lvert \left(1 - i\tau \lambda_\ell\right)^{-1}\left(u_n^{\ell,m} - i\Delta L^{\ell, m}_n\right)\right\rvert^2\right] \\
    & = \frac{1}{1 + \tau^2 \lambda_\ell^2}\left(\E\left[\left\lvert u_n^{\ell,m}\right\rvert^2 \right] + \E\left[\left\lvert \Delta L^{\ell, m}_n\right\rvert^2\right]\right)\\
    & = \frac{1}{1 + \tau^2 \lambda_\ell^2}\left(\E\left[\cM_n^{\ell,m} \right] + \tau v_{\ell, m}\right)\\
    & \leq \frac{1}{1 + \tau^2}\left(\E\left[\cM_n^{\ell,m} \right] + \tau v_{\ell, m}\right).
\end{align*}
An iteration of this procedure then yields
\begin{align*}
    \E\left[\cM_n^{\ell,m}\right] \leq \frac{1}{\left(1+\tau^2\right)^n}\E\left[\cM_0^{\ell,m}\right] + v_{\ell,m}\sum_{j=1}^{n}\frac{\tau }{\left(1+\tau^2\right)^j}.
\end{align*}
Bounding the geometric sum, one arrives at the estimates
\begin{align*}
    \E\left[\cM_n^{\ell,m}\right] \leq \E\left[\cM_0^{\ell,m}\right] + \frac{v_{\ell, m}}{\tau}.
\end{align*}
Summing over $\ell \geq 1$ and $m = -\ell, \dots, \ell$, as in the proof of Proposition~\ref{prop:semi_EM}, one obtains
\begin{equation}\label{eqn:mass_BEM_bound_lgeq1}
            \E\left[\cM^\kappa_n - \cM_n^{0,0}\right] \leq \E\left[\cM_0 - \cM_0^{0,0}\right] + \frac{\tr(Q^\kappa)}{\tau}.
\end{equation}
Lastly, we consider the mass for the component $\ell=0$ and note that it increases linearly in time:
\begin{equation*}
    \E\left[\cM_n^{0,0}\right] = \E\left[\cM_0^{0,0}\right] + t_n v_{0,0}.
\end{equation*}
Adding this equation on both sides of equation~\eqref{eqn:mass_BEM_bound_lgeq1} yields the claimed long-time behavior of the mass~\eqref{eqn:mass_BEM_bound} as well as equation~\eqref{eqn:mass_BEM_limit}.

 \item The components of the discretized energy are given, for $\ell \geq 1, m = -\ell, \dots, \ell$, by
    \begin{align*}
        \E\left[\norm{\nablalb^\kappa u_{n+1}^{\ell, m}}^2\right] &= \E\left[\norm{\nablalb^\kappa(\Id + i \tau \Deltalb^\kappa)^{-1}\left(u_n^{\ell, m}Y_{\ell , m} - i\Delta L^{\ell, m}_n Y_{\ell , m}\right)}^2\right]\\
        & = \frac{1}{1 + \tau^2\lambda_\ell^2}\E\left[\norm{\nablalb^\kappa u_n^{\ell, m} Y_{\ell , m}}^2 + \norm{\nablalb^\kappa \Delta L^{\ell, m}_nY_{\ell , m}}^2\right]\\
        & \leq \frac{1}{1+\tau^2}\left(\E\left[\cE_n^{\ell, m}\right] + \tau v_{\ell, m} \norm{\sqrt{\lambda_\ell}\bPhi_{\ell, m}}^2\right).
    \end{align*}
    Using that $\lambda_\ell v_{\ell, m} = \langle Q^\kappa (-\Deltalb^\kappa) Y_{\ell , m}, Y_{\ell , m}\rangle$,
    one finishes the proof as in Item~\ref{pf-bem-mass} and obtains the estimates
    \begin{equation*}
        \E\left[\cE^\kappa_n\right] \leq \E\left[\cE^\kappa_0\right] + \frac{\tr\left(Q^\kappa(-\Deltalb^\kappa)\right)}{\tau}.
    \end{equation*}
    Note that the energy of the mode $\ell = 0$ satisfies $\E\left[\cE_n^{0, 0}\right]=0$.\qed
\end{enumerate}
\renewcommand{\qedsymbol}{}
\end{appendixproof}

In this subsection, we show the long-time behavior, with respect to the mass, energy, and momentum, of the stochastic exponential Euler scheme~\eqref{expEM}, which reads
\begin{align}
        \label{eqn:exp_Euler_SSE}
    u^\kappa_{n+1} = \expp{-i\tau\Deltalb^\kappa} u^\kappa_n - i\expp{-i\tau\Deltalb^\kappa}\Delta L^\kappa_{n}.
\end{align}
The following proposition shows that the stochastic exponential Euler method~\eqref{eqn:exp_Euler_SSE} captures the correct evolution of these quantities,
see the corresponding result for the spatially semi-discrete solution in Proposition~\ref{prop:semi}.
\begin{proposition}
    \label{prop:semi_exp}
    Consider the stochastic Schrödinger equation on the sphere~\eqref{eqn:SSE} with a (possibly complex-valued) $Q$-L\'evy process $L$ which is of trace class.
    Let $u^\kappa_n$ denote the fully-discrete approximation given by the stochastic exponential Euler method~\eqref{eqn:exp_Euler_SSE}.
    \begin{enumerate}[label=\roman*)]
    \item If the initial value $u_0 \in {L^2(\IS^2; \IC)}$ and $\E\left[\cM^\kappa_0\right] < \infty$, then
    $u^\kappa_n$ satisfies the trace formula for the mass:
    \begin{equation}\label{longexpM}
        \E\left[\cM^\kappa_n\right] = \E\left[\cM^\kappa_0\right] + t_n\tr(Q^\kappa),
    \end{equation}
    for every discrete time $t_n = n\tau, n \geq 0$.
    \item If $u_0 \in {H^1(\IS^2; \IC)}$ and $\E\left[\cE^\kappa_0\right] < \infty$, then
    $u^\kappa_n$ satisfies the trace formula for the energy:
    \begin{equation}\label{longexpE}
        \E\left[\cE^\kappa_n\right] = \E\left[\cE^\kappa_0\right] + t_n\tr\left(\cP_\kappa Q^{1/2}(-\Deltalb^\kappa) \cP_\kappa Q^{1/2}\right),
    \end{equation}
    for every discrete time $t_n = n\tau, n \geq 0$.
    \item If $u_0 \in {H^1(\IS^2; \IC)}$, then $u^\kappa_n$ satisfies the trace formula for the momentum:
\begin{equation}\label{longexpMom}
	\E\left[p^\kappa(u^\kappa_n)\right] = \E\left[p^\kappa(u^\kappa_0)\right] -2t_n\operatorname{Im}\left(\left<\left<\cP_\kappa Q^{1/2},\nablalb^\kappa \cP_\kappa Q^{1/2}\right>\right>_{HS}\right)
\end{equation}
    for every discrete time $t_n = n\tau, n \geq 0$.

If, in addition, $Q^\kappa$ is diagonalized by the (real-valued) spherical harmonics $Y_{\ell,m}$ or $L$ is a real-valued $Q$-L\'evy process, then we have
$$
\E\left[p^\kappa(u^\kappa_n)\right] = \E\left[p^\kappa(u^\kappa_0)\right].
$$
\end{enumerate}
\end{proposition}

\begin{appendixproof}[Proof of Proposition~\ref{prop:semi_exp}]
    \begin{enumerate}[label=\roman*)]
    \item \label{pf-exp-mass} In order to compute the mass of the stochastic exponential Euler scheme~\eqref{eqn:exp_Euler_SSE}, we rewrite this numerical scheme as
   \begin{equation}\label{expNEWdef}
        u^\kappa_{n+1} = \expp{-i\tau\Deltalb^\kappa} u^\kappa_n - i \int_{t_n}^{t_{n+1}} \expp{-i\tau\Deltalb^\kappa}\cP_\kappa \dd L(s)
   \end{equation}
   and directly obtain
    \begin{align*}
    \E\left[\cM^\kappa_{n+1}\right]&=\E\left[\norm{u^\kappa_{n+1}}^2\right] = \E\left[\cM^\kappa_n\right] + {\tau}\tr(Q^\kappa).
    \end{align*}
    The result~\eqref{longexpM} on the correct long-time behavior of the mass then follows by induction on the index $n$.
\item The long-time behavior of the energy~\eqref{longexpE} is shown as in Item~\ref{pf-exp-mass}.
\item To show the long-time behavior of the momentum, we note that for the stochastic integral part in the time integrator~\eqref{expNEWdef},
we have
\begin{align*}
	\mathbb E &\left[i\int_{\IS^2} \int_{t_n}^{t_{n+1}}\expp{-i\tau\Deltalb^\kappa} \cP_\kappa \dd L(s)\int_{t_n}^{t_{n+1}}\nablalb^\kappa \expp{-i\tau\Deltalb^\kappa} \cP_\kappa \dd L(s)\right]\\
	&=i\int_{t_n}^{t_{n+1}}\left<\left<\expp{-i\tau\Deltalb^\kappa}\cP_\kappa Q^{1/2},\nablalb^\kappa\expp{-i\tau\Deltalb^\kappa}\cP_\kappa Q^{1/2}\right>\right>_{HS}\dd L(s)
\end{align*}
and the long-time behavior of the momentum~\eqref{longexpMom} follows as in the corresponding proof for the exact solution given in  Proposition~\ref{prop:exact}.\qed
\end{enumerate}
\renewcommand{\qedsymbol}{}
\end{appendixproof}

\subsection{Numerical experiments}\label{sch-num}
We conclude this section with some numerical experiments in order to illustrate the above results on the long-time behavior of the expected energy and mass.
Implementing a simulation for the momentum would be highly non-trivial since it requires the numerical evaluation of the integrals $\int_{\IS^2} Y_{\ell , m}(x) \nablalb \overline{Y_{\ell , m}(x)} \dd x$. This is outside of the scope of this article.

We simulate $M=10000$ Monte Carlo samples of the numerical solutions of the stochastic Schrödinger equation on the sphere~\eqref{eqn:SSE}
up to final time $T=3$ using the stochastic exponential Euler scheme~\eqref{eqn:exp_Euler_SSE},
the forward Euler--Maruyama scheme~\eqref{eqn:SSE_EM}, and the backward Euler--Maruyama scheme~\eqref{eqn:BEM}, respectively. The number of time steps is taken to be $N=300$. The truncation parameter for our spatial discretization is taken to be $\kappa = 2^3$. 
As initial values, we choose the Gaussian random field
$$
v = \sum_{\ell=0}^{\kappa} \sum_{m=-\ell}^{\ell} \left(\ell^{-\gamma} u_{\ell , m}+i\ell^{-\gamma+1} z_{\ell , m}\right) Y_{\ell , m}
$$
where $\gamma=3+10^{-5}$ (setting $0^{-x}=1$ for $x>0$), and where $u_{\ell , m}$ and $z_{\ell , m}$ are iid standard Gaussian random variables.
For the L\' evy noise,
we choose $L=\sum_{\ell=0}^{\kappa} \sum_{m=-\ell}^{\ell} A_\ell L_{\ell, m} Y_{\ell , m}$, for independent Lévy processes $L_{\ell,m}$,
given by $L_{\ell, m}(t) = (W_{\ell, m}(t) + P_{\ell, m}(t) - t) / \sqrt{2}$.
That is, $L_{\ell, m}$ is a linear combination of a standard Brownian motion $W_{\ell, m}(t)$ and a compensated Poisson process $P_{\ell, m}(t) - t$ that is independent of $W_{\ell, m}$. This choice implies, by \cite{marinucciRandomFieldsSphere2011a}, that the covariance operator of the truncated noise $Q^\kappa$ is the covariance operator of an isotropic random field.
Furthermore, we choose $A_0 = 1$ and $A_\ell = \ell^{-4}$ for $\ell = 1, \dots, \kappa$. 

The results of these simulations are shown in Figure~\ref{fig:masses}.
In the left part of this figure, it can be observed that the expected mass in the stochastic exponential Euler method~\eqref{eqn:exp_Euler_SSE} behaves as
the expected mass of the exact solution, hence confirming Proposition~\ref{prop:semi_exp}.
The forward Euler--Maruyama scheme~\eqref{eqn:SSE_EM}, however, exhibits the exponential increase in mass that was shown in Proposition~\ref{prop:semi_EM}. Lastly, the backward Euler--Maruyama method~\eqref{eqn:BEM} shows a mass that grows more slowly than that of the exact solution as shown in Proposition~\ref{prop:semi_BEM}.

\begin{figure}[h!]
    \centering
    \begin{subfigure}[t]{0.5\textwidth}
        \centering
        \includegraphics[width=\textwidth]{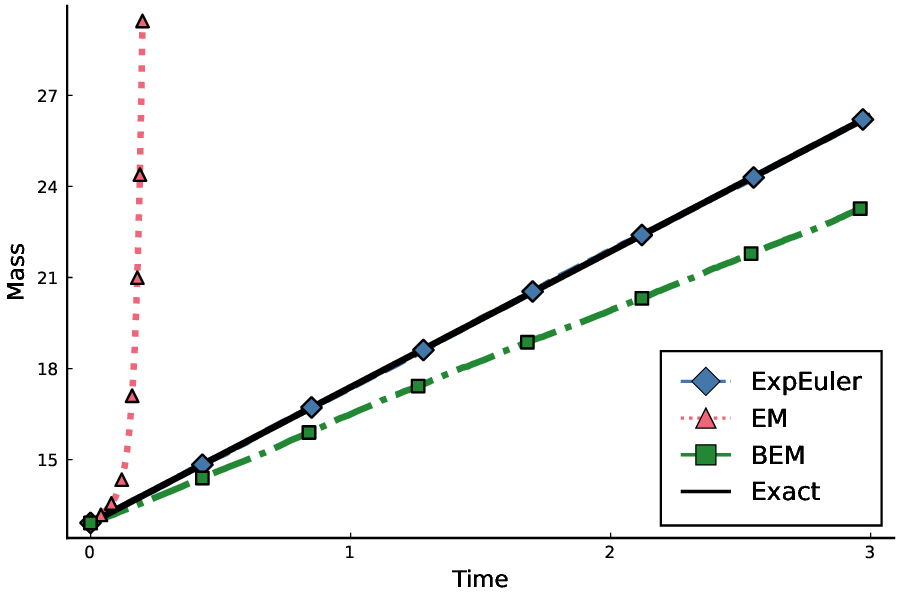}
        \caption{$T=3$}
        \label{fig:mass}
    \end{subfigure}%
    \begin{subfigure}[t]{0.5\textwidth}
        \centering
        \includegraphics[width=\textwidth]{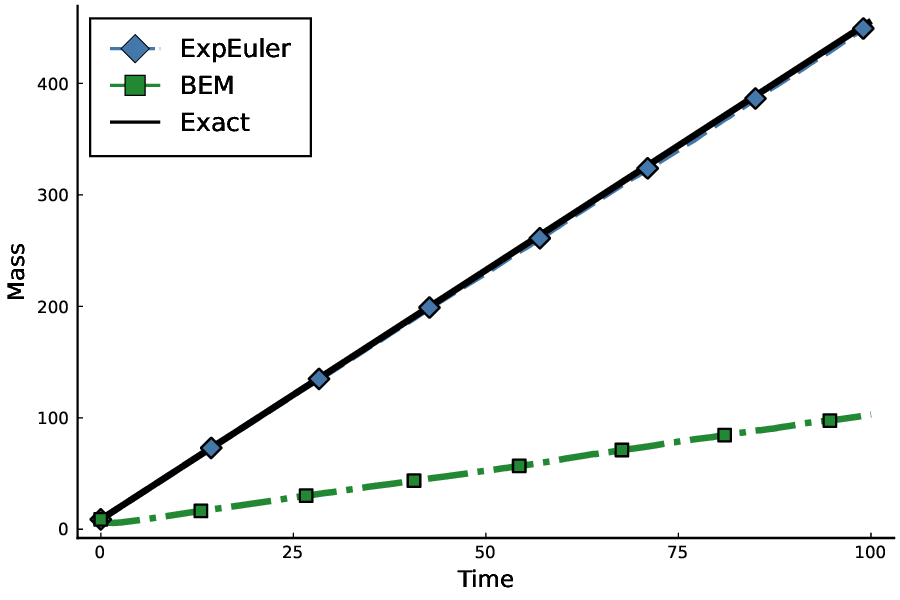}
        \caption{$T=100$}
        \label{fig:mass_longtime}
    \end{subfigure}
    \caption{Expected mass of the stochastic Schrödinger equation on the sphere~\eqref{eqn:SSE}:
    The stochastic exponential Euler scheme~\eqref{eqn:exp_Euler_SSE} ({\upshape ExpEuler}),
    the forward Euler--Maruyama scheme~\eqref{eqn:SSE_EM} ({\upshape EM}), and the backward Euler--Maruyama scheme~\eqref{eqn:BEM} ({\upshape BEM}).}
    \label{fig:masses}
\end{figure}

We further performed a long-time simulation of the behavior of the mass of the stochastic exponential Euler and backward Euler--Maruyama methods.
Due to the rapid explosion of the mass in the forward Euler--Maruyama method, we exclude it in this plot.
We choose the same parameters as in the previous simulation besides choosing the final time $T=100$.
This means that, since the number of time steps $N=300$, $\tau\approx 0.33$. Figure~\ref{fig:mass_longtime} shows
once again the correct long-time behavior of the stochastic exponential Euler scheme and the incorrect behavior of the backward Euler--Maruyama
scheme.

We finally present a simulation for the expected energy. We take the same parameters as in the previous experiment.
The results can be seen in Figure~\ref{fig:SSE_energies}. What contrasts this figure from the behavior of the mass in Figure~\ref{fig:masses}
is that we observe the boundedness of the expected energy for the backward Euler--Maruyama scheme, see Figure~\ref{fig:SSE_energy_longtime}.

\begin{figure}[h!]
    \centering
    \begin{subfigure}[t]{0.5\textwidth}
        \centering
        \includegraphics[width=\textwidth]{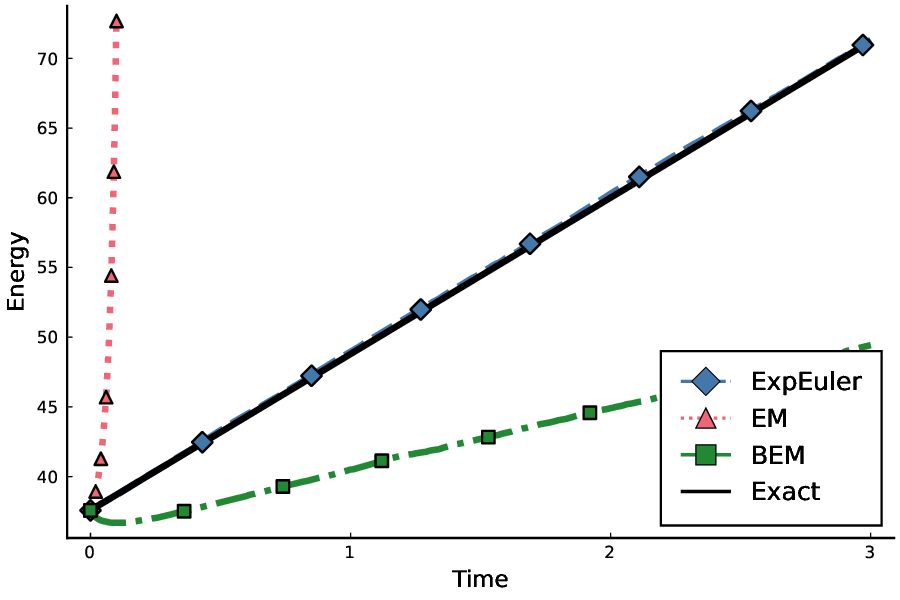}
        \caption{$T=3$}
        \label{fig:SSE_energy}
    \end{subfigure}%
    \begin{subfigure}[t]{0.5\textwidth}
        \centering
        \includegraphics[width=\textwidth]{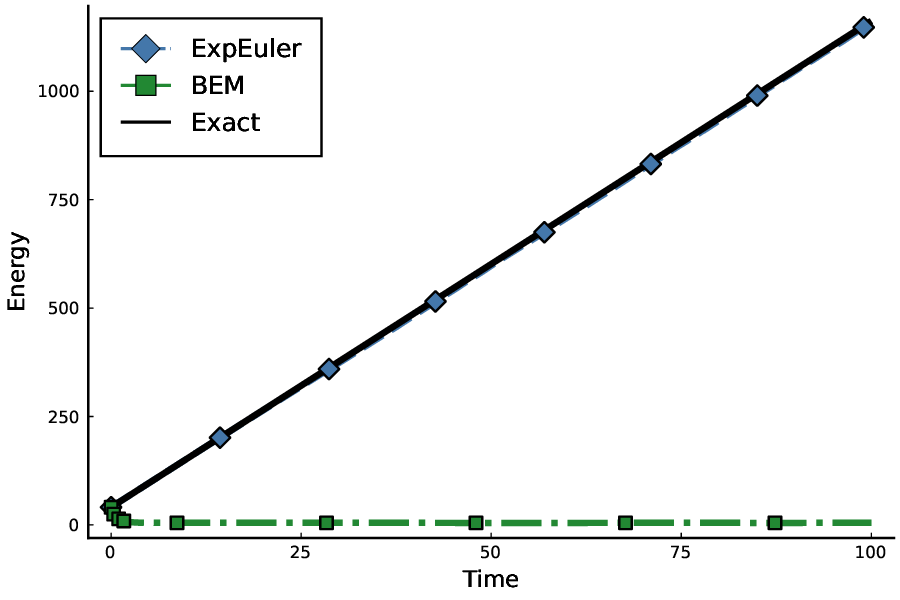}
        \caption{$T=100$}
        \label{fig:SSE_energy_longtime}
    \end{subfigure}
    \caption{Expected energy of the stochastic Schrödinger equation on the sphere~\eqref{eqn:SSE}: The stochastic exponential Euler scheme~\eqref{eqn:exp_Euler_SSE} ({\upshape ExpEuler}),
    the forward Euler--Maruyama scheme~\eqref{eqn:SSE_EM} ({\upshape EM}), and the backward Euler--Maruyama scheme~\eqref{eqn:SSE_EM} ({\upshape BEM}).}
    \label{fig:SSE_energies}
\end{figure}

\section{Stochastic Maxwell's equations on the sphere}\label{sect-maxeq}
\begin{toappendix}
  \label{apx:Maxwell}
\end{toappendix}
Maxwell's equations form the foundation of classical electromagnetism and are
traditionally formulated in three-dimensional Euclidean space or, more
generally, in four-dimensional spacetime, see for example \cite{CabralLobo2017,flanders1963differential,Milani1988}.
Under these usual formulations, Maxwell's equations are naturally defined on oriented Riemannian or Lorentzian
manifolds and do not rely on a specific choice of coordinates
\cite{CabralLobo2017,Hiptmair_2002,Munshi2022Complex,shi70,WECK1974410}.

This geometric perspective raises the question of whether Maxwell's equations can be meaningfully posed on manifolds of dimension lower than three. In particular, if it is possible to investigate such equations with Hamiltonian structure intrinsically on the unit sphere $\IS^2$, endowed with the geodesic metric, as a compact, boundaryless, two-dimensional Riemannian manifold. At first glance, the reduction in dimension appears problematic: familiar vector operators such as the three-dimensional $\operatorname{curl}$ do not exist intrinsically on $\IS^2$, and the physical interpretation as an electromagnetic field becomes less direct. 
Nevertheless,
Maxwell-type equations can still be defined consistently on the sphere via its embedding into $\R^3$.

Inspired by results such as \cite{1138693,MR4739347}, we therefore consider electromagnetic fields tangent to the sphere and we formulate Maxwell's equations in differential form intrinsically on the surface. 
While such equations no longer model physical electromagnetism in free space, they retain the mathematical structure of Maxwell's equations. Such formulations arise naturally in numerical analysis and in mathematical studies of Maxwell-type boundary problems \cite{MR3635830,MR4447703,Hiptmair_2002,MR4410964,MR4739347}.

In this section, we first define the proper phase space where 
the solutions of the stochastic Maxwell's equations in a vacuum on
the sphere~\eqref{maxeq} will live. We then investigate the long-time behavior of the expected energy of the solution to this stochastic Hamiltonian system.
Next, we study the spectral discretization of the SPDE~\eqref{maxeq}. Inspired by \cite{MR4077824}, the long-time behavior of the three time integrators is then presented. Finally, we provide numerical experiments to support and illustrate our theoretical results.
The proofs of the results in this section are given in Appendix~\ref{apx:Maxwell}.

\subsection{Exact solution}

We first need to
define the proper phase space where the solutions of the SPDE will live.
We consider the space $\mathcal{H}=L^2(\IS^2,\R^3)^{2}$,
which has the orthonormal basis given by the vector spherical harmonics~\eqref{3d-onb}.
We recall that, in a deterministic setting, Maxwell's equations on the sphere can be written formally
as
\begin{align}\label{detmax}
		\begin{split}
	\frac{\text{\large d}}{\text{\large dt}} \begin{pmatrix}
			E(t) \\ H(t)
		\end{pmatrix} &= \begin{pmatrix}
			\operatorname{curl}_{\IS^2} H(t) \\ -\operatorname{curl}_{\IS^2} E(t)
		\end{pmatrix},\\
		\operatorname{div}_{\IS^2}(E(t))&=0,\\
		\operatorname{div}_{\IS^2}(H(t))&=0,
		\end{split}
\end{align}
where $(E,H)\in \mathcal H$ represent the electric and magnetic field, respectively. In this formulation, we assume that permittivity and permeability are both equal to one, i.\,e. $\varepsilon=\mu\equiv1$ and that Gauss's law is respected, i.\,e. the divergence of the electric field, $\operatorname{div}_{\IS^2}(E)$, is zero.

Following the works \cite{MR3635830,FANG2025855,MR4739347}, we consider the transverse electric (TE) mode. That is, we assume that the electric field $E$ is tangential to the sphere $\IS^2$ and the magnetic field $H$ is in the direction
of the outward normal of the sphere. As such, we can restrict the space $\mathcal{H}$ by defining a subspace, which is still a Hilbert space, where the solutions to \eqref{detmax} naturally live. For this reason, we consider the space
\begin{equation}\label{spaceME}
	\mathcal{H}_0=L^2(\IS^2,T\IS^2)\times L^2(\IS^2),
\end{equation}
where the first space is the space of square-integrable tangential vector fields, representing the field $E$, and the second space is the usual space of square-integrable scalar fields, representing the normal component of $H$. Here, $T\IS^2$ denotes the usual tangent bundle on the sphere, i.\,e.,
$T\IS^2=\bigcup_{x\in\IS^2}\{x\}\times T_x\IS^2$, where we identify the tangent space as the plane in $\R^3$ tangent to the point $x\in \IS^2$, that is
$$
T_x\IS^2=\{v\in\R^3| v\cdot x=0\}.
$$
This means that we consider solutions $(E,H)\in\mathcal{H}_0$ such that $E\colon\IS^2\to\R^3$ and $E(x)\cdot x=0$ for all points $x$ on the sphere, and $H\colon\IS^2\to\R$.

\begin{remark}
	We have defined Maxwell's equations on the sphere considering solutions $E,H\in L^2(\IS^2,\R^3)$.
	Due to the choice of the TE mode and the geometry of the problem, we note that the vector field $E(x)$ is tangential to the point $x\in\IS^2$. This means that
	$$
	L^2(\IS^2,T\IS^2)\hookrightarrow L^2(\IS^2,\R^3).
	$$
	Furthermore, the magnetic field $H$ is normal to the surface, hence only its radial component is present. For this reason, we identify the magnetic field $H$ with its last component, the radial component and work on $L^2(\IS^2,\R)$.
\end{remark}

After these preparations, we can then write the Hamiltonian operator in equation~\eqref{detmax} as
\begin{align}\label{operatorME}
\mathcal{A}=\left(\begin{matrix}
	&0 \quad &\operatorname{curl}_{\IS^2}\\
	&-\operatorname{curl}_{\IS^2} \quad &0
\end{matrix}\right).
\end{align}
Following from \cite{monk2003finiteElement}, we have $D(\mathcal A)=H^1_{\operatorname{curl}}(\IS^2,T\IS^2)\times H^1_{\operatorname{curl}}(\IS^2)$, which, in the divergence-free setting reduces to $H^1(\IS^2,T\IS^2)\times H^1(\IS^2)$.
Considering the TE mode,
we note that the operator on the top right in~\eqref{operatorME} is a ``vector curl of a scalar'', i.\,e.,
$$
\operatorname{curl}_{\IS^2} H= \nabla\times (H\hat{{r}})|_{T\IS^2}=-\hat{n}\times(\nabla_{\IS^2}H),
$$
based on the identification of the magnetic scalar field with is radial component as a vector in $\R^3$.
We use the notation $\cdot |_{T {\IS^{2}}}$ to refer to the orthogonal projection onto the tangent bundle of the sphere.
The other operator in~\eqref{operatorME} is the surface curl, i.\,e., a scalar curl of a tangential field, defined as the coefficient of the projection of the curl in $\R^3$ onto the
normal component of the sphere
$$
\operatorname{curl}_{\IS^2} E= \hat{n}\cdot(\nabla\times \tilde E),
$$
where $\tilde E$ is a smooth extension of the field $E$ in $\R^3$. 
With this at hand, the Hamiltonian operator $\mathcal{A}$ is skew symmetric:
$\mathcal{A}^*=-\mathcal{A}$, see \cite{monk2003finiteElement} for more detail.

\begin{remark}\label{rmk:surf_op}
	In order to ease the computations below, we present some properties of the operators in the PDE~\eqref{detmax}.
	
	Let us consider a scalar function $f \in C^\infty(\mathbb S^2)$, its identification as a radial vector field $f\hat{r}$, and the operator $curl_{\IS^2}$ acting on it.
	To compute the semigroup associated to the Hamiltonian operator \eqref{operatorME}, we need to consider
	\[
	\operatorname{curl}_{\mathbb S^2} \operatorname{curl}_{\mathbb S^2} f
	=\operatorname{curl}_{\mathbb S^2} \left(\nabla\times (f\hat{r})|_{T\IS^2}\right),
	\]
	where the radial coordinate $\hat r$ coincides with the outward unit normal $\hat n$.
    Pointwise, we then obtain the relation
	\[
	\operatorname{curl}_{\mathbb S^2} \left(\nabla\times (f\hat{r})|_{T\IS^2}\right)= \hat r\cdot\left[\nabla\times \nabla\times (f\hat r)\right],
	\]
    which follows by the fact that $\nabla \times (f \hat{r}) = -\hat{r}\times \nabla f$ is in $T {\IS^{2}}$.
	Using the standard $\R^3$ vector calculus, we then obtain the relation
	\[
	\hat r\cdot\left[\nabla\times \nabla\times (f\hat r)\right]=\hat r\cdot\left[\nabla(\nabla\cdot f\hat r)-\Delta(f\hat r)\right]=-\Deltalb f.
	\]

	Let us now consider a tangential vector field $X$. Using results from \cite{tangVecField} and references therein, we obtain that
	\[
	\operatorname{curl}_{\mathbb S^2} \operatorname{curl}_{\mathbb S^2} X
	=
	\nabla_{\IS^{2}}\operatorname{div}_{\IS^2} X-\Delta_{\IS^2}^{H} X,
	\]
	using the fact that $\hat n\times \nabla_{\IS^{2}}(\nabla_{\IS^2}\times X)=J\nabla_{\IS^{2}}(\nabla_{\IS^2}\times X)$, where $J$ is a rotation by $90$ degrees, and recalling that $\Delta_{\IS^2}^{H} X$ is the vector Hodge Laplacian. 
\end{remark}

Applying Stone's Theorem from \cite{Stone}, we obtain that $\mathcal A$ generates a unitary $\mathcal{C}_0$ semigroup
\begin{align}\label{ME_semi}
	e^{t\mathcal{A}}=\begin{pmatrix}
		\cos(t(\Delta^{H}_{\IS^2})^{\frac12}) & 
		(\Delta^{H}_{\IS^2})^{-\frac12}\sin(t(\Delta^{H}_{\IS^2})^{\frac12})\operatorname{curl}_{\IS^2}\\
		-(\Delta_{\IS^2})^{-\frac12}\sin(t(\Delta_{\IS^2})^{\frac12})\operatorname{curl}_{\IS^2} & 
		\cos(t(\Delta_{\IS^2})^{\frac12})
	\end{pmatrix},
\end{align}
where $\Delta^H_{\IS^2}$ denotes the vector Hodge Laplacian.
This result follows formally applying the definition of the operator exponential and separating the partial sums of even and odd powers of $\mathcal{A}$ using the fact that
$$
\mathcal{A}^{2k}=(-1)^k
\begin{pmatrix} 
	\operatorname{curl}_{\IS^2}^{2k} & 0\\ 
	0 & \operatorname{curl}_{\IS^2}^{2k}
\end{pmatrix},\quad 
\mathcal{A}^{2k+1}=(-1)^k
\begin{pmatrix} 0 & \operatorname{curl}_{\IS^2}^{2k+1}\\ -\operatorname{curl}_{\IS^2}^{2k+1} & 0
\end{pmatrix}.
$$ 

The above setting then allows us to define the solution of the deterministic Maxwell's equations on the sphere~\eqref{detmax} as
\begin{align}
	\label{eqn:ME}
	\begin{split}
	E(t)&=\cos(t(\Deltalb^H)^{1/2})E(0)+(\Deltalb^H)^{-1/2}\sin(t(\Deltalb^H)^{1/2})\operatorname{curl}_{\IS^2} H(0)\\
	H(t)&=-(\Deltalb)^{1/2}\sin(t(\Deltalb)^{1/2})\operatorname{curl}_{\IS^2} E(0)+\cos(t(\Deltalb)^{1/2})H(0).
	\end{split}
\end{align}
As a consequence, the Hamiltonian (or the total energy of the system)
\begin{equation}\label{sme-energy}
	\mathcal{E}(t)=\frac12\left(\norm{E(t)}^2+\norm{H(t)}^2\right),
\end{equation}
with the proper $L^2$--norm on the product space, is a conserved quantity for the PDE~\eqref{detmax}.

We are now ready to introduce the driving noise, $L(t,x)=(L_E(t,x),L_H(t,x)) \in L^2(\Omega, \mathcal{H}_0)$, as presented in Section~\ref{sect-setting},
into Maxwell's equations on the sphere~\eqref{detmax}. We thus define the stochastic Maxwell's equations on the sphere~\eqref{maxeq} as
the SPDE system
\begin{align}\label{eq:sme_surf}
	\begin{split}
		\dd E(t,x)&=\operatorname{curl}_{\IS^2}H(t,x)\dd t+\dd L_E(t,x),\\
		\dd H(t,x)&= -\operatorname{curl}_{\IS^2} E(t,x) \dd t+\dd L_H(t,x).
	\end{split}
\end{align}
Applying the semigroup~\eqref{ME_semi}, we obtain the mild solution
\begin{align}
	\label{eqn:ME_mildsol}
	&\begin{split}
		E(t)&=\cos(t(\Deltalb^H)^{1/2})E(0)+(\Deltalb^H)^{-1/2}\sin(t(\Deltalb^H)^{1/2})\operatorname{curl}_{\IS^2} H(0)\\
		&+\int_0^t\cos((t-s)(\Deltalb^H)^{1/2})\dd L_E(s)+\int_0^t(\Deltalb^H)^{-1/2}\sin((t-s)(\Deltalb^H)^{1/2})\operatorname{curl}_{\IS^2}\dd L_H(s),\\
		H(t)&=-(\Deltalb)^{1/2}\sin(t(\Deltalb)^{1/2})\operatorname{curl}_{\IS^2} E(0)+\cos(t(\Deltalb)^{1/2})H(0)\\
		&-\int_0^t(\Deltalb)^{1/2}\sin((t-s)(\Deltalb)^{1/2})\operatorname{curl}_{\IS^2}\dd L_E(s)+\int_0^t\cos((t-s)(\Deltalb)^{1/2})\dd L_H(s).
	\end{split}
\end{align}
The above formulation of stochastic Maxwell's equations on the sphere is used for theoretical purposes. For numerical implementation,
we spectrally decompose the SPDE~\eqref{eq:sme_surf} to obtain a system for the coefficients in the spectral expansions of $E$ and $H$. We now present some details on how to obtain this system.

Using the orthonormal basis~\eqref{3d-onb}, we can decompose the phase space \eqref{spaceME} into a direct sum of finite dimensional subspaces spanned by the vector spherical harmonics
$$
L^2(\IS^2,T\IS^2)=\overline{\bigoplus_{\ell \in \N_0,m=-\ell,\dots,\ell}\operatorname{span}\{\mathbf{\Phi_{\ell,m}},\mathbf{\Psi_{\ell, m}}\}},\ L^2(\IS^2)=\overline{\bigoplus_{\ell \in \N_0,m=-\ell,\dots,\ell}\operatorname{span}\{Y_{\ell,m}\}}.
$$
Using Remark~\ref{rmk:surf_op} and direct computations on the orthonormal basis, we can show that
\begin{align*}
	&\operatorname{div}_{\IS^2}\mathbf{\Phi}_{\ell,m}=\Deltalb Y_{\ell,m}=-\sqrt{\lambda_\ell}Y_{\ell,m},\quad \operatorname{div}_{\IS^2}\mathbf{\Psi}_{\ell,m}=0,\quad \operatorname{curl}_{\IS^2}\mathbf{\Phi}_{\ell,m}=0,\\
	&\operatorname{curl}_{\IS^2}\mathbf{\Psi}_{\ell,m}=-\operatorname{curl}_{\IS^2}\operatorname{curl}_{\IS^2}{Y}_{\ell,m}=-\sqrt{\lambda_\ell}Y_{\ell,m},\,\\
	&\operatorname{curl}_{\IS^2}Y_{\ell,m}=\nabla\times \hat r Y_{\ell,m}|_{T\IS^2}=-\sqrt{\lambda_\ell}\mathbf{\Psi}_{\ell,m}.
\end{align*}
In the TE mode, see above and \cite{MR3635830}, the magnetic field is in the normal direction and therefore interpreted as a scalar, while the electric field is on the tangent plane. Hence, we can assume the following expansion for the solution to the SPDE~\eqref{eq:sme_surf}:
\begin{align*}
	H(t,x)&=\sum_{\ell=0}^\infty\sum_{m=-\ell}^\ell h_{\ell,m}(t) Y_{\ell,m}(x)\\
	E(t,x)&=\sum_{\ell=0}^\infty\sum_{m=-\ell}^\ell \left(e^{(1)}_{\ell,m}(t) \mathbf{\Phi}_{\ell,m}(x)+ e^{(2)}_{\ell,m}(t) \mathbf{\Psi}_{\ell,m}(x)\right).
\end{align*}
Analogously, we can identify the noise with the formal trace class expansions
in the basis of the two spaces defining $\mathcal{H}_0$ as follows:
\begin{align*}
	L_H(t,x)&=\sum_{\ell=0}^\infty\sum_{m=-\ell}^\ell L_{\ell,m}(t) Y_{\ell,m}(x)\\
	L_E(t,x)&=\sum_{\ell=0}^\infty\sum_{m=-\ell}^\ell L^{(1)}_{\ell,m}(t) \mathbf{\Phi}_{\ell,m}(x)+ L^{(2)}_{\ell,m}(t) \mathbf{\Psi}_{\ell,m}(x).
\end{align*}
Using the orthogonality properties of the vector spherical harmonics and techniques from Remark~\ref{rmk:surf_op},
the SPDE~\eqref{eq:sme_surf} can be rewritten as the system of one-dimensional SDEs
\begin{align}\label{eq:max-fin}
	\begin{cases}
	\dd e^{(1)}_{\ell,m}&=\dd L^{(1)}_{\ell,m}(t),\\
	\dd e^{(2)}_{\ell,m}&=-\sqrt{\lambda_\ell} h_{\ell,m}\dd t+\dd L^{(2)}_{\ell,m}(t)\\
	\dd h_{\ell,m}&=\sqrt{\lambda_\ell}e^{(2)}_{\ell,m}\dd t+\dd L_{\ell,m}(t).
	\end{cases}
\end{align}
Finally, if one wants to preserve the physical properties of Maxwell's equations in the vacuum,
we obtain from $\operatorname{div}_{\IS^2}E=0$ that
$L^{(1)}_{\ell,m}=0$ for all $\ell \in \N_0, m=-\ell, \dots, \ell$ 
and from the conservation of flux of $H$, $\int_{\IS^{2}}H(t,x)\dd x=\text{Const}$, that 
$L_{0,0}=0$.
The final system of SDEs used for numerical implementation then reads
\begin{align}\label{eq:max_modes}
	\begin{cases}
		\dd e^{(2)}_{\ell,m}&=-\sqrt{\lambda_\ell}h_{\ell,m}\dd t+\dd L^{(2)}_{\ell,m}(t)\\
		\dd h_{\ell,m}&=\sqrt{\lambda_\ell}e^{(2)}_{\ell,m}\dd t+\dd L_{\ell,m}(t).
	\end{cases}
\end{align}
The mild solution to this system is then seen to be given by
\[
\begin{pmatrix}
	e^{(2)}_{\ell m}(t) \\
	h_{\ell m}(t)
\end{pmatrix}
=
e^{t\mathcal A_\ell}
\begin{pmatrix}
	e^{(2)}_{\ell m}(0) \\
	h_{\ell m}(0)
\end{pmatrix}
+
\int_0^t
e^{(t-s)\mathcal A_\ell}
\begin{pmatrix}
	\dd L^{(2)}_{\ell m} \\
	\dd L_{\ell m}
\end{pmatrix},
\]
where
\[
\mathcal A_\ell
=
\begin{pmatrix}
	0 & -\sqrt{\lambda_\ell} \\
	\sqrt{\lambda_\ell} & 0
\end{pmatrix},
\quad
e^{t\mathcal A_\ell}
=
\begin{pmatrix}
	\cos(\sqrt{\lambda_\ell} t) & -\sin(\sqrt{\lambda_\ell}t) \\
	\sin(\sqrt{\lambda_\ell} t) & \cos(\sqrt{\lambda_\ell} t)
\end{pmatrix}.
\]
Observe that the above has a similar structure as the stochastic wave equation.

We conclude this subsection by stating a trace formula for the expected energy of the exact mild solution~\eqref{eqn:ME_mildsol}
of the stochastic Maxwell's equations on the sphere~\eqref{eqn:ME}.
\begin{proposition}\label{prop:energy-exact_max}
	Consider the stochastic Maxwell's equations on the sphere~\eqref{eq:sme_surf} with initial values $(E(0),H(0))\in \mathcal{H}_0$.
	Assume that the expected value of the initial energy $\E\left[\mathcal{E}(0)\right]<\infty$ and that the $Q$-L\'evy process $L$
	is trace-class with covariance operators of $L_E$ and $L_H$ denoted by $Q_E$ and $Q_H$, respectively.
	The solution to this SPDE satisfies the trace formula for the energy:
	\begin{align*}
		\E\left[\mathcal{E}(t)\right]=\E\left[\mathcal{E}(0)\right]+\frac{t}2\left(\tr(Q_E)+\tr(Q_H)\right),
	\end{align*}
	for every time $t\geq0$.
\end{proposition}

\begin{appendixproof}[Proof of Proposition~\ref{prop:energy-exact_max}]
	Let us recall that we are assuming that $\mathbb E\left[L_H(t)\right]=0,\ \mathbb E\left[L_E(t)\right]=0$.
	We use the expression~\eqref{eqn:ME_mildsol} for the mild solution to  the considered SPDE and compute the expectation of the energy as follows
	\begin{align*}
		2&\E\left[\cE(t)\right] = \mathbb{E}\left[\norm{E(t)}^2+\norm{H(t)}^2\right]\\
		&= \E\left[ \norm{\cos(t(\Deltalb^H)^{1/2})E(0)}^2 + \norm{(\Deltalb)^{-1/2}\sin(t(\Deltalb)^{1/2})\operatorname{curl}_{\IS^2}H(0)}^2\right.\\
		& \left.\quad -2 \langle (\Deltalb)^{-1/2}\sin(t(\Deltalb)^{1/2})\operatorname{curl}_{\IS^2}H(0), \cos(t(\Deltalb^H)^{1/2})E(0)\rangle\right.\\
		& \left.\quad + \norm{ \int_0^t\cos((t-s)(\Deltalb^H)^{1/2})\dd L_E(s)}^2\right.\left.+ \norm{ \int_0^t(\Deltalb)^{-1/2}\sin((t-s)(\Deltalb)^{1/2})\operatorname{curl}_{\IS^2}\dd L_H(s)}^2\right.\\
		& \left.\quad + \norm{ (\Deltalb^H)^{-1/2}\sin(t(\Deltalb^H)^{1/2})\operatorname{curl}_{\IS^2}E(0)}^2 + \norm{\cos(t(\Deltalb)^{1/2})H(0)}^2  \right.\\
		&\left. \quad + 2 \langle (\Deltalb^H)^{-1/2} \sin(t(\Deltalb^H)^{1/2})\operatorname{curl}_{\IS^2}E(0), \cos(t(\Deltalb)^{1/2})H(0)\rangle  \right.\\
		&\left. \quad + \norm{\int_0^t(\Deltalb^H)^{-1/2}\sin((t-s)(\Deltalb^H)^{1/2})\operatorname{curl}_{\IS^2}\dd L_E(s)}^2
		+  \norm{\int_0^t\cos((t-s)(\Deltalb)^{1/2})\dd L_H(s)}^2\right].
\end{align*}
We notice that all the expressions containing $E(0)$ and $H(0)$ can be added up, giving the expectation of the initial energy $\mathbb E\left[\mathcal{E}(0)\right]$.
Using the polar decomposition theorem on Hilbert spaces \cite[Theorem~7.5, Theorem~7.20]{weidmann1980linear} and applying the definition of $\operatorname{curl}_{\IS^2}$ on the orthonormal basis of $\mathcal{H}_0$, we get
$$\operatorname{curl}_{\IS^2}E=\sqrt{-\operatorname{curl}_{\IS^2}\operatorname{curl}_{\IS^2}}E=(\Deltalb^H)^{1/2} E,\quad \operatorname{curl}_{\IS^2}H=\Deltalb^{1/2} H,$$ 
and the fact that the 
cosine and sine operators satisfy the trigonometric identities
$$\cos^2(t(\Deltalb^H)^{1/2})+\sin^2(t(\Deltalb^H)^{1/2})=I,\quad \cos^2(t(\Deltalb)^{1/2})+\sin^2(t(\Deltalb)^{1/2})=I.$$
We thus obtain the relation
\begin{align*}
		2&\E\left[\cE(t)\right] = 2\E\left[\cE(0)\right] + \E\left[\norm{ \int_0^t\cos((t-s)(\Deltalb^H)^{1/2})\dd L_E(s)}^2\right. \\
		&\quad+ \left.\norm{ \int_0^t(\Deltalb)^{-1/2}\sin((t-s)(\Deltalb)^{1/2})\operatorname{curl}_{\IS^2}\dd L_H(s)}^2\right]\\
		&\quad+ \E\left[ \norm{\int_0^t(\Deltalb^H)^{-1/2}\sin((t-s)(\Deltalb^H)^{1/2})\operatorname{curl}_{\IS^2}\dd L_E(s)}^2\right.\\
		&\quad+ \left.\norm{ \int_0^t\cos((t-s)(\Deltalb)^{1/2})\dd L_H(s)}^2\right].
	\end{align*}
	The result then follows using an It\^o isometry and properties of the Hilbert--Schmidt norms.
\end{appendixproof}

\subsection{Spectral discretization}

A spectral discretization of the stochastic Maxwell's equations on the sphere~\eqref{eq:sme_surf} is obtained by applying the projection operator $\cP_\kappa$ defined in Section~\ref{sect-setting}. For a truncation index $\kappa$, we then get the following mild solution
\begin{align}
	\label{eqn:ME_mildsol-truncs}
	\begin{split}
		E&^\kappa(t)=\cos(t(\Deltalb^{H,\kappa})^{1/2})E^\kappa(0)+(\Deltalb^\kappa)^{-1/2}\sin(t(\Deltalb^\kappa)^{1/2})\operatorname{curl}^\kappa_{\IS^2} H^\kappa(0)\\
		&+\int_0^t\cos((t-s)(\Deltalb^{H,\kappa})^{1/2})\cP_\kappa\dd L_E(s)+\int_0^t(\Deltalb^\kappa)^{-1/2}\sin((t-s)(\Deltalb^\kappa)^{1/2})\operatorname{curl}^\kappa_{\IS^2} \cP_\kappa\dd L_H(s),\\
		H&^\kappa(t)=-(\Deltalb^{H,\kappa})^{1/2}\sin(t(\Deltalb^{H,\kappa})^{1/2})\operatorname{curl}^\kappa_{\IS^2} E^\kappa(0)+\cos(t(\Deltalb^\kappa)^{1/2})H^\kappa(0)\\
		&-\int_0^t(\Deltalb^{H,\kappa})^{1/2}\sin((t-s)(\Deltalb^{H,\kappa})^{1/2})\operatorname{curl}^\kappa_{\IS^2} \cP_\kappa\dd L_E(s)+\int_0^t\cos((t-s)(\Deltalb^\kappa)^{1/2})\cP_\kappa\dd L_H(s),
	\end{split}
\end{align}
where $\Deltalb^{H, \kappa} = \cP_\kappa \Deltalb^H$ and $\operatorname{curl}^\kappa_{\IS^{2}} = \cP_\kappa \operatorname{curl}_{\IS^{2}}$.
It is not a surprise that the above spectral discretization has the same long-time behavior as the exact solution to the considered stochastic Maxwell's equations given in Proposition~\ref{prop:energy-exact_max}. This is the subject of the next result.
\begin{proposition}\label{prop:energy-spectral_max}
	Consider the stochastic Maxwell's equations on the sphere~\eqref{eq:sme_surf} with initial values $(E(0),H(0))\in \mathcal H_0$.
	Assume that the expected value of the initial energy $\E\left[\mathcal{E}(0)\right]<\infty$ and that the $Q$-L\'evy process $L$
	is trace-class with covariance operators denoted by $Q_E$ and $Q_H$. Let ($E^\kappa$,$H^\kappa$) denote the spectral approximation obtained from equation~\eqref{eqn:ME_mildsol-truncs}. The spectral approximation satisfies the trace formula for the energy:
	\begin{align*}
		\E\left[\mathcal{E}^\kappa(t)\right]=\E\left[\mathcal{E}^\kappa(0)\right]+\frac{t}2\left(\tr\left(\cP_\kappa Q_E\cP_\kappa\right)+\tr\left(\cP_\kappa Q_H\cP_\kappa\right)\right),
	\end{align*}
	for every time $t\geq0$.
\end{proposition}

\begin{appendixproof}[Proof of Proposition~\ref{prop:energy-spectral_max}]
	The proof of this result follows analogously to the proof of Proposition~\ref{prop:energy-exact_max}, using similar ideas as in the proof of Proposition~\ref{prop:energy-spectral} for the stochastic wave equation on the sphere thanks to Remark~\ref{rmk:dblw}.
\end{appendixproof}

\subsection{Temporal discretizations}

We can now investigate the long-time behavior, with respect to the energy, of the three time integrators as it was done
for the stochastic wave and Schr\"odinger equations on the sphere in the previous sections.

Let us first observe that deterministic Maxwell's equations on the sphere~\eqref{detmax} can be written as a system of wave equations
on the sphere.
\begin{remark}\label{rmk:dblw}
	Under the setting of Remark~\ref{rmk:surf_op} and the selection of the orientation, applying the time derivative to
	Maxwell's equations~\eqref{detmax}, we obtain the equivalent system of wave equations
	\begin{align*}
		\partial_{tt}E(t)+\nabla_{\IS^2}\operatorname{div}_{\IS^2} E-\Delta^{H}_{\IS^2} E(t)&=0\\
		\partial_{tt}H(t)-\Deltalb H(t)&=0.
	\end{align*}
	In addition, when considering the case of Maxwell's equations in vacuum, one has $\operatorname{div}_{\IS^2}E=0$, and the above
	system reduces to the system of wave equations on the sphere
	\begin{align*}
		\partial_{tt}E(t)-\Delta^{H}_{\IS^2} E(t)&=0\\
		\partial_{tt}H(t)-\Deltalb H(t)&=0.
	\end{align*}
\end{remark}
With this remark at hand, the proofs for the time discretizations of the stochastic Maxwell's equations on the sphere~\eqref{eq:sme_surf},
can be adapted from the the ones given for the stochastic wave equation on the sphere in Section~\ref{sect-swe}.

First we consider the forward Euler--Maruyama scheme~\eqref{eEM} which, when applied to the truncated stochastic Maxwell's equations on the sphere~\eqref{eqn:ME_mildsol-truncs}, reads
\begin{align}\label{eq:max_EM}
	\begin{split}
		E^\kappa_{n+1}&=E^\kappa_{n}+\tau\operatorname{curl}^\kappa_{\IS^2}H^\kappa_n+\cP_\kappa \Delta L^E_n\\
		H^\kappa_{n+1}&=H^\kappa_{n}-\tau\operatorname{curl}^\kappa_{\IS^2}E^\kappa_n+\cP_\kappa \Delta L^H_n.
	\end{split}
\end{align}
This time integrator does not have the correct behavior, with respect to the energy, as seen in the following proposition.
\begin{proposition}\label{prop:energy-em_max}
	Consider the stochastic Maxwell's equations on the sphere~\eqref{eq:sme_surf} with initial values $(E(0),H(0))\in \mathcal H_0.$
	Assume that the expected value of the initial energy $\E\left[\mathcal{E}(0)\right]<\infty$ and that the $Q$-L\'evy process $L$
	is trace-class with covariance operators of $L_E$ and $L_H$ denoted by $Q_E$ and $Q_H$, respectively.
	Furthermore, assume that the expected value of the initial energy satisfies $0<\E\left[\mathcal{E}^\kappa_0-\cE_0^{0,0}\right]$.
	Let $(E_{n}^\kappa, H_{n}^\kappa)$ denote the fully-discrete approximation given by the forward Euler--Maruyama scheme~\eqref{eq:max_EM}.
	The energy of this fully-discrete approximation grows exponentially in time:
	\begin{equation*}
		\E\left[\mathcal{E}^\kappa_n\right] \geq \exp\left(\frac12\tau t_n\right)\E\left[\mathcal{E}^\kappa_0 - \mathcal{E}_0^{0,0}\right]
	\end{equation*}
	for every discrete times $t_n=n\tau$ with integers $n\geq0$.
\end{proposition}

\begin{appendixproof}[Proof of Proposition~\ref{prop:energy-em_max}]
	The result follows as the proof of Proposition~\ref{prop:energy-em} thanks to Remark~\ref{rmk:dblw}.
\end{appendixproof}

The next step is to investigate the long-time behavior, with respect to the energy, of the backward Euler--Maruyama scheme~\eqref{bEM} when applied to the spatial truncation of the stochastic Maxwell's equations on the sphere~\eqref{eqn:ME_mildsol-truncs}. This numerical scheme is given by the following implicit relation
\begin{align}\label{eq:max_BEM}
	\begin{split}
		E^\kappa_{n+1}&=E^\kappa_{n}+\tau\operatorname{curl}^\kappa_{\IS^2}H^\kappa_{n+1}+\cP_\kappa \Delta L^E_n\\
		H^\kappa_{n+1}&=H^\kappa_{n}-\tau\operatorname{curl}^\kappa_{\IS^2}E^\kappa_{n+1}+\cP_\kappa \Delta L^H_n.
	\end{split}
\end{align}
The behavior of the backward Euler--Maruyama scheme is also not correct and this is made precise by the following result.
\begin{proposition}\label{prop:energy-bem_max}
	Consider the stochastic Maxwell's equations on the sphere~\eqref{eq:sme_surf} with initial values $(E(0),H(0))\in \mathcal H_0$.
	Assume that the expected value of the initial energy $\E\left[\mathcal{E}(0)\right]<\infty$ and that the $Q$-L\'evy process $L$
	is trace-class with covariance operators of $L_E$ and $L_H$ denoted by $Q_E$ and $Q_H$, respectively.
	Furthermore, assume that the expected value of the initial energy satisfies $0<\E\left[\mathcal{E}^\kappa_0\right]$.
	Let $(E_{n}^\kappa, H_{n}^\kappa)$ denote the fully-discrete approximation given by the backward Euler--Maruyama scheme~\eqref{eq:max_BEM}.
	The energy of this fully-discrete approximation grows at a slower rate than in the exact solution to SPDE~\eqref{eq:sme_surf}:
	\begin{equation*}
		\E\left[\cE^\kappa_n\right] <
		\E\left[\cE^\kappa_0\right]+ \frac{1}{2\tau}\left(\tr\left(\cP_\kappa Q_H\cP_\kappa\right)+\tr\left(\cP_\kappa Q_E\cP_\kappa\right)\right) + \frac {t_n} {2} \left(v^H_{0,0}+v^E_{0,0}\right)
	\end{equation*}
	for every discrete times $t_n=n\tau$ with integers $n\geq0$. Here, we denote $v^E_{0,0}=\Var(L^E_{0,0}(1))$ and $v^H_{0,0}=\Var\left(L^H_{0,0}(1)\right)$. Moreover, we have the relation
	\begin{equation}\label{eqn:SME_energy_BEM_limit}
		\lim_{n\to\infty}\frac{\E\left[\mathcal{E}^\kappa_n\right]}{t_n}=\frac 1 2 \left(v^E_{0,0}+v^H_{0,0}\right).
	\end{equation}
\end{proposition}
\begin{appendixproof}[Proof of Proposition~\ref{prop:energy-bem_max}]
    The result follows analogously to the proof of Proposition~\ref{prop:energy-bem} using the identities from Remark~\ref{rmk:surf_op}.
\end{appendixproof}

Looking back at the results in Sections~\ref{sect-swe}~and~\ref{sect-sch}, we expect that the exponential Euler scheme~\eqref{expEM} will have the correct long-time behavior with respect to the energy of the stochastic Maxwell's equations on the sphere~\eqref{eq:sme_surf}.
In this context, this time integrator reads
\begin{align}\label{eqn:SME_trigo}
	\begin{split}
	E_{n+1}^\kappa&=\cos(\tau(\Deltalb^\kappa)^{1/2})E_{n}^\kappa+(\Deltalb^\kappa)^{-1/2}\sin(\tau(\Deltalb^\kappa)^{1/2})\operatorname{curl}_{\IS^2}^\kappa H_{n}^\kappa\\
	&\quad +\cos(\tau(\Deltalb^\kappa)^{1/2})\cP_\kappa\Delta L^E_n+(\Deltalb^\kappa)^{-1/2}\sin(\tau(\Deltalb^\kappa)^{1/2})\operatorname{curl}_{\IS^2}^\kappa \cP_\kappa\Delta L^H_n,\\
	H_{n+1}^\kappa&=-(\Deltalb^\kappa)^{1/2}\sin(\tau(\Deltalb^\kappa)^{1/2})\operatorname{curl}_{\IS^2}^\kappa E_{n}^\kappa+\cos(\tau(\Deltalb^\kappa)^{1/2})H_{n}^\kappa\\
	&\quad -(\Deltalb^\kappa)^{-1/2}\sin(\tau(\Deltalb^\kappa)^{1/2})\operatorname{curl}_{\IS^2}^\kappa \cP_\kappa\Delta L^E_n+\cos(\tau(\Deltalb^\kappa)^{1/2})\cP_\kappa\Delta L^H_n.
	\end{split}
\end{align}
The next proposition states that, indeed, the exponential Euler scheme has the same behavior as the spatially semi-discrete solution, see Proposition~\ref{prop:energy-spectral_max}.
\begin{proposition}\label{prop:energy-trigo_max}
	Consider the stochastic Maxwell's equations on the sphere~\eqref{eq:sme_surf} with initial values $(E(0),H(0))\in \mathcal H_0$.
	Assume that the expected value of the initial energy $\E\left[\mathcal{E}(0)\right]<\infty$ and that the $Q$-L\'evy process $L$
	is trace-class with covariance operators of $L_E$ and $L_H$ denoted by $Q_E$ and $Q_H$, respectively. Let $(E_{n}^\kappa, H_{n}^\kappa)$ denote the fully-discrete approximation given by the stochastic exponential integrator~\eqref{eqn:SME_trigo}.
	This fully-discrete approximation satisfies the trace formula for the energy:
	\begin{align*}
		\E\left[\mathcal{E}^\kappa_n\right]=\E\left[\mathcal{E}^\kappa_0\right]+\frac{t_n}2\left(\tr\left(\cP_\kappa Q_E\cP_\kappa\right)+\tr\left(\cP_\kappa Q_H\cP_\kappa\right)\right),
	\end{align*}
	for every discrete times $t_n=n\tau$ with integers $n\geq0$.
\end{proposition}

\begin{appendixproof}[Proof of Proposition~\ref{prop:energy-trigo_max}]
	The proof follows the main steps of the proof of Proposition~\ref{prop:energy-exact_max}, see also the
	proof of Proposition~\ref{prop:energy-trigo} for the stochastic wave equation on the sphere.
\end{appendixproof}

\subsection{Numerical experiments}\label{max-num}
We conclude this section with numerical experiments in order to support the above theoretical
results on the long-time behavior of the expected energy
of the stochastic Maxwell's equations on the sphere~\eqref{eq:sme_surf}.

We simulate $M=10000$ Monte Carlo samples of the numerical solutions to this SPDE up to final times $T=3$ and $T=100$ using
the stochastic exponential Euler scheme~\eqref{eqn:SME_trigo}, the forward Euler--Maruyama scheme~\eqref{eq:max_EM} and the backward Euler--Maruyama scheme~\eqref{eq:max_BEM}. The number of time steps is taken to be $N=300$.
The truncation parameter for our spatial discretization is taken to be $\kappa = 2^5$.
As initial values, we choose the Gaussian random fields
$$
E_0 = \sum_{\ell=0}^{\kappa} \sum_{m=-\ell}^{\ell} \ell^{-\gamma} u_{\ell , m} Y_{\ell , m}, \quad H_0 = \sum_{\ell=0}^{\kappa} \sum_{m=-\ell}^{\ell} \ell^{-\gamma+1} z_{\ell , m} Y_{\ell , m},
$$
where $\gamma=3+10^{-5}$ (setting $0^{-x}=1$ for $x>0$), and where $u_{\ell , m}$ and $z_{\ell , m}$ are iid standard Gaussian random variables.
For the L\' evy noise $L^E$,
we choose $L^E=\sum_{\ell=0}^{\kappa} \sum_{m=-\ell}^{\ell} A_\ell L^E_{\ell, m} Y_{\ell , m},$ for independent Lévy processes $L^E_{\ell,m}$,
given by $L^E_{\ell, m}(t) = \left(W^E_{\ell, m}(t) + P^E_{\ell, m}(t) - t\right) / \sqrt{2}$
for a standard Brownian motion $W^E_{\ell, m}$ and independent Poisson process $P^E_{\ell, m}$.
The $Q$-L\'evy process $L^H$ is independent of and identically distributed to $L^E$.
This choice implies, by \cite{marinucciRandomFieldsSphere2011a}, that the covariance operator $Q^\kappa$ of the truncated noises $\cP_\kappa L^E$ and $\cP_\kappa L^H$ is the covariance operator of an isotropic random field.
Furthermore, we choose $A_0 = 1$ and $A_\ell = \ell^{-4}$ for $\ell = 1, \dots, \kappa$. 

The results of these simulations are shown in Figure~\ref{fig:SME_energies}. For short and long time intervals, 
the correct behavior of the expected energy in the stochastic exponential Euler method~\eqref{eqn:SME_trigo} can be observed, illustrating Proposition~\ref{prop:energy-trigo_max}.
In contrast, as shown in Propositions~\ref{prop:energy-em_max}~and~\ref{prop:energy-bem_max}, the forward Euler--Maruyama scheme~\eqref{eq:max_EM}, exhibits an exponential increase in energy, while the backward Euler--Maruyama method~\eqref{eq:max_BEM}
shows an energy that grows more slowly than that of the exact solution.

\begin{figure}[h!]
    \centering
    \begin{subfigure}[t]{0.5\textwidth}
        \centering
        \includegraphics[width=\textwidth]{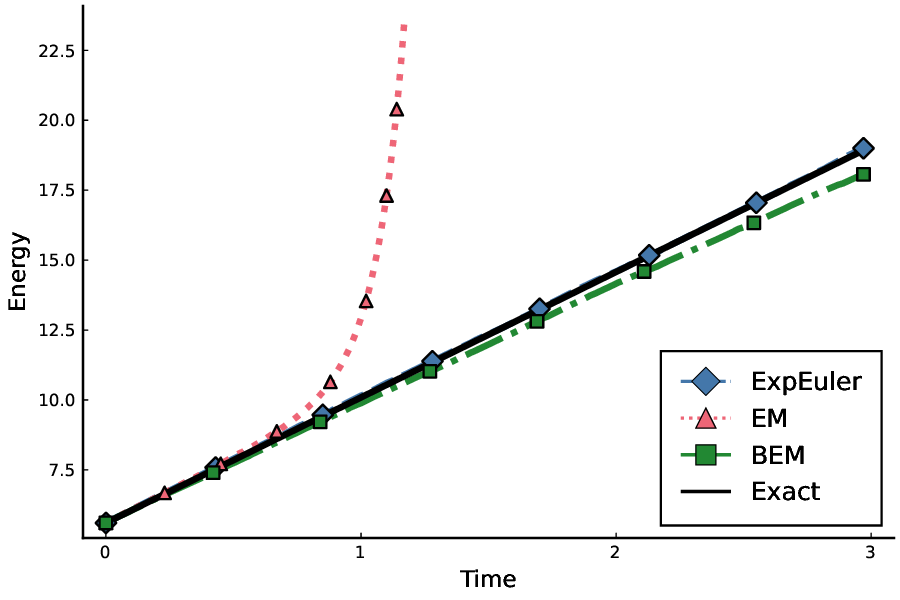}
        \caption{$T=3$}
        \label{fig:SME_energy}
    \end{subfigure}%
    \begin{subfigure}[t]{0.5\textwidth}
        \centering
        \includegraphics[width=\textwidth]{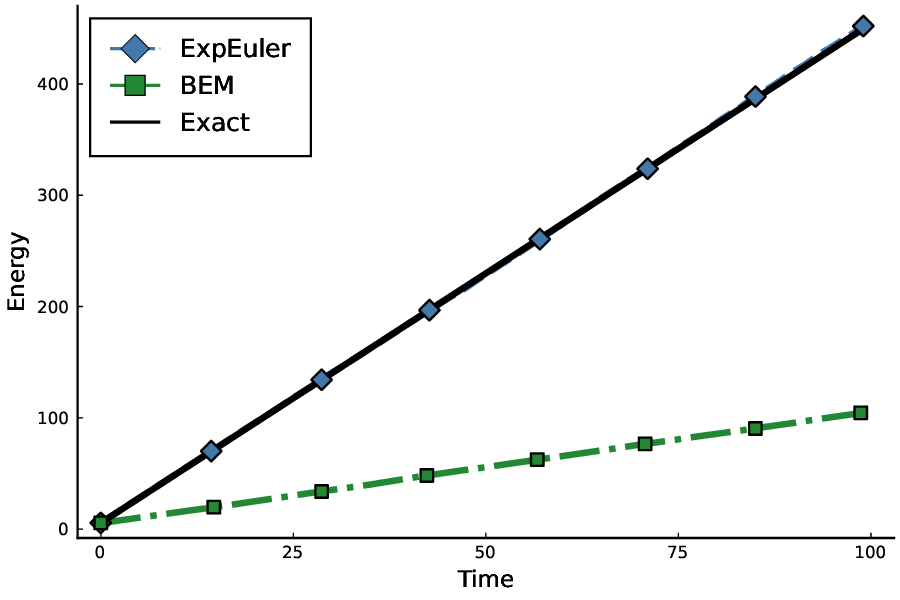}
        \caption{$T=100$}
        \label{fig:SME_energy_longtime}
    \end{subfigure}
    \caption{Expected energy of the stochastic Maxwell's equations on the sphere~\eqref{eq:sme_surf}: Exact solution and stochastic exponential Euler scheme~\eqref{eqn:SME_trigo} ({\upshape ExpEuler})
    the forward Euler--Maruyama scheme~\eqref{eqn:SSE_EM} ({\upshape EM}), and the backward Euler--Maruyama scheme~\eqref{eqn:BEM} ({\upshape BEM}).
    }
    \label{fig:SME_energies}
\end{figure}

\section{Acknowledgements}
The work of DC was partially supported by the Swedish Research Council (VR) (projects nr. $2018-04443$
and $2024-04536$). The work of DC, BM, and AP was partially supported by the European Union (ERC, StochMan, 101088589, PI A. Lang).
Views and opinions expressed are however those of the author(s) only and do not necessarily reflect those
of the European Union or the European Research Council. Neither the European Union nor the granting authority can be held responsible for them.
The computations were performed on resources provided by
the National Academic Infrastructure for Supercomputing in Sweden (NAISS) at Vera, Chalmers e-Commons
at Chalmers University of Technology and partially funded by the Swedish Research Council
through grant agreement no. 2022-06725.

\appendix\label{appdx}

\bibliographystyle{plain}
\bibliography{labib}

@article {MR2333534,
    AUTHOR = {Hong, J. and Scherer, R. and Wang, L.},
     TITLE = {Predictor-corrector methods for a linear stochastic oscillator
              with additive noise},
   JOURNAL = {Math. Comput. Modelling},
  FJOURNAL = {Mathematical and Computer Modelling},
    VOLUME = {46},
      YEAR = {2007},
    NUMBER = {5-6},
     PAGES = {738--764},
      ISSN = {0895-7177},
   MRCLASS = {65L06 (37M15 65C30 65P10 70H05 70L05)},
  MRNUMBER = {2333534},
MRREVIEWER = {Inmaculada\ Higueras},
       DOI = {10.1016/j.mcm.2006.12.009},
       URL = {https://doi.org/10.1016/j.mcm.2006.12.009},
}

@article {MR2370085,
    AUTHOR = {Brze\'zniak, Z. and Ondrej\'at, M.},
     TITLE = {Strong solutions to stochastic wave equations with values in
              {R}iemannian manifolds},
   JOURNAL = {J. Funct. Anal.},
  FJOURNAL = {Journal of Functional Analysis},
    VOLUME = {253},
      YEAR = {2007},
    NUMBER = {2},
     PAGES = {449--481},
      ISSN = {0022-1236,1096-0783},
   MRCLASS = {60H15 (35R60 58J65)},
  MRNUMBER = {2370085},
MRREVIEWER = {Robert\ C.\ Dalang},
       DOI = {10.1016/j.jfa.2007.03.034},
       URL = {https://doi.org/10.1016/j.jfa.2007.03.034},
}

@article {MR4161976,
    AUTHOR = {Brze\'zniak, Z. and Hussain, J.},
     TITLE = {Global solution of nonlinear stochastic heat equation with
              solutions in a {H}ilbert manifold},
   JOURNAL = {Stoch. Dyn.},
  FJOURNAL = {Stochastics and Dynamics},
    VOLUME = {20},
      YEAR = {2020},
    NUMBER = {6},
     PAGES = {2040012, 29},
      ISSN = {0219-4937,1793-6799},
   MRCLASS = {60H15 (35K55 35R60 58J35 58J65 60J60 60J65)},
  MRNUMBER = {4161976},
       DOI = {10.1142/S0219493720400122},
       URL = {https://doi.org/10.1142/S0219493720400122},
}

@article{MR4942808,
    AUTHOR = {Brze\'zniak, Z. and Cerrai, S.},
     TITLE = {Stochastic wave equations with constraints: well-posedness and
              {S}moluchowski-{K}ramers diffusion approximation},
   JOURNAL = {Comm. Math. Phys.},
  FJOURNAL = {Communications in Mathematical Physics},
    VOLUME = {406},
      YEAR = {2025},
    NUMBER = {9},
     PAGES = {Paper No. 223, 59},
      ISSN = {0010-3616,1432-0916},
   MRCLASS = {60H15},
  MRNUMBER = {4942808},
       DOI = {10.1007/s00220-025-05397-0},
       URL = {https://doi.org/10.1007/s00220-025-05397-0},
}

@misc{cohen2026fullydiscreteapproximationsemilinear,
      title={Fully discrete approximation of the semilinear stochastic wave equation on the sphere}, 
      author={D. Cohen and S. Di Giovacchino and A. Lang},
      year={2026},
      eprint={2602.00556},
      archivePrefix={arXiv},
      primaryClass={math.NA},
      url={https://arxiv.org/abs/2602.00556}, 
}

@book{Elworthy_1982, place={Cambridge}, series={London Mathematical Society Lecture Note Series}, title={Stochastic Differential Equations on Manifolds}, publisher={Cambridge University Press}, author={Elworthy, K. D.}, year={1982}, collection={London Mathematical Society Lecture Note Series}}

@article{tangVecField,
	author = {Brandt, C. and Scandolo, L. and Eisemann, E. and Hildebrandt, K.},
	title = {Spectral Processing of Tangential Vector Fields},
	journal = {Computer Graphics Forum},
	volume = {36},
	number = {6},
	pages = {338-353},
	doi = {https://doi.org/10.1111/cgf.12942},
	url = {https://onlinelibrary.wiley.com/doi/abs/10.1111/cgf.12942},
	eprint = {https://onlinelibrary.wiley.com/doi/pdf/10.1111/cgf.12942},
	year = {2017}
}

@article{Stone,
	ISSN = {0003486X, 19398980},
	URL = {http://www.jstor.org/stable/1968538},
	author = {M. H. Stone},
	journal = {Annals of Mathematics},
	number = {3},
	pages = {643--648},
	publisher = {[Annals of Mathematics, Trustees of Princeton University on Behalf of the Annals of Mathematics, Mathematics Department, Princeton University]},
	title = {On {O}ne-{P}arameter {U}nitary {G}roups in {H}ilbert {S}pace},
	urldate = {2026-01-20},
	volume = {33},
	year = {1932}
}

@article{CabralLobo2017,
	author  = {Cabral, F. and Lobo, F. S. N.},
	title   = {Electrodynamics and {S}pacetime {G}eometry: Foundations},
	journal = {Foundations of Physics},
	volume  = {47},
	number  = {2},
	pages   = {208--228},
	year    = {2017},
	doi     = {10.1007/s10701-016-0051-6},
	url     = {https://doi.org/10.1007/s10701-016-0051-6}
}

@book{flanders1963differential,
	title={Differential forms with applications to the physical sciences},
	author={Flanders, H.},
	volume={11},
	year={1963},
	publisher={Courier Corporation}
}

@article{shi70,
	author = {Ohkuro, S.},
	title = {Differential {F}orms and {M}axwell's {F}ield: {A}n {A}pplication of {H}armonic {I}ntegrals},
	journal = {Journal of Mathematical Physics},
	volume = {11},
	number = {6},
	pages = {2005-2012},
	year = {1970},
	month = {06},
	issn = {0022-2488},
	doi = {10.1063/1.1665359},
	url = {https://doi.org/10.1063/1.1665359},
	eprint = {https://pubs.aip.org/aip/jmp/article-pdf/11/6/2005/19200215/2005_1_online.pdf},
}

@Inbook{Milani1988,
	author="Milani, A.
	and Picard, R.",
	editor="Hildebrandt, Stefan
	and Leis, Rolf",
	title="Decomposition theorems and their application to non-linear electro- and magneto-static boundary value problems",
	bookTitle="Partial Differential Equations and Calculus of Variations",
	year="1988",
	publisher="Springer Berlin Heidelberg",
	address="Berlin, Heidelberg",
	pages="317--340",
	isbn="978-3-540-46024-4",
	doi="10.1007/BFb0082873",
	url="https://doi.org/10.1007/BFb0082873"
}

@article{Hiptmair_2002, title={Finite elements in computational electromagnetism}, volume={11}, DOI={10.1017/S0962492902000041}, journal={Acta Numerica}, author={Hiptmair, R.}, year={2002}, pages={237–339}}

@article{WECK1974410,
	title = {Maxwell's boundary value problem on Riemannian manifolds with nonsmooth boundaries},
	journal = {Journal of Mathematical Analysis and Applications},
	volume = {46},
	number = {2},
	pages = {410-437},
	year = {1974},
	issn = {0022-247X},
	doi = {https://doi.org/10.1016/0022-247X(74)90250-9},
	url = {https://www.sciencedirect.com/science/article/pii/0022247X74902509},
	author = {N. Weck}
}

@article{FANG2025855,
	title = {Dispersive estimates for {M}axwell's equations in the exterior of a sphere},
	journal = {Journal of Differential Equations},
	volume = {415},
	pages = {855-885},
	year = {2025},
	issn = {0022-0396},
	doi = {https://doi.org/10.1016/j.jde.2024.10.024},
	url = {https://www.sciencedirect.com/science/article/pii/S002203962400679X},
	author = {Y. Fang and A. Waters}
}

@article{Munshi2022Complex,
	author = {Munshi, S. and Yang, R.},
	title = {Complex solutions to {M}axwell's equations},
	journal = {Complex Analysis and its Synergies},
	volume = {8},
	number = {1},
	year = {2022},
	doi = {10.1007/s40627-022-00091-6},
	publisher = {Springer}
}

@Article{math14010017,
AUTHOR = {Wang, Y. and Lang, Z. and Yin, X. and Zhao, Z.},
TITLE = {Exponentially Fitted Midpoint Scheme for a Stochastic Oscillator},
JOURNAL = {Mathematics},
VOLUME = {14},
YEAR = {2026},
NUMBER = {1},
ARTICLE-NUMBER = {17},
URL = {https://www.mdpi.com/2227-7390/14/1/17},
ISSN = {2227-7390},
DOI = {10.3390/math14010017}
}

@article {MR3702519,
    AUTHOR = {Yin, X. and Zhang, C. and Liu, Y.},
     TITLE = {Exponentially fitted trapezoidal scheme for a stochastic
              oscillator},
   JOURNAL = {J. Comput. Math.},
  FJOURNAL = {Journal of Computational Mathematics},
    VOLUME = {35},
      YEAR = {2017},
    NUMBER = {6},
     PAGES = {801--813},
      ISSN = {0254-9409,1991-7139},
   MRCLASS = {65C30 (34F05 60H10)},
  MRNUMBER = {3702519},
MRREVIEWER = {Latifa\ Debbi},
       DOI = {10.4208/jcm.1612-m2016-0677},
       URL = {https://doi.org/10.4208/jcm.1612-m2016-0677},
}

@article {MR3579605,
    AUTHOR = {Wang, P. and Hong, J. and Xu, D.},
     TITLE = {Construction of symplectic {R}unge-{K}utta methods for
              stochastic {H}amiltonian systems},
   JOURNAL = {Commun. Comput. Phys.},
  FJOURNAL = {Communications in Computational Physics},
    VOLUME = {21},
      YEAR = {2017},
    NUMBER = {1},
     PAGES = {237--270},
      ISSN = {1815-2406,1991-7120},
   MRCLASS = {65P10 (60H35 65C30)},
  MRNUMBER = {3579605},
       DOI = {10.4208/cicp.261014.230616a},
       URL = {https://doi.org/10.4208/cicp.261014.230616a},
}

@article {MR2922170,
    AUTHOR = {Cohen, D.},
     TITLE = {On the numerical discretisation of stochastic oscillators},
   JOURNAL = {Math. Comput. Simulation},
  FJOURNAL = {Mathematics and Computers in Simulation},
    VOLUME = {82},
      YEAR = {2012},
    NUMBER = {8},
     PAGES = {1478--1495},
      ISSN = {0378-4754,1872-7166},
   MRCLASS = {65C30 (34F05 60H10 60H35)},
  MRNUMBER = {2922170},
MRREVIEWER = {Renato\ G. C. Spigler},
       DOI = {10.1016/j.matcom.2012.02.004},
       URL = {https://doi.org/10.1016/j.matcom.2012.02.004},
}

@article {MR3831957,
    AUTHOR = {D'Ambrosio, R. and Moccaldi, M. and Paternoster,
              B.},
     TITLE = {Numerical preservation of long-term dynamics by stochastic
              two-step methods},
   JOURNAL = {Discrete Contin. Dyn. Syst. Ser. B},
  FJOURNAL = {Discrete and Continuous Dynamical Systems. Series B. A Journal
              Bridging Mathematics and Sciences},
    VOLUME = {23},
      YEAR = {2018},
    NUMBER = {7},
     PAGES = {2763--2773},
      ISSN = {1531-3492,1553-524X},
   MRCLASS = {65C30 (34F05)},
  MRNUMBER = {3831957},
MRREVIEWER = {Haziem\ Mohammad\ Hazaimeh},
       DOI = {10.3934/dcdsb.2018105},
       URL = {https://doi.org/10.3934/dcdsb.2018105},
}

@article{akrivisFiniteDifferenceDiscretization1993,
  title = {Finite Difference Discretization of the Cubic {{Schr\"odinger}} Equation},
  author = {Akrivis, G. D.},
  year = 1993,
  month = jan,
  journal = {IMA Journal of Numerical Analysis},
  volume = {13},
  number = {1},
  pages = {115--124},
  issn = {0272-4979},
  doi = {10.1093/imanum/13.1.115},
  urldate = {2025-11-26}
}

@article{antoineComputationalMethodsDynamics2013,
  title = {Computational Methods for the Dynamics of the Nonlinear {{Schr\"odinger}}/{{Gross}}--{{Pitaevskii}} Equations},
  author = {Antoine, X. and Bao, W. and Besse, C.},
  year = 2013,
  month = dec,
  journal = {Computer Physics Communications},
  volume = {184},
  number = {12},
  pages = {2621--2633},
  issn = {0010-4655},
  doi = {10.1016/j.cpc.2013.07.012},
  urldate = {2025-11-26}
}

@article{besseEnergypreservingMethodsNonlinear2021,
  title = {Energy-Preserving Methods for Nonlinear {{Schr\"odinger}} Equations},
  author = {Besse, C. and Descombes, S. and Dujardin, G. and {Lacroix-Violet}, I.},
  year = 2021,
  month = jan,
  journal = {IMA Journal of Numerical Analysis},
  volume = {41},
  number = {1},
  pages = {618--653},
  issn = {0272-4979},
  doi = {10.1093/imanum/drz067},
  urldate = {2025-11-26}
}

@article{besseRelaxationSchemeNonlinear2004,
  title = {A {{Relaxation Scheme}} for the {{Nonlinear Schr\"odinger Equation}}},
  author = {Besse, C.},
  year = 2004,
  month = jan,
  journal = {SIAM Journal on Numerical Analysis},
  volume = {42},
  number = {3},
  pages = {934--952},
  publisher = {{Society for Industrial and Applied Mathematics}},
  issn = {0036-1429},
  doi = {10.1137/S0036142901396521},
  urldate = {2025-11-26}
}

@book{carlesSemiclassicalAnalysisNonlinear2008,
  title = {Semi-Classical {{Analysis For Nonlinear Schr\"odinger Equations}}},
  author = {Carles, R.},
  year = 2008,
  publisher = {World Scientific Publishing Company},
  address = {Singapore},
  urldate = {2025-11-24},
  isbn = {978-981-279-313-3}
}

@article{cohenOnestageExponentialIntegrators2012,
  title = {One-Stage Exponential Integrators for Nonlinear {{Schr\"odinger}} Equations over Long Times},
  author = {Cohen, D. and Gauckler, L.},
  year = 2012,
  month = dec,
  journal = {BIT Numerical Mathematics},
  volume = {52},
  number = {4},
  pages = {877--903},
  issn = {1572-9125},
  doi = {10.1007/s10543-012-0385-1},
  urldate = {2025-11-26},
  langid = {english}
}

@book{lubichQuantumClassicalMolecular2008,
  title = {From {{Quantum}} to {{Classical Molecular Dynamics}}: {{Reduced Models}} and {{Numerical Analysis}}},
  shorttitle = {From {{Quantum}} to {{Classical Molecular Dynamics}}},
  author = {Lubich, C.},
  year = 2008,
  month = sep,
  series = {Zurich {{Lectures}} in {{Advanced Mathematics}}},
  edition = {1},
  volume = {12},
  publisher = {EMS Press},
  doi = {10.4171/067},
  urldate = {2025-11-24},
  isbn = {978-3-03719-067-8 978-3-03719-567-3},
  langid = {english}
}

@book{marinucciRandomFieldsSphere2011a,
  title = {Random {{Fields}} on the {{Sphere}}: {{Representation}}, {{Limit Theorems}} and {{Cosmological Applications}}},
  shorttitle = {Random {{Fields}} on the {{Sphere}}},
  author = {Marinucci, D. and Peccati, G.},
  year = 2011,
  series = {London {{Mathematical Society Lecture Note Series}}},
  publisher = {Cambridge University Press},
  address = {Cambridge},
  doi = {10.1017/CBO9780511751677},
  urldate = {2024-03-15},
  isbn = {978-0-521-17561-6}
}

@article{sanz-sernaConerservativeNonconservativeSchemes1986,
  title = {Conerservative and {{Nonconservative Schemes}} for the {{Solution}} of the {{Nonlinear Schr\"odinger Equation}}},
  author = {{Sanz-Serna}, J. M. and Verwer, J. G.},
  year = 1986,
  month = jan,
  journal = {IMA Journal of Numerical Analysis},
  volume = {6},
  number = {1},
  pages = {25--42},
  issn = {0272-4979},
  doi = {10.1093/imanum/6.1.25},
  urldate = {2025-11-26}
}

@book {MR1214374,
    AUTHOR = {Kloeden, P. E. and Platen, E.},
     TITLE = {Numerical solution of stochastic differential equations},
    SERIES = {Applications of Mathematics (New York)},
    VOLUME = {23},
 PUBLISHER = {Springer-Verlag, Berlin},
      YEAR = {1992},
     PAGES = {xxxvi+632},
      ISBN = {3-540-54062-8},
   MRCLASS = {60H10 (34A50 34F05 65L99 65P05)},
  MRNUMBER = {1214374},
MRREVIEWER = {G. N. Mil\cprime shte\u{\i}n},
       DOI = {10.1007/978-3-662-12616-5},
       URL = {https://doi.org/10.1007/978-3-662-12616-5},
}

@book{szeg1939orthogonal,
	title={Orthogonal polynomials},
	author={Szeg, G.},
	volume={23},
	year={1939},
	publisher={American Mathematical Soc.}
}

@article{Barrera_1985,
	doi = {10.1088/0143-0807/6/4/014},
	url = {https://doi.org/10.1088/0143-0807/6/4/014},
	year = {1985},
	month = {oct},
	publisher = {},
	volume = {6},
	number = {4},
	pages = {287},
	author = {R G Barrera and G A Estevez and J Giraldo},
	title = {Vector spherical harmonics and their application to magnetostatics},
	journal = {European Journal of Physics}
}

@book {MR4369963,
    AUTHOR = {Milstein, G. N. and Tretyakov, M. V.},
     TITLE = {Stochastic numerics for mathematical physics},
    SERIES = {Scientific Computation},
      NOTE = {Second edition [of 2069903]},
 PUBLISHER = {Springer, Cham},
      YEAR = {[2021] \copyright 2021},
     PAGES = {xxv+736},
      ISBN = {978-3-030-82039-8; 978-3-030-82040-4},
   MRCLASS = {60-02 (35R60 37M15 60H10 60H15 60H35 65C30)},
  MRNUMBER = {4369963},
       DOI = {10.1007/978-3-030-82040-4},
       URL = {https://doi.org/10.1007/978-3-030-82040-4},
}

@article {MR3404631,
    AUTHOR = {Lang, A. and Schwab, C.},
     TITLE = {Isotropic {G}aussian random fields on the sphere: regularity,
              fast simulation and stochastic partial differential equations},
   JOURNAL = {Ann. Appl. Probab.},
  FJOURNAL = {The Annals of Applied Probability},
    VOLUME = {25},
      YEAR = {2015},
    NUMBER = {6},
     PAGES = {3047--3094},
      ISSN = {1050-5164,2168-8737},
   MRCLASS = {60G60 (60G15 60G17 60H15 60H35 65N30 65N75)},
  MRNUMBER = {3404631},
MRREVIEWER = {Peter\ E.\ Kloeden},
       DOI = {10.1214/14-AAP1067},
       URL = {https://doi.org/10.1214/14-AAP1067},
}

@book {MR2597943,
    AUTHOR = {Evans, L. C.},
     TITLE = {Partial differential equations},
    SERIES = {Graduate Studies in Mathematics},
    VOLUME = {19},
   EDITION = {Second},
 PUBLISHER = {American Mathematical Society, Providence, RI},
      YEAR = {2010},
     PAGES = {xxii+749},
      ISBN = {978-0-8218-4974-3},
   MRCLASS = {35-01},
  MRNUMBER = {2597943},
MRREVIEWER = {Diego M. Maldonado},
       DOI = {10.1090/gsm/019},
       URL = {https://doi.org/10.1090/gsm/019},
}

@article {MR1246744,
    AUTHOR = {Gy\"{o}ngy, I.},
     TITLE = {Stochastic partial differential equations on manifolds. {I}},
   JOURNAL = {Potential Anal.},
  FJOURNAL = {Potential Analysis. An International Journal Devoted to the
              Interactions between Potential Theory, Probability Theory,
              Geometry and Functional Analysis},
    VOLUME = {2},
      YEAR = {1993},
    NUMBER = {2},
     PAGES = {101--113},
      ISSN = {0926-2601},
   MRCLASS = {60H15 (58G32)},
  MRNUMBER = {1246744},
MRREVIEWER = {Ya. \={I}. B\={\i}lopol\cprime s\cprime ka},
       DOI = {10.1007/BF01049295},
       URL = {https://doi.org/10.1007/BF01049295},
}

@article {MR4059183,
    AUTHOR = {Su, W.},
     TITLE = {On the peaks of a stochastic heat equation on a sphere with a
              large radius},
   JOURNAL = {Electron. J. Probab.},
  FJOURNAL = {Electronic Journal of Probability},
    VOLUME = {25},
      YEAR = {2020},
     PAGES = {Paper No. 5, 38},
   MRCLASS = {60H15 (35R60 60G15)},
  MRNUMBER = {4059183},
MRREVIEWER = {G. V. Riabov},
       DOI = {10.1214/20-ejp415},
       URL = {https://doi.org/10.1214/20-ejp415},
}

@Article{Anh2018,
author={Anh, V. V.
and Broadbridge, P.
and Olenko, A.
and Wang, Y. G.},
title={On approximation for fractional stochastic partial differential equations on the sphere},
journal={Stochastic Environmental Research and Risk Assessment},
year={2018},
month={Sep},
day={01},
volume={32},
number={9},
pages={2585-2603},
issn={1436-3259},
doi={10.1007/s00477-018-1517-1},
url={https://doi.org/10.1007/s00477-018-1517-1}
}

@inProceedings{Gia2019,
author = "Le Gia, Q. T.
 and Peach, J.",
title = "A spectral method to the stochastic {S}tokes equations on the sphere",
series = "ANZIAM J.",
volume = "60",
pages = "C52--C64",
year = 2019,
booktitle = "Proceedings of the 18th Biennial Computational Techniques and Applications Conference , CTAC-2018",
editor = "Lamichhane, B. and Tran, T. and Bunder, J.",
month = jun,
doi = "10.21914/anziamj.v60i0.13987",
subjclass = "",
}

@misc{janssonNonstationaryGaussianRandom2024,
  title = {Non-Stationary {{Gaussian}} Random Fields on Hypersurfaces: Sampling and Strong Error Analysis},
  shorttitle = {Non-Stationary {{Gaussian}} Random Fields on Hypersurfaces},
  author = {Jansson, E. and Lang, A. and Pereira, M.},
  year = 2024,
  number = {arXiv:2406.08185},
  eprint = {2406.08185},
  primaryclass = {math},
  publisher = {arXiv},
  doi = {10.48550/arXiv.2406.08185},
  archiveprefix = {arXiv}
}

@article {MR3907363,
    AUTHOR = {Kazashi, Y. and Le Gia, Q. T.},
     TITLE = {A non-uniform discretization of stochastic heat equations with
              multiplicative noise on the unit sphere},
   JOURNAL = {J. Complexity},
  FJOURNAL = {Journal of Complexity},
    VOLUME = {50},
      YEAR = {2019},
     PAGES = {43--65},
      ISSN = {0885-064X},
   MRCLASS = {65M70 (35R60 60H15 65M75 86A10)},
  MRNUMBER = {3907363},
MRREVIEWER = {Sanda Micula},
       DOI = {10.1016/j.jco.2018.09.001},
       URL = {https://doi.org/10.1016/j.jco.2018.09.001},
}

@book{leeIntroductionRiemannianManifolds2018,
  title = {Introduction to {{Riemannian Manifolds}}},
  author = {Lee, J. M.},
  year = 2018,
  series = {Graduate {{Texts}} in {{Mathematics}}},
  volume = {176},
  publisher = {Springer International Publishing},
  address = {Cham},
  doi = {10.1007/978-3-319-91755-9},
  urldate = {2025-05-15},
  copyright = {http://www.springer.com/tdm},
  isbn = {978-3-319-91754-2 978-3-319-91755-9}
}

@article {MR4714764,
    AUTHOR = {Alodat, T. and Le Gia, Q. T. and Sloan, I. H.},
     TITLE = {On approximation for time-fractional stochastic diffusion
              equations on the unit sphere},
   JOURNAL = {J. Comput. Appl. Math.},
  FJOURNAL = {Journal of Computational and Applied Mathematics},
    VOLUME = {446},
      YEAR = {2024},
     PAGES = {Paper No. 115863, 33},
      ISSN = {0377-0427},
   MRCLASS = {65M70 (65M75)},
  MRNUMBER = {4714764},
       DOI = {10.1016/j.cam.2024.115863},
       URL = {https://doi.org/10.1016/j.cam.2024.115863},
}

@article {MR4701212,
    AUTHOR = {Lang, A. and Motschan-Armen, I.},
     TITLE = {Euler-{M}aruyama approximations of the stochastic heat
              equation on the sphere},
   JOURNAL = {J. Comput. Dyn.},
  FJOURNAL = {Journal of Computational Dynamics},
    VOLUME = {11},
      YEAR = {2024},
    NUMBER = {1},
     PAGES = {23--42},
      ISSN = {2158-2491},
   MRCLASS = {60H35 (33C55 35K05 35R60 60H15 65C30)},
  MRNUMBER = {4701212},
       DOI = {10.3934/jcd.2023012},
       URL = {https://doi.org/10.3934/jcd.2023012},
}

@article {MR4935789,
    AUTHOR = {Lang, A. and M\"{u}ller, B.},
     TITLE = {Isotropic {$Q$}-fractional {B}rownian motion on the sphere:
              regularity and fast simulation},
   JOURNAL = {Philos. Trans. Roy. Soc. A},
  FJOURNAL = {Philosophical Transactions of the Royal Society A.
              Mathematical, Physical and Engineering Sciences},
    VOLUME = {383},
      YEAR = {2025},
    NUMBER = {2298},
     PAGES = {Paper No. 20240238, 13},
      ISSN = {1364-503X},
   MRCLASS = {60G22 (35B65 65C35)},
  MRNUMBER = {4935789},
}

@misc{papi2025,
      title={Approximation of the {L}\'evy-driven stochastic heat equation on the sphere},
      author={Lang, A. and Papini, A. and Schwarz, V.},
      year={2025},
      eprint={2507.05005},
      archivePrefix={arXiv},
      primaryClass={math.PR},
      url={https://arxiv.org/abs/2507.05005},
}

@article {MR3923632,
    AUTHOR = {Malzoumati-Khiaban, M. and Foroush Bastani, A. and Yaghouti,
              M. R.},
     TITLE = {Long-term adaptive symplectic numerical integration of linear
              stochastic oscillators driven by additive white noise},
   JOURNAL = {Numer. Algorithms},
  FJOURNAL = {Numerical Algorithms},
    VOLUME = {80},
      YEAR = {2019},
    NUMBER = {3},
     PAGES = {1059--1095},
      ISSN = {1017-1398},
   MRCLASS = {65C30 (60H35)},
  MRNUMBER = {3923632},
MRREVIEWER = {Haziem Mohammad Hazaimeh},
       DOI = {10.1007/s11075-018-0517-z},
       URL = {https://doi.org/10.1007/s11075-018-0517-z},
}

@article {MR3906838,
    AUTHOR = {Senosiain, M. J. and Tocino, A.},
     TITLE = {On the numerical integration of the undamped harmonic
              oscillator driven by independent additive gaussian white
              noises},
   JOURNAL = {Appl. Numer. Math.},
  FJOURNAL = {Applied Numerical Mathematics. An IMACS Journal},
    VOLUME = {137},
      YEAR = {2019},
     PAGES = {49--61},
      ISSN = {0168-9274},
   MRCLASS = {65C30},
  MRNUMBER = {3906838},
MRREVIEWER = {Peter E. Kloeden},
       DOI = {10.1016/j.apnum.2018.12.001},
       URL = {https://doi.org/10.1016/j.apnum.2018.12.001},
}

@article {MR3608313,
    AUTHOR = {de la Cruz, H. and Jimenez, J. C. and Zubelli, J. P.},
     TITLE = {Locally linearized methods for the simulation of stochastic
              oscillators driven by random forces},
   JOURNAL = {BIT},
  FJOURNAL = {BIT. Numerical Mathematics},
    VOLUME = {57},
      YEAR = {2017},
    NUMBER = {1},
     PAGES = {123--151},
      ISSN = {0006-3835},
   MRCLASS = {65C30 (34F05 60H10 60H35)},
  MRNUMBER = {3608313},
MRREVIEWER = {Haziem Mohammad Hazaimeh},
       DOI = {10.1007/s10543-016-0620-2},
       URL = {https://doi.org/10.1007/s10543-016-0620-2},
}

@article {MR3348201,
    AUTHOR = {Senosiain, M. J. and Tocino, A.},
     TITLE = {A review on numerical schemes for solving a linear stochastic
              oscillator},
   JOURNAL = {BIT},
  FJOURNAL = {BIT. Numerical Mathematics},
    VOLUME = {55},
      YEAR = {2015},
    NUMBER = {2},
     PAGES = {515--529},
      ISSN = {0006-3835},
   MRCLASS = {65C30 (60H10 60H35)},
  MRNUMBER = {3348201},
MRREVIEWER = {Mikhail V. Tretyakov},
       DOI = {10.1007/s10543-014-0507-z},
       URL = {https://doi.org/10.1007/s10543-014-0507-z},
}

@article {MR3172331,
    AUTHOR = {Burrage, P. M. and Burrage, K.},
     TITLE = {Structure-preserving {R}unge-{K}utta methods for stochastic
              {H}amiltonian equations with additive noise},
   JOURNAL = {Numer. Algorithms},
  FJOURNAL = {Numerical Algorithms},
    VOLUME = {65},
      YEAR = {2014},
    NUMBER = {3},
     PAGES = {519--532},
      ISSN = {1017-1398},
   MRCLASS = {65C30 (60H10 65P10)},
  MRNUMBER = {3172331},
MRREVIEWER = {Mikhail V. Tretyakov},
       DOI = {10.1007/s11075-013-9796-6},
       URL = {https://doi.org/10.1007/s11075-013-9796-6},
}

@article {MR2926251,
    AUTHOR = {Burrage, K. and Burrage, P. M.},
     TITLE = {Low rank {R}unge-{K}utta methods, symplecticity and stochastic
              {H}amiltonian problems with additive noise},
   JOURNAL = {J. Comput. Appl. Math.},
  FJOURNAL = {Journal of Computational and Applied Mathematics},
    VOLUME = {236},
      YEAR = {2012},
    NUMBER = {16},
     PAGES = {3920--3930},
      ISSN = {0377-0427},
   MRCLASS = {65L06 (65C30 65P10)},
  MRNUMBER = {2926251},
MRREVIEWER = {Alexander Ostermann},
       DOI = {10.1016/j.cam.2012.03.007},
       URL = {https://doi.org/10.1016/j.cam.2012.03.007},
}

@article {MR2909913,
    AUTHOR = {Cohen, D. and Sigg, M.},
     TITLE = {Convergence analysis of trigonometric methods for stiff
              second-order stochastic differential equations},
   JOURNAL = {Numer. Math.},
  FJOURNAL = {Numerische Mathematik},
    VOLUME = {121},
      YEAR = {2012},
    NUMBER = {1},
     PAGES = {1--29},
      ISSN = {0029-599X},
   MRCLASS = {65C30 (34F05 60H10 65L04)},
  MRNUMBER = {2909913},
MRREVIEWER = {Roger Pettersson},
       DOI = {10.1007/s00211-011-0426-8},
       URL = {https://doi.org/10.1007/s00211-011-0426-8},
}

@article {MR2312503,
    AUTHOR = {Tocino, A.},
     TITLE = {On preserving long-time features of a linear stochastic
              oscillator},
   JOURNAL = {BIT},
  FJOURNAL = {BIT. Numerical Mathematics},
    VOLUME = {47},
      YEAR = {2007},
    NUMBER = {1},
     PAGES = {189--196},
      ISSN = {0006-3835},
   MRCLASS = {60H10 (34F05 65C30)},
  MRNUMBER = {2312503},
MRREVIEWER = {G. N. Mil\cprime shte\u{\i}n},
       DOI = {10.1007/s10543-007-0115-2},
       URL = {https://doi.org/10.1007/s10543-007-0115-2},
}

@article {MR4375496,
    AUTHOR = {Cohen, D. and Vilmart, G.},
     TITLE = {Drift-preserving numerical integrators for stochastic
              {P}oisson systems},
   JOURNAL = {Int. J. Comput. Math.},
  FJOURNAL = {International Journal of Computer Mathematics},
    VOLUME = {99},
      YEAR = {2022},
    NUMBER = {1},
     PAGES = {4--20},
      ISSN = {0020-7160},
   MRCLASS = {65C30 (60H10 60H35 65P10)},
  MRNUMBER = {4375496},
       DOI = {10.1080/00207160.2021.1922679},
       URL = {https://doi.org/10.1080/00207160.2021.1922679},
}

@article {MR4285774,
    AUTHOR = {Banjai, L. and Lord, G. and Molla, J.},
     TITLE = {Strong convergence of a {V}erlet integrator for the semilinear
              stochastic wave equation},
   JOURNAL = {SIAM J. Numer. Anal.},
  FJOURNAL = {SIAM Journal on Numerical Analysis},
    VOLUME = {59},
      YEAR = {2021},
    NUMBER = {4},
     PAGES = {1976--2003},
      ISSN = {0036-1429},
   MRCLASS = {65M60 (60H15 60H35 65C20 65M12 65M75)},
  MRNUMBER = {4285774},
MRREVIEWER = {Peter E. Kloeden},
       DOI = {10.1137/20M1364746},
       URL = {https://doi.org/10.1137/20M1364746},
}

@article {MR3012932,
    AUTHOR = {Schurz, H.},
     TITLE = {Stochastic wave equations with cubic nonlinearity and
              {Q}-regular additive noise in {$\Bbb R^2$}},
   JOURNAL = {Discrete Contin. Dyn. Syst.},
  FJOURNAL = {Discrete and Continuous Dynamical Systems. Series A},
      YEAR = {2011},
    NUMBER = {Dynamical systems, differential equations and applications.
              8th AIMS Conference. Suppl. Vol. II},
     PAGES = {1299--1308},
      ISSN = {1078-0947},
      ISBN = {978-1-60133-008-6; 1-60133-008-1},
   MRCLASS = {60H15 (37L55)},
  MRNUMBER = {3012932},
MRREVIEWER = {Mar\'{\i}a J. Garrido-Atienza},
}

@article {MR2384109,
    AUTHOR = {Schurz, H.},
     TITLE = {A numerical method for nonlinear stochastic wave equations in
              {$\Bbb R^1$}},
   JOURNAL = {Dyn. Contin. Discrete Impuls. Syst. Ser. A Math. Anal.},
  FJOURNAL = {Dynamics of Continuous, Discrete \& Impulsive Systems. Series
              A. Mathematical Analysis},
    VOLUME = {14},
      YEAR = {2007},
    NUMBER = {Advances in Dynamical Systems, suppl. S2},
     PAGES = {74--78},
      ISSN = {1201-3390},
   MRCLASS = {65C30 (34F05 35R60 60H15 60H35)},
  MRNUMBER = {2384109},
MRREVIEWER = {Dror Givon},
}

@article {MR2379913,
    AUTHOR = {Schurz, H.},
     TITLE = {Analysis and discretization of semi-linear stochastic wave
              equations with cubic nonlinearity and additive space-time
              noise},
   JOURNAL = {Discrete Contin. Dyn. Syst. Ser. S},
  FJOURNAL = {Discrete and Continuous Dynamical Systems. Series S},
    VOLUME = {1},
      YEAR = {2008},
    NUMBER = {2},
     PAGES = {353--363},
      ISSN = {1937-1632},
   MRCLASS = {60H15 (35L70 35R60 65C30)},
  MRNUMBER = {2379913},
MRREVIEWER = {Hisao Watanabe},
       DOI = {10.3934/dcdss.2008.1.353},
       URL = {https://doi.org/10.3934/dcdss.2008.1.353},
}

@article {MR3007549,
    AUTHOR = {Jiang, S. and Wang, L. and Hong, J.},
     TITLE = {Stochastic multi-symplectic integrator for stochastic
              nonlinear {S}chr\"{o}dinger equation},
   JOURNAL = {Commun. Comput. Phys.},
  FJOURNAL = {Communications in Computational Physics},
    VOLUME = {14},
      YEAR = {2013},
    NUMBER = {2},
     PAGES = {393--411},
      ISSN = {1815-2406},
   MRCLASS = {60H15 (65M75 65P10)},
  MRNUMBER = {3007549},
       DOI = {10.4208/cicp.230212.240812a},
       URL = {https://doi.org/10.4208/cicp.230212.240812a},
}

@article {MR3432362,
    AUTHOR = {Chen, C. and Hong, J. and Zhang, L.},
     TITLE = {Preservation of physical properties of stochastic {M}axwell
              equations with additive noise via stochastic multi-symplectic
              methods},
   JOURNAL = {J. Comput. Phys.},
  FJOURNAL = {Journal of Computational Physics},
    VOLUME = {306},
      YEAR = {2016},
     PAGES = {500--519},
      ISSN = {0021-9991},
   MRCLASS = {65P10 (35Q61 35R60 37H10 65M75 78A25)},
  MRNUMBER = {3432362},
MRREVIEWER = {\"{O}m\"{u}r Umut},
       DOI = {10.1016/j.jcp.2015.11.052},
       URL = {https://doi.org/10.1016/j.jcp.2015.11.052},
}

@article {MR3484400,
    AUTHOR = {Anton, R. and Cohen, D. and Larsson, S. and Wang,
              X.},
     TITLE = {Full discretization of semilinear stochastic wave equations
              driven by multiplicative noise},
   JOURNAL = {SIAM J. Numer. Anal.},
  FJOURNAL = {SIAM Journal on Numerical Analysis},
    VOLUME = {54},
      YEAR = {2016},
    NUMBER = {2},
     PAGES = {1093--1119},
      ISSN = {0036-1429},
   MRCLASS = {65M75 (35R60 60H15 60H35 65M60)},
  MRNUMBER = {3484400},
MRREVIEWER = {Sever A. Hirstoaga},
       DOI = {10.1137/15M101049X},
       URL = {https://doi.org/10.1137/15M101049X},
}

@article {MR4077238,
    AUTHOR = {Chen, C. and Cohen, D. and D'Ambrosio, R. and
              Lang, A.},
     TITLE = {Drift-preserving numerical integrators for stochastic
              {H}amiltonian systems},
   JOURNAL = {Adv. Comput. Math.},
  FJOURNAL = {Advances in Computational Mathematics},
    VOLUME = {46},
      YEAR = {2020},
    NUMBER = {2},
     PAGES = {Paper No. 27, 22},
      ISSN = {1019-7168},
   MRCLASS = {65C30 (60H10 60H35 65C20)},
  MRNUMBER = {4077238},
MRREVIEWER = {Marco P. Cabral},
       DOI = {10.1007/s10444-020-09771-5},
       URL = {https://doi.org/10.1007/s10444-020-09771-5},
}

@article {MR4410964,
    AUTHOR = {Sun, J. and Shu, C. and Xing, Y.},
     TITLE = {Multi-symplectic discontinuous {G}alerkin methods for the
              stochastic {M}axwell equations with additive noise},
   JOURNAL = {J. Comput. Phys.},
  FJOURNAL = {Journal of Computational Physics},
    VOLUME = {461},
      YEAR = {2022},
     PAGES = {Paper No. 111199, 30},
      ISSN = {0021-9991},
   MRCLASS = {65M60 (60H35 65M15 65M75 78A25)},
  MRNUMBER = {4410964},
       DOI = {10.1016/j.jcp.2022.111199},
       URL = {https://doi.org/10.1016/j.jcp.2022.111199},
}

@ARTICLE{1138693,
	author={K. Yee},
	journal={IEEE Transactions on Antennas and Propagation}, 
	title={Numerical solution of initial boundary value problems involving {M}axwell's equations in isotropic media},
	year={1966},
	volume={14},
	number={3},
	pages={302-307},
	doi={10.1109/TAP.1966.1138693}}

@article {MR4739347,
    AUTHOR = {Zhou, Y. and Liang, D.},
     TITLE = {Modeling and {FDTD} discretization of stochastic {M}axwell's
              equations with {D}rude dispersion},
   JOURNAL = {J. Comput. Phys.},
  FJOURNAL = {Journal of Computational Physics},
    VOLUME = {509},
      YEAR = {2024},
     PAGES = {Paper No. 113033, 29},
      ISSN = {0021-9991},
   MRCLASS = {65M06 (78A99)},
  MRNUMBER = {4739347},
       DOI = {10.1016/j.jcp.2024.113033},
       URL = {https://doi.org/10.1016/j.jcp.2024.113033},
}

@article {MR4447703,
    AUTHOR = {Chun, S. and Marcon, J. and Peir\'{o}, J. and Sherwin,
              S. J.},
     TITLE = {Reducing errors caused by geometrical inaccuracy to solve
              partial differential equations with moving frames on
              curvilinear domain},
   JOURNAL = {Comput. Methods Appl. Mech. Engrg.},
  FJOURNAL = {Computer Methods in Applied Mechanics and Engineering},
    VOLUME = {398},
      YEAR = {2022},
     PAGES = {Paper No. 115261, 24},
      ISSN = {0045-7825},
   MRCLASS = {65M50 (58E15)},
  MRNUMBER = {4447703},
       DOI = {10.1016/j.cma.2022.115261},
       URL = {https://doi.org/10.1016/j.cma.2022.115261},
}

@article {MR3635830,
    AUTHOR = {Chun, S.},
     TITLE = {Method of moving frames to solve time-dependent {M}axwell's
              equations on anisotropic curved surfaces: {A}pplications to
              invisible cloak and {ELF} propagation},
   JOURNAL = {J. Comput. Phys.},
  FJOURNAL = {Journal of Computational Physics},
    VOLUME = {340},
      YEAR = {2017},
     PAGES = {85--104},
      ISSN = {0021-9991},
   MRCLASS = {78A25 (65M60)},
  MRNUMBER = {3635830},
       DOI = {10.1016/j.jcp.2017.03.031},
       URL = {https://doi.org/10.1016/j.jcp.2017.03.031},
}

@article {MR4077824,
    AUTHOR = {Cohen, D. and Cui, J. and Hong, J. and Sun, L.},
     TITLE = {Exponential integrators for stochastic {M}axwell's equations
              driven by {I}t\^{o} noise},
   JOURNAL = {J. Comput. Phys.},
  FJOURNAL = {Journal of Computational Physics},
    VOLUME = {410},
      YEAR = {2020},
     PAGES = {109382, 21},
      ISSN = {0021-9991},
   MRCLASS = {65M75 (35Q60 35R60 60H15 78A25)},
  MRNUMBER = {4077824},
MRREVIEWER = {Dirk Bl\"{o}mker},
       DOI = {10.1016/j.jcp.2020.109382},
       URL = {https://doi.org/10.1016/j.jcp.2020.109382},
}

@article {MR3771721,
    AUTHOR = {Anton, R. and Cohen, D.},
     TITLE = {Exponential integrators for stochastic {S}chr\"{o}dinger equations
              driven by {I}t\^{o} noise},
   JOURNAL = {J. Comput. Math.},
  FJOURNAL = {Journal of Computational Mathematics},
    VOLUME = {36},
      YEAR = {2018},
    NUMBER = {2},
     PAGES = {276--309},
      ISSN = {0254-9409},
   MRCLASS = {65N75 (35Q55 35R60 60H15 65C50)},
  MRNUMBER = {3771721},
       DOI = {10.4208/jcm.1701-m2016-0525},
       URL = {https://doi.org/10.4208/jcm.1701-m2016-0525},
}

@article {MR2083326,
    AUTHOR = {Str{\o}mmen Melb{\o} , A. H. and Higham, D. J.},
     TITLE = {Numerical simulation of a linear stochastic oscillator with
              additive noise},
   JOURNAL = {Appl. Numer. Math.},
  FJOURNAL = {Applied Numerical Mathematics. An IMACS Journal},
    VOLUME = {51},
      YEAR = {2004},
    NUMBER = {1},
     PAGES = {89--99},
      ISSN = {0168-9274},
   MRCLASS = {65C30 (34F05 60H10)},
  MRNUMBER = {2083326},
       DOI = {10.1016/j.apnum.2004.02.003},
       URL = {https://doi.org/10.1016/j.apnum.2004.02.003},
}

@article {MR3033008,
    AUTHOR = {Cohen, D. and Larsson, S. and Sigg, M.},
     TITLE = {A trigonometric method for the linear stochastic wave
              equation},
   JOURNAL = {SIAM J. Numer. Anal.},
  FJOURNAL = {SIAM Journal on Numerical Analysis},
    VOLUME = {51},
      YEAR = {2013},
    NUMBER = {1},
     PAGES = {204--222},
      ISSN = {0036-1429},
   MRCLASS = {65M75 (60H15 60H35 65C30 65M60)},
  MRNUMBER = {3033008},
MRREVIEWER = {Roger Pettersson},
       DOI = {10.1137/12087030X},
       URL = {https://doi.org/10.1137/12087030X},
}

@article {MR4462619,
    AUTHOR = {Cohen, D. and Lang, A.},
     TITLE = {Numerical approximation and simulation of the stochastic wave
              equation on the sphere},
   JOURNAL = {Calcolo},
  FJOURNAL = {Calcolo. A Quarterly on Numerical Analysis and Theory of
              Computation},
    VOLUME = {59},
      YEAR = {2022},
    NUMBER = {3},
     PAGES = {Paper No. 32, 32},
      ISSN = {0008-0624},
   MRCLASS = {65C30 (33C55 41A25 60H15 60H35)},
  MRNUMBER = {4462619},
MRREVIEWER = {Xinjie Dai},
       DOI = {10.1007/s10092-022-00472-7},
       URL = {https://doi.org/10.1007/s10092-022-00472-7},
}

@book{peszat2007,
	title={Stochastic Partial Differential Equations with {L}\'evy noise: An Evolution Equation Approach},
	author={Peszat, S. and Zabczyk, J.},
	year={2007},
	publisher={Cambridge University Press}
}

@article{AnMan,
	title = {Analysis of the {L}aplacian on the complete {R}iemannian manifold},
	author ={Strichartz, R. S.},
	journal = {Journal of Functional Analysis},
	year = {1983},
	publisher={Academic Press Inc.}
}

@book{byerly1893elementary,
	title={An Elementary Treatise on Fourier's Series and Spherical, Cylindrical, and Ellipsoidal Harmonics: With Applications to Problems in Mathematical Physics},
	author={Byerly, W. E.},
	year={1893},
	publisher={Ginn}
}

@book{Lord_Powell_Shardlow_2014,
AUTHOR = {Lord, G. J. and Powell, C. E. and Shardlow, T.},
     TITLE = {An introduction to computational stochastic {PDE}s},
    SERIES = {Cambridge Texts in Applied Mathematics},
 PUBLISHER = {Cambridge University Press, New York},
      YEAR = {2014},
     PAGES = {xii+503},
      ISBN = {978-0-521-72852-2},
   MRCLASS = {60-01 (35R60 60H15 60H35 65-01 65M75 65N75)},
  MRNUMBER = {3308418},
MRREVIEWER = {Roger Pettersson},
       DOI = {10.1017/CBO9781139017329},
       URL = {https://doi.org/10.1017/CBO9781139017329},
}

@book{weidmann1980linear,
	title={Linear Operators in Hilbert Spaces},
	author={Weidmann, J.},
	year={1980},
	publisher={Springer-Verlag},
	address={New York},
	volume={68},
	series={Graduate Texts in Mathematics},
	isbn={978-1-4612-6029-5},
	doi={10.1007/978-1-4612-6027-1}
}

@book{monk2003finiteElement,
	author = {Monk, Peter},
	title = {Finite Element Methods for Maxwell's Equations},
	publisher = {Oxford University Press},
	year = {2003},
	month = {04},
	isbn = {9780198508885},
	doi = {10.1093/acprof:oso/9780198508885.001.0001},
	url = {https://doi.org/10.1093/acprof:oso/9780198508885.001.0001},
}

\end{document}